\documentclass[11pt]{amsart}
\usepackage{amsmath, amssymb, mathrsfs, bm, enumerate}
\usepackage{yhmath} 

\usepackage{tikz-cd} 
\usepackage[all]{xy}
\tikzset{Isom/.style={draw=none,every to/.append style={edge node={node [sloped, allow upside down, autoah=false]{$\simeq$}}}}}   
\tikzset{Equal/.style={draw=none,every to/.append style={edge node={node [sloped, allow upside down, auto=false]{$=$}}}}}   
\tikzset{Inside/.style={draw=none,every to/.append style={edge node={node [sloped, allow upside down, auto=false]{$\in$}}}}}   

\usepackage{stmaryrd} 

\usepackage{xargs}                      
\usepackage[colorinlistoftodos, prependcaption 
]{todonotes}                                                                                                                 

\newcommandx{\toDoChi}[2][1=]{\todo[linecolor=blue, backgroundcolor=white,bordercolor=blue,#1]{#2}}           

\newcommandx{\todoA}[2][1=]{\todo[linecolor=green, backgroundcolor=white,bordercolor=green,#1]{#2}}          

\newcommandx{\todoN}[2][1=]{\todo[linecolor=red, backgroundcolor=white,bordercolor=red,#1]{#2}}			

\newcommand{\excise}[1]{}

\newcommand{\LV}{\mathfrak{L}(V)}

\newcommand{\ZZ}{\mathbb{Z}}

\newcommand{\CC}{\mathbb{C}}

\newcommand{\tQp}{s_+}
\newcommand{\tQm}{s_-}
\newcommand{\tQpm}{s_\pm}

\newcommand{\Qp}{{Q_{+}}}
\newcommand{\Qm}{{Q_{-}}}
\newcommand{\Qpm}{{Q_{\pm}}}

\newcommand{\Vv}{\mathcal{V}}
\newcommand{\Ll}{\mathcal{L}}

\newcommand{\Cc}{\mathscr{C}}

\newcommand{\Ct}{\widetilde{C}}

\newcommand{\Spec}{\text{Spec}}

\newtheorem{theorem}[subsubsection]{Theorem}
\newtheorem{theorem*}{Theorem}
\newtheorem*{corVB}{VB corollary}
\newtheorem*{Fthm}{Factorization theorem}
\newtheorem{lemma}[subsubsection]{Lemma}
\newtheorem{proposition}[subsubsection]{Proposition}

\newtheorem{corollary}[subsubsection]{Corollary}

\theoremstyle{definition}
\newtheorem{rmk}[subsubsection]{Remark}
\newtheorem{defn}[subsubsection]{Definition}

\theoremstyle{definition}

\theoremstyle{remark}

\makeatletter
\newcommand{\setword}[2]{%
  \phantomsection
  #1\def\@currentlabel{\unexpanded{#1}}\label{#2}%
}
\makeatother

\makeatletter
\@namedef{subjclassname@2010}{\textup{MSC2010}}
\makeatother

\usepackage[pagebackref, colorlinks=true, linkcolor=black, citecolor=black, urlcolor=black]{hyperref}

\begin{document}

\title[On factorization of conformal blocks from vertex algebras]{On factorization and vector bundles of conformal blocks 
from vertex algebras}


\begin{abstract}
Representations of vertex operator algebras define sheaves of coinvariants and conformal blocks on moduli of stable pointed curves.
Assuming certain finiteness and semisimplicity conditions, we prove that such sheaves satisfy the factorization conjecture and consequently are vector bundles. Factorization is essential to a recursive formulation of invariants, like ranks and Chern classes, and to produce new constructions of rational conformal field theories and cohomological field theories. 
\end{abstract}


{\tiny{
\author[C.~Damiolini]{Chiara Damiolini}
\address{Chiara Damiolini \newline \indent Department of Mathematics,  \newline \indent  University of Texas at Austin, Austin TX 78712
 \newline \indent {\textit{Previous address:}} Department of Mathematics, Princeton University}
\email{chiara.damiolini@austin.utexas.edu}}}
{\tiny{
\author[A.~Gibney]{Angela Gibney}
\address{Angela Gibney \newline \indent Department of Mathematics,  \newline \indent University of Pennsylvania, Philadelphia, PA 19104-6395
\newline \indent {\textit{Previous address:}} Department of Mathematics, Rutgers University}
\email{agibney@sas.upenn.edu}}}
{\tiny{
\author[N.~Tarasca]{Nicola Tarasca}
\address{Nicola Tarasca 
\newline \indent Department of Mathematics \& Applied Mathematics,
\newline \indent Virginia Commonwealth University, Richmond, VA 23284
\newline \indent {\textit{Previous address:}} Department of Mathematics, Rutgers University}
\email{tarascan@vcu.edu}}}

\maketitle


By assigning a module over a vertex operator algebra to each marked point on a stable pointed curve, one can construct dual vector spaces of coinvariants and conformal blocks, giving rise to sheaves on moduli spaces of stable pointed curves.  The main result of this paper is that these sheaves satisfy the factorization property (Theorem \ref{thm:factorization}), as conjectured in \cite{ZhuGlobal, bzf}. Namely, if certain  finiteness and semisimplicity conditions hold, vector spaces of coinvariants and conformal blocks at a nodal curve decompose as products of analogous spaces at each component of its normalization.

We also show that sheaves of coinvariants satisfy the sewing property (Theorem \ref{thm:SewingAndFactorization}),  a refined version of factorization at infinitesimal smoothings of nodal curves.  From this and their projectively flat connection on families of smooth curves \cite{dgt}, we deduce they are vector bundles (\hyperlink{thm:VBS}{VB corollary}).    

Our findings generalize a number of results known in special cases and for low genus. A historical account with references  is given in  \S\S\ref{History} and \ref{FiniteHistory}.  

Factorization leads to recursive formulas for ranks and is used to show that the Chern characters define semisimple cohomological field theories, hence the Chern classes lie in the tautological ring  (\cite{dgt3}, building on results for affine Lie algebras from \cite{mop, moppz}). A study of these  tautological classes may lead to progress on open questions: As proposed by Pandharipande \cite{panrio}, a computation of such Chern classes independently of the projective flatness of the connection \cite{dgt} would yield relations in the tautological ring, and could be used to test Pixton's conjecture  \cite{AP, FJ}.  

For sheaves defined by integrable modules over affine Lie algebras, vector spaces of conformal blocks are canonically isomorphic to generalized theta functions (see \cite{LaszloSorger} and references therein).   When in addition the genus is zero, vector bundles of such coinvariants are globally generated \cite{fakhr}, hence their Chern classes have positivity properties.  For instance, first Chern classes are base-point-free and  thus give rise to morphisms, some with images having modular interpretations \cite{GiaGib, gjms}.  It is natural to expect that the more general vector bundles of coinvariants and conformal blocks studied here can be shown to have analogous properties under appropriate assumptions.  This has been supported by a preliminary investigation \cite{DG}.

To outline our results, we set some notation. We refer to a stable pointed coordinatized curve as a triple $(C,P_{\bullet}, t_\bullet)$ where $(C,P_{\bullet})$ is a stable $n$-pointed curve,  $P_\bullet=(P_1,\dots,P_n)$, and $t_\bullet=(t_1,\dots,t_n)$ with $t_i$ a  formal coordinate  at the point $P_i$. Let $M^\bullet=(M^1, \dots, M^n)$ be an $n$-tuple of finitely generated admissible modules over a vertex operator algebra $V$ (see \S\S \ref{sec:CVADef}-\ref{sec:VMod}).  When $C\setminus P_{\bullet}$ is affine, the vector space of coinvariants $\mathbb{V}(V;M^{\bullet})_{(C,P_{\bullet}, t_\bullet)}$ is defined as the largest quotient of the tensor product $\otimes_{i=1}^n M^i$ by the action of a Lie algebra determined by $V$ and $(C,P_{\bullet}, t_\bullet)$, see \S \ref{sec:Coinvariants}. In general, by adding more marked points, one can reduce to the case when \mbox{$C\setminus P_{\bullet}$} is affine, see \eqref{eq:coinvdef}. Before now, two Lie algebras have been used to define coinvariants: Zhu's Lie algebra $\mathfrak{g}_{\,C\setminus P_ \bullet}(V)$ and the (former) chiral Lie algebra $\mathscr{L}_{C\setminus P_ \bullet}(V)$.  Here, we introduce a new chiral Lie algebra $\Ll_{C\setminus P_ \bullet}(V)$.

Zhu's Lie algebra (\S \ref{sec:gCoinv}) is defined when the vertex algebra $V$ is quasi-primary generated and $\mathbb{Z}_{\geq 0}$-graded with lowest degree space of dimension one, for either \textit{fixed} smooth curves  \cite{ZhuGlobal, an1}, or  for \textit{rational} stable pointed curves \textit{with coordinates} \cite{nt}. To show that $\mathfrak{g}_{\, C\setminus P_ \bullet}(V)$ is a Lie algebra, Zhu uses that any \textit{fixed} smooth algebraic curve admits an atlas such that all transition functions are M\"obius transformations. Transition functions between charts on families of curves of arbitrary genus are more complicated, and for an arbitrary vertex operator algebra,  $\mathfrak{g}_{\, C\setminus P_ \bullet}(V)$ is not well-defined.

The chiral Lie algebra $\mathscr{L}_{C\setminus P_\bullet}(V)$ is available for curves of arbitrary genus and for more general vertex operator algebras~$V$, not necessarily quasi-primary generated. Defined for smooth pointed curves by Frenkel and Ben-Zvi \cite[\S 19.4.14]{bzf} and shown in \cite{bzf} to coincide with that studied by Beilinson and Drinfeld  \cite{bd}, the chiral Lie algebra was extended to nodal curves in \cite{dgt}, enabling the construction of coinvariants on such curves.

In \S \ref{Chiral}, we introduce and study a new chiral Lie algebra $\Ll_{C\setminus P_ \bullet}(V)$, whose coinvariants are projectively flat over $\mathcal{M}_{g,n}$, and additionally satisfy factorization under some natural hypotheses.  See \S\ref{CloserLook} for a description of $\Ll_{C\setminus P_ \bullet}(V)$ on nodal curves. 

In \cite{nt}, Nagatomo and Tsuchiya remark that the coinvariants on rational curves with coordinates using Zhu's Lie algebra are equivalent to those considered by Beilinson and Drinfeld.  In Proposition \ref{prop:IsoCo}, we verify their genus zero statement, and further show that coinvariants  from the chiral Lie algebra are isomorphic to those given by  $\mathfrak{g}_{\,C\setminus P_ \bullet}(V)$ whenever Zhu's Lie algebra is defined.  For instance, when $V$ is quasi-primary generated and one works over a family of curves admitting an atlas where all transition functions are M\"obius transformations, both perspectives are equivalent.

To state our main result, let $(C,P_\bullet, t_\bullet)$ be a stable pointed coordinatized curve, as above, with one node $Q$. Let $\Ct\to C$ be the  normalization,   \mbox{$Q_\bullet:=(\Qp,\Qm)$}  the pair of preimages of $Q$ in $\Ct$, and $s_\bullet:=(s_+,s_-)$ with  $s_\pm$ a formal coordinate at $\Qpm$.   Let $\mathscr{W}$ be  a set of representatives of isomorphism classes of simple $V$-modules, and for $W \in \mathscr{W}$, let $W'$ be its contragredient module  (\S \ref{ContraMod}).

\begin{Fthm}[Theorem \ref{thm:factorization}] \hypertarget{thm:Fact}{}
Let $V$ be a  rational, $C_2$-cofinite vertex operator algebra with one-dimensional weight zero space, and let $M^\bullet$ be an $n$-tuple of finitely-generated admissible $V$-modules. One has: \begin{equation}\label{fact}
\mathbb{V}\left(V;M^{\bullet}\right)_{(C,P_{\bullet}, t_\bullet)} \cong \bigoplus_{W \in \mathscr{W}} \mathbb{V}\left(V;M^{\bullet}\otimes W\otimes W' \right)_{\left(\Ct,P_{\bullet}\sqcup Q_\bullet, t_\bullet\sqcup s_\bullet\right)}.
\end{equation}
\end{Fthm}

If $\Ct= C_{+}\sqcup C_{-}$ is disconnected, with $Q_{\pm} \in C_{\pm}$,  then \eqref{fact} is isomorphic to
\[ \bigoplus_{W \in \mathscr{W}} \mathbb{V}\left(V;M^{\bullet}_{+}\otimes W \right)_{X_{+}}\otimes \mathbb{V}\left(V;M^{\bullet}_{-}\otimes W' \right)_{X_-}
\]
where $X_{\pm}:=(C_{\pm}, P_\bullet|_{C_{\pm}}\sqcup Q_{\pm}, t_\bullet|_{C_{\pm}}\sqcup s_{\pm})$,  and $M^{\bullet}_{\pm}$ are the modules at those points $P_\bullet$ that are in $C_{\pm}$.

The isomorphism giving the factorization theorem is constructed in \S \ref{section:Fact}.

Propositions \ref{prop:chiralnodal}, \ref{prop:FiniteTwisted}, and \ref{prop:coinvZhat} are at the heart of this work, and arise from the construction in \S\ref{sec:VAstable} of the sheaf $\Vv_C$ globalizing a vertex algebra $V$ over a nodal curve $C$, integral to the definition of the chiral Lie algebra. The sheaf $\Vv_C$ is a slight variant on the sheaf $\mathscr{V}_C$ that we defined in \cite{dgt}. The two sheaves coincide for smooth curves, and as we explain here, we recover the results of \cite{dgt} using $\Vv_C$ in place of  $\mathscr{V}_C$.  So while the projectively flat connection can be obtained in both constructions (see Proposition \ref{prop:LC} for $\Vv_C$), to have factorization, we have used $\Vv_C$ (see \S \ref{VSheaf}, and \S \ref{VSheafS}).  In Proposition \ref{prop:chiralnodal}, we explicitly describe the chiral Lie algebra on a nodal curve in terms of elements of the chiral Lie algebra on its normalization.  This involves stable $k$-differentials that satisfy an infinite number of identities. 

In Proposition \ref{prop:FiniteTwisted}, we show that vector spaces of coinvariants defined from the chiral Lie algebra and smooth curves of arbitrary genus are finite-dimensional.  Known to be true in special cases, this result is a natural generalization of work of Abe and Nagatomo \cite{an1} for coinvariants defined from Zhu's Lie algebra and smooth curves  with formal coordinates (see \S \ref{FiniteHistory} for the history of the problem).    As in \cite{an1}, we assume here that $V$ is $C_2$-cofinite, 
which implies $C_k$-cofiniteness for $k \ge 1$ \cite{BuhlSpanning,KarelLi}.

Proposition \ref{prop:coinvZhat} reinterprets coinvariants at a nodal curve as coinvariants on the normalization by the action of a Lie subalgebra of the chiral Lie algebra.  The proof of this result, for which we assume $V$ is $C_1$-cofinite, uses that lowest weight $V$-modules admit spanning sets of PBW-type, following \cite[Cor.~3.12]{KarelLi}.

In \S\ref{sec:sheafofcoinariantsfinal}, following Tsuchimoto \cite{ts} in the case of simple affine vertex algebras, we show the sheaf of coinvariants $\mathbb{V}\left(V;M^{\bullet}\right)$ is coordinate-free and descends to the moduli space $\overline{\mathcal{M}}_{g,n}$ of stable $n$-pointed curves of genus $g$. As an application of the \hyperlink{thm:Fact}{factorization theorem}, we show:

\begin{corVB} \hypertarget{thm:VBS}{}

Let $V$ be a  rational, $C_2$-cofinite vertex operator algebra with one-dimensional weight zero space, and $M^\bullet$ be an $n$-tuple of finitely-generated admissible $V$-modules.    Then $\mathbb{V}\left(V;M^\bullet \right)$ is a vector bundle of finite rank on $\overline{\mathcal{M}}_{g,n}$.
\end{corVB} 

The proof of the \hyperlink{thm:VBS}{VB corollary} is presented in \S \ref{sec:proofofVBC}. The result on the interior of $\overline{\mathcal{M}}_{g,n}$, the moduli space ${\mathcal{M}}_{g,n}$ of smooth pointed curves, follows from finite-dimensionality of coinvariants (Proposition \ref{prop:FiniteTwisted}) and the existence of a projectively flat connection \cite{dgt}. Two ingredients are needed to give \hyperlink{thm:VBS}{VB corollary}  on the whole space $\overline{\mathcal{M}}_{g,n}$:   Theorem \ref{thm:CoherenceOnS}, a more general result on finite-dimensionality of coinvariants, and the sewing property (Theorem \ref{thm:SewingAndFactorization}). Each of these involve evaluation of the sheaf of coinvariants on a family formed by infinitesimally smoothing a nodal curve (\S \ref{smoothings}). The proof of the sewing theorem (\S \ref{section:Sewing})   relies on the \hyperlink{thm:Fact}{factorization theorem}  and a sewing procedure originally found in \cite[\S 6.2]{tuy}. 

\subsection{Results on sheaves of coinvariants from \cite{dgt}}\label{DGT1Sheaf} 
As stated, the vertex algebra  sheaf $\Vv_C$ defined in \S\ref{sec:VAstable} is a slight variant on the sheaf  $\mathscr{V}_C$ defined in \cite{dgt}.  These give rise to sheaves of chiral Lie algebras $\Ll_{C\setminus P_{\bullet}}(V)$ and $\mathscr{L}_{C\setminus P_{\bullet}}(V)$, respectively, which in turn define two sheaves of coinvariants on the moduli space of curves, potentially different on the boundary. As these sheaves agree on the interior $\mathcal{M}_{g,n}$, i.e., the moduli space of smooth pointed curves, the results of \cite{dgt} on $\mathcal{M}_{g,n}$ hold for the sheaves of coinvariants defined here. In  particular, the identification of the Atiyah algebra acting on sheaves of coinvariants on $\mathcal{M}_{g,n}$  \cite[Theorem]{dgt}, and consequently, the computation of their Chern character on $\mathcal{M}_{g,n}$ \cite[Corollary]{dgt} hold here. Combining these results with the \hyperlink{thm:Fact}{factorization theorem}, we show in \cite{dgt3} that the sheaves of coinvariants defined here in the case of self-contragredient simple vertex algebras give rise to semisimple cohomological field theories, thus allowing one to determine their Chern character on $\overline{\mathcal{M}}_{g,n}$ in terms of the fusion rules.

A natural question is whether the two sheaves of coinvariants coincide on the  whole of $\overline{\mathcal{M}}_{g,n}$.  In particular, it is natural to ask whether  sheaves of coinvariants defined by $\mathscr{L}_{C\setminus P_{\bullet}}(V)$ satisfy factorization.

\subsection{History of factorization and sewing}\label{History}
Tsuchiya and Kanie used integrable modules at a fixed level over affine Lie algebras to form  spaces of coinvariants on smooth pointed rational curves with coordinates  \cite{TK}.   Generalized by Tsuchiya, Ueno, and Yamada  to moduli of stable pointed coordinatized curves of arbitrary genus, these sheaves were shown to satisfy a number of good properties including factorization and sewing \cite{tuy}. Tsuchimoto  \cite{ts} proved the bundles are independent of coordinates and descend to $\overline{\mathcal{M}}_{g,n}$. Beilinson, Feigin, and Mazur \cite{bfm} showed that factorization holds for coinvariants defined by modules over  Virasoro  algebras.  Our arguments  follow  \cite{nt} in the genus zero case after our study of the chiral Lie algebra allows one to replace Zhu's Lie algebra in the general~case. 

Sewing was proved for $g \in \{0,1\}$ by Huang \cite{HuangDiffEQ, HuangDualityModularInvariance, RigMod, VOAVer}.   His approach was to prove the operator product expansion and the modular invariance of intertwining operators, both conjectured by Moore and Seiberg  \cite{MooreSeiberg}. Huang  assumes that (i) $V=\oplus_{i\geq 0} V_i$ with $V_0\cong \mathbb{C}$; (ii) Every $\mathbb{N}$-gradable weak $V$-module is completely reducible; and  (iii) $V$ is $C_2$-cofinite. Our assumptions are (i); (ii') $V$ is rational;  and (iii).   Conditions (ii) and (iii) of Huang are equivalent to our conditions (ii') and (iii) (see \S \ref{RVAS} for more details). Huang shows that if one assumes in addition (iv) $V \cong V'$, then the modular tensor categories he constructs for $g \in \{0,1\}$ are {\em{rigid}} and {\em{nondegenerate}}. In case (i)-(iv) hold, $V$ is sometimes called {\em{strongly rational}}.

Codogni in \cite{codogni} proves factorization in case our hypotheses hold, with the additional assumptions (iv') $V \cong V'$ and (v) $V$ has no nontrivial modules.   Like Nagatomo and Tsuchiya, Codogni works with coinvariants defined on the moduli space of curves with formal coordinates.  We note that the additional  assumption (iv') gives that $V$ is quasi-primary generated and in particular,  coinvariants defined from Zhu's Lie algebra and the chiral Lie algebra are both well-defined and agree (see \S \ref{sec:CoinvariantsHistory}).  

\smallskip

Here we do not assume condition (iv) nor condition (v).  There are examples satisfying  (1--3) but not (v), see \S\ref{ExampleSection}. Furthermore, through private communications with Sven M\"oller, we have learned the existence of a vertex operator algebra that satisfies our conditions (1--3) but  not condition (iv).  Therefore, while establishing
 factorization for all $g\ge 0$, our work also covers  new examples of factorization for $g\in \{0,1\}$.

\subsection{History of coherence}\label{FiniteHistory}
 In \cite[page 3 and \S 5.5.4]{bzf} the authors single out rational vertex algebras as good candidates for defining finite-dimensional coinvariants (and hence leading to finite-rank vector bundles).  At that time, rationality and  $C_2$-cofiniteness were conjectured to be equivalent conditions \cite{DongLiMasonRegular, AbeBuhlDong}.  This has been disproved, as Abe has given a $C_2$-cofinite, non-rational vertex operator algebra~\cite{AbeNotRational}. 

Coherence of sheaves of coinvariants was shown previously in  special cases: for (1)   integrable modules at positive integral levels over affine Lie algebras  \cite{tuy}; (2)  modules over $C_2$-cofinite Virasoro vertex algebras  \cite{bfm};  (3) curves of genus zero and modules over $C_2$-cofinite vertex operator algebras $V=\oplus_{i\in \mathbb{N}}\,V_i$ such that  $V_0\cong \mathbb{C}$  \cite{nt};  (4) fixed smooth curves of positive  genus and modules over a quasi-primary generated, $C_2$-cofinite vertex operator algebras $V=\oplus_{i\in \mathbb{N}}\, V_i$ such that  \mbox{$V_0\cong \mathbb{C}$} \cite{an1}. After the first draft of this paper, \cite{vanekeren.heluani:2021} showed finite-dimensionality of spaces of coinvariants associated with trivial modules over elliptic curves under a weaker assumption than $C_2$-cofiniteness. In  \cite{DG} it is shown that if $V$ is generated in degree $1$ and  $g=0$, then coinvariants are finite-dimensional.


\section{Background}

\subsection{Vertex operator algebras}\label{sec:CVADef}
We work with a non-negatively graded  \textit{vertex operator algebra}, that is, a $4$-tuple  $\left(V, \bm{1}^{V}, \omega, Y(\cdot,z)\right)$, throughout simply denoted $V$ for short, such that:
\begin{enumerate}[(i)]
\item $V=\oplus_{i \in \mathbb{Z}_{\geq 0}} V_i$ is a vector space over $\mathbb{C}$ with $\dim V_i<\infty$;
\item  $\bm{1}^{V} \in V_0$ (the \textit{vacuum vector}), and $\omega \in V_2$ (the \textit{conformal vector});
\item $Y(\cdot,z)\colon V \rightarrow  \textrm{End}(V)\left\llbracket z,z^{-1} \right\rrbracket$ is a linear function  assigning to every element $A \in V$ the \textit{vertex operator} $Y(A,z) :=\sum_{i\in\mathbb{Z}} A_{(i)}z^{-i-1}$.
\end{enumerate}
The datum $\left(V, \bm{1}^{V}, \omega, Y(\cdot,z)\right)$ must satisfy the following axioms:

\begin{enumerate}[(a)]
\item \textit{(vertex operators are fields)} for all $A,B\in V$,  $A_{(i)}B=0$, for $i \gg 0$;
\item \textit{(vertex operators of the vacuum)}  $Y(\bm{1}^{V}, z)=\textrm{id}_V$:
\[
\bm{1}^{V}_{(-1)} = \textrm{id}_V \qquad \mbox{and} \qquad \bm{1}^{V}_{(i)} = 0, \quad\mbox{for $i\neq -1$,}
\]
and for all $A\in V$, \ $Y(A,z)\bm{1}^{V} \in A+ zV\llbracket z \rrbracket$:
\[
A_{(-1)}\bm{1}^{V} = A \qquad \mbox{and} \qquad A_{(i)}\bm{1}^{V} =0, \qquad \mbox{for $i\geq 0$};
\]
\item \textit{(weak commutativity)} for all $A,B\in V$, there exists an $N\in\mathbb{Z}_{\geq 0}$ such~that
\[
(z-w)^N \, [Y(A,z), Y(B,w)]=0 \quad \mbox{in }\textrm{End}(V)\left\llbracket z^{\pm 1}, w^{\pm 1} \right\rrbracket;
\]
\item \label{axiom.conf}
 \textit{(conformal structure)} for 
$Y(\omega,z)=\sum_{i\in\mathbb{Z}} \omega_{(i)}z^{-i-1}$,
\[
\left[\omega_{(p+1)}, \omega_{(q+1)} \right] = (p-q) \,\omega_{(p+q+1)} + \frac{c}{12} \,\delta_{p+q,0} \,(p^3-p) \,\textrm{id}_V. 
\]
Here $c\in\mathbb{C}$ is  the \textit{central charge} of $V$.  Moreover:
\[
\omega_{(1)}|_{V_i}=i\cdot\textrm{id}_V, \quad \text{for all } i, \qquad \mbox{and} \qquad Y\left(\omega_{(0)}A,z \right) = \partial_z Y(A,z).
\]
\end{enumerate}

\subsubsection{Action of Virasoro}\label{VirasoroSection}
As we next explain, the conformal structure encodes an action of the Virasoro (Lie) algebra $\textrm{Vir}$ on $V$.  The \textit{Witt (Lie) algebra} $\textrm{Der}\, \mathcal{K}$ represents the functor assigning to a $\mathbb{C}$-algebra $R$ the Lie algebra $\textrm{Der}\,\mathcal{K}(R):=R(\!( z)\!) \partial_z$  generated over $R$ by the derivations $L_p:=-z^{p+1}\partial_z$, for $p\in \mathbb{Z}$, with  relations $[L_p, L_q] = (p-q) L_{p+q}$.  The \textit{Virasoro (Lie) algebra} $\textrm{Vir}$ represents the functor assigning to $R$  the Lie algebra generated over $R$ by a formal vector $K$ and the elements $L_p$, for $p\in \mathbb{Z}$, with Lie bracket given~by
\[
[K, L_p]=0, \qquad [L_p, L_q] = (p-q) L_{p+q} + \frac{K}{12} (p^3-p)\delta_{p+q,0}.
\]

Setting $L_p=\omega_{(p+1)} \in \textrm{End}(V)$,  axiom \eqref{axiom.conf} gives an action of $\textrm{Vir}$  on $V$ with  \textit{central charge} $c\in \mathbb{C}$, that is, $K\in \textrm{Vir}$ acts as $c\cdot \textrm{id}_V$.

\subsubsection{Degree of $A_{(i)}$}
\label{sec:Degree} 
As a consequence of the axioms, one has  $A_{(i)}V_k\subseteq V_{k+d -i-1}$ for homogeneous $A\in V_d$ (see e.g.,~\cite{zhu}). Thus, we have
\begin{equation}
\label{eq:degAi}
\deg A_{(i)}:= \deg (A) -i-1, \qquad \mbox{for homogeneous $A$ in $V$.}
\end{equation}
Axiom \eqref{axiom.conf} implies that $L_0$ acts as a \textit{degree} operator on $V$, and  $L_{-1}$, called the \textit{translation}  operator, is determined by $L_{-1}A=A_{(-2)} \bm{1}^{V}$, for $A\in V$.

\subsection{Modules of vertex operator algebras}
\label{sec:VMod} 
Let $\left(V, \bm{1}^{V}, \omega,Y(\cdot,z)\right)$ be a vertex operator algebra. A \textit{weak $V$-module} is a pair $\left(M,Y^M(\cdot,z)\right)$, where:
\begin{enumerate}[(i)]
    \item $M$ is a vector space over $\mathbb{C}$;
    \item $Y^M(\cdot,z)\colon V \rightarrow  \textrm{End}(M)\left\llbracket z,z^{-1} \right\rrbracket$ is a linear function that assigns to $A \in V$ an  $\textrm{End}(M)$-valued \textit{vertex operator} $Y^M(A,z) :=\sum_{i\in\mathbb{Z}} A^M_{(i)}z^{-i-1}$.    
\end{enumerate}
The pair $\left(M,Y^M(\cdot,z)\right)$ must satisfy the following axioms:
\begin{enumerate}[(a)]
    \item for all $A \in V$ and $v \in M$, one has $A^M_{(i)}v=0$, for $i \gg 0$;
    \item $Y^M\left(\bm{1}^{V},z\right)=\textrm{id}_M$;
    \item \label{wcomm} for all $A, B \in V$, there exists $N\in\mathbb{Z}_{\geq 0}$ such that for all $v \in M$ one has
    \[
(z-w)^N \left[ Y^M(A,z), Y^M(B,w) \right]v=0;
\]
    \item \label{wass} for all $A \in V$ and $v \in M$,  there exists $N\in\mathbb{Z}_{\geq 0}$, such that for all $B \in V$ one has
     \[
(w+z)^N \left( Y^M(Y(A,w)B,z)- Y^M(A,w+z)Y^M(B,z)\right)v =0;
\]
    \item \label{VirAct} For   $Y^M(\omega,z) = \sum_{i \in \mathbb{Z}} \omega_{(i)}^M z^{-i-1}$, one has
    \[
    \left[\omega^M_{(p+1)}, \omega^M_{(q+1)}\right] = (p-q) \,\omega^M_{(p+q+1)} + \frac{c}{12} \,\delta_{p+q,0} \,(p^3-p) \,\textrm{id}_M, 
    \] 
	where $c \in \mathbb{C}$ is the central charge of $V$. We identify $\omega_{(p+1)} \in \textrm{End}(M)$ with an action of $L_{p} \in \textrm{Vir}$ on $M$. Moreover $Y^M\left(L_{-1}A,z \right) = \partial_z Y^M(a,z)$.
\end{enumerate}
 
In the literature, axiom (\ref{wcomm}) is referred to as \textit{weak commutativity}, and axiom (\ref{wass}) as  \textit{weak associativity}. Weak associativity and  weak commutativity are known to be equivalent to the \textit{Jacobi identity} (see e.g., \cite{dl, fhl,  lepli, lvertexsuperalg}).  Moreover, by \cite[Lemma 2.2]{DongLiMasonRegular}, axiom (\ref{VirAct}) is redundant.

Throughout,  by $V$-modules we mean admissible $V$-modules. These are   weak $V$-modules  such that 
\begin{enumerate}[(f)]
\item \label{admissibility} 
$M=\oplus_{i\in \mathbb{N}}\, M_i$ is $\mathbb{N}$-graded and
$A^M_{(i)} M_k \subseteq M_{k + \deg(A) -i-1}$ for every homogeneous $A \in V$.
\end{enumerate}
Note that $V$ is a module over itself  (see \cite[Thm 3.5.4]{lepli} or \cite[\S 3.2.1]{bzf}). 
We give a description of $V$-modules in \S \ref{sec:VmodAVmod}.

\subsection{The Lie algebra ancillary to $V$}
\label{LV}
Given a formal variable $t$, we call
\[
\mathfrak{L}_t(V) := \big( V\otimes \mathbb{C}(\!(t)\!) \big) \big/ \textrm{Im}\, \partial
\]
the \textit{Lie algebra} $\mathfrak{L}(V)=\mathfrak{L}_t(V)$ \textit{ancillary to} $V$. Here 
\begin{equation}
\label{eq:partialancillary}
\partial:= L_{-1}\otimes \textrm{id}_{\mathbb{C}(\!(t)\!)} +  \textrm{id}_V \otimes \partial_t.
\end{equation}
The space $\mathfrak{L}(V)$ is spanned by series of type $\sum_{i\geq i_0} c_i \, A_{[i]}$, for $A\in V$, $c_i\in\mathbb{C}$, and $i_0\in \mathbb{Z}$, where $A_{[i]}$ denotes the projection in $\mathfrak{L}(V)$ of $A\otimes t^i\in V\otimes \mathbb{C}(\!(t)\!)$. The Lie bracket of $\mathfrak{L}(V)$ is induced by
\begin{equation}
\label{bracket}
\left[A_{[i]}, B_{[j]} \right] := \sum_{k\geq 0} {i \choose k} \left(A_{(k)}\cdot B \right)_{[i+j-k]}.
\end{equation}
The axiom on the vacuum vector $\bm{1}^V$ implies that $\bm{1}^V _{[-1]}$ is central. The Lie algebra $\mathfrak{L}(V)$ is isomorphic to the current Lie algebra in \cite{nt}. For $t$ we will use a formal coordinate at a point $P$ on a curve, and we will denote $\mathfrak{L}_P(V)=\mathfrak{L}_t(V)$. A coordinate-free description of $\mathfrak{L}_t(V)$ is discussed in \S\ref{coordfreeLV}.

\subsection{The universal enveloping algebra} 
\label{sec:UV}
For a vertex algebra $V$, there is a complete topological associative algebra $\mathscr{U}\!(V)$, defined originally by Frenkel and Zhu \cite{FrenkelZhu}.  We review it here following the presentation in \cite{bzf}. Consider the universal enveloping algebra $U(\mathfrak{L}(V))$ of $\mathfrak{L}(V)$, and its completion
\[
\widetilde{U}(\mathfrak{L}(V)):=\varprojlim_N \ U(\mathfrak{L}(V))/I_N,
\]
where $I_N$ is the left ideal generated by $A_{[i]}$, for homogeneous $A \in V$ and $i \geq N+\deg(A)$.
The \textit{universal enveloping algebra} $\mathscr{U}\!(V)$ of $V$ is defined as the quotient of $\widetilde{U}(\mathfrak{L}(V))$ by the two-sided ideal generated by the Fourier coefficients of the series
 \[
 Y\left[A_{(-1)}B, z\right]\ -\ :\!Y[A,z] \ Y[B,z] \! :,  \quad \mbox{for all } A, \ B \in V,
 \]
where $Y[A,z]=\sum_{i\in \mathbb{Z}}A_{[i]}z^{-i-1}$ and  ``$: \quad :$'' is the normal ordering (see \cite{Monster}).

For an affine vertex algebra $V=V_{\ell}(\mathfrak{g})$, one has $\mathscr{U}\!(V)$ is isomorphic to a completion $\widetilde{U}_\ell(\widehat{\mathfrak{g}})$ of  $U_\ell(\widehat{\mathfrak{g}})$, for all  $\ell\in\mathbb{C}$. Here, $U_\ell(\widehat{\mathfrak{g}})$ is the quotient of $U(\widehat{\mathfrak{g}})$ by the two-sided ideal generated by $K-\ell$, where $K\in \widehat{\mathfrak{g}}/\mathfrak{g}\otimes \mathbb{C}(\!(t)\!)$, and 
\[
\widetilde{U}_\ell(\widehat{\mathfrak{g}}) := \varprojlim_N U_\ell(\widehat{\mathfrak{g}}) / U_\ell(\widehat{\mathfrak{g}}) \cdot \mathfrak{g}\otimes t^N\mathbb{C}\llbracket t \rrbracket.
\]
For more detail and other examples see \cite[\S 4.3.2, \S 5.1.8]{bzf}.

\subsection{Action on $V$-modules}\label{sec:VmodAVmod}
Both  $\mathfrak{L}(V)$ and $\mathscr{U}\!(V)$ act on any $V$-module $M$ via the Lie algebra homomorphism $\mathfrak{L}(V)\rightarrow \textrm{End}(M)$ obtained by mapping $A_{[i]}$ to the Fourier coefficient $A_{(i)}$ of the vertex operator $Y^M(A,z)=\sum_i A_{(i)}z^{-i-1}$.  Thus, the series $\sum_{i\geq i_0} c_i \, A_{[i]}$ acts on a $V$-module $M$ via
\begin{equation*}
\label{action}
\textrm{Res}_{z=0}\; Y^M(A,z)\sum_{i\geq i_0} c_i z^i dz.
\end{equation*}

An $\mathfrak{L}(V)$-module need not be a $V$-module. On the other hand, there is an equivalence between the categories of weak $V$-modules and smooth $\mathscr{U}\!(V)$-modules (see \cite[\S 5.1.6]{bzf}). A $\mathscr{U}\!(V)$-module $M$ is \textit{smooth} if for any $w\in M$ and $A\in V$, one has $A_{[i]}w=0$ for $i \gg 0$.   

A $V$-module is \textit{finitely generated} if it is finitely generated as a $\mathscr{U}(V)$-module.
The modules  in \cite{nt} are also finitely generated and admissible.

\subsection{Correspondence between $V$-modules and $A(V)$-modules} \label{sec:irrVAVmod}

A $V$-module $W$ is \textit{irreducible}, or \textit{simple},  if it is non zero and it has no sub-representation other than the trivial representation $0$ and $W$ itself. We review here Zhu's associative algebra $A(V)$, and the one-to-one correspondence between isomorphism classes of finite-dimensional simple $A(V)$-modules and isomorphism classes of simple $V$-modules \cite{zhu}. 

\textit{Zhu's algebra} is the quotient $A(V):=V/O(V)$, where $O(V)$ is the subspace of $V$ linearly spanned by elements of the form
\[
\textrm{Res}_{z=0}\,\frac{(1+z)^{\deg A}}{z^2}Y(A,z)B,
\]
where $A$ is homogeneous in $V$. The image of an element $A \in V$ in $A(V)$ is denoted by $o(A)$. The product in $A(V)$ is defined by
\[
o(A)*o(B)= \textrm{Res}_{z=0}\,\frac{(1+z)^{\deg A}}{z}Y(A,z)B,
\]
for homogeneous $A$ in $V$.  Nagatomo and Tsuchiya  \cite{nt} consider an  isomorphic copy of $A(V)$, which they refer to as the {\em{zero-mode algebra}}. 

Given a  $V$-module $W=\bigoplus_{i\geq 0}\, W_i$, one has that $W_0$ is an $A(V)$-module \cite[Thm 2.2.2]{zhu}. The action of $A(V)$ on $W_0$ is defined as follows: an element $o(A) \in A(V)$, image of a homogeneous element $A\in V$, acts on $W_0$ as the endomorphism $A_{(\deg A -1)}$, a Fourier coefficient of $Y^W(A,z)$.  

For the other direction, recall the triangular decomposition: 
\begin{equation}
\label{eq:triangdecLV}
\mathfrak{L}(V)= \mathfrak{L}(V)_{<0} \oplus \mathfrak{L}(V)_0 \oplus \mathfrak{L}(V)_{>0},
\end{equation}
determined by the degree  $\deg(A_{[i]}):=\deg(A)-i-1$, for homogeneous $A\in V$. From the  Lie bracket \eqref{bracket} of $\mathfrak{L}(V)$, one checks that each summand above is a Lie subalgebra of $\mathfrak{L}(V)$. This induces a subalgebra $\mathscr{U}\!(V)_{\leq 0}$ of $\mathscr{U}\!(V)$.

Given a finite-dimensional $A(V)$-module $E$, the \textit{generalized Verma $\mathscr{U}\!(V)$-module}  is 
\[
M(E):= \mathscr{U}\!(V)\otimes_{\mathscr{U}\!(V)_{\leq 0}} E.
\]
To make $E$ into an $\mathscr{U}\!(V)_{\leq 0}$-module, one lets  $\mathfrak{L}(V)_{<0}$ act trivially on $E$, and  $\mathfrak{L}(V)_{0}$ act by the homomorphism of Lie algebras $\mathfrak{L}(V)_{0} \twoheadrightarrow A(V)_{\rm Lie}$ induced by the identity endomorphism of $V$ \cite[Lemma 3.2.1]{lithesis}. For homogeneous $A\in V_k$, the image of the element $A_{[k-1]}\in \mathfrak{L}(V)_{0}$ in $A(V)$ is $o(A)$. By construction, $M(E)$ is automatically a $V$-module. 

Given an \textit{irreducible} $V$-module $W=\bigoplus_{i\geq 0} \,W_i$, the space $W_0$ is a finite-dimensional \textit{irreducible} $A(V)$-module; conversely, given a finite-dimensional \textit{irreducible} $A(V)$-module $E$,  there is a unique maximal proper $V$-submodule $N(E)$ of the $V$-module $M(E)$ with $N(E)\cap E=0$ such that $L(E)=M(E)/N(E)$ is an \textit{irreducible} $V$-module \cite{zhu}.

\subsection{Rational vertex algebras}
\label{RVAS}
A vertex algebra $V$ is \textit{rational} if every finitely generated $V$-module is a direct sum of irreducible $V$-modules. A rational vertex algebra has only finitely many isomorphism classes of irreducible modules, and an irreducible module $M=\bigoplus_{i\geq 0} M_i$ over a rational vertex algebra satisfies $\dim M_i<\infty$ \cite{DongLiMasonTwisted}.

An \textit{ordinary} $V$-module is a weak $V$-module  which carries a ${\mathbb{C}}$-grading $M=\bigoplus_{\lambda\in\mathbb{C}} M_\lambda$ such that: $L_0|_{M_\lambda}=\lambda \,\mathrm{id}_{M_\lambda}$; $\dim M_\lambda <\infty$; and for fixed $\lambda$, one has $M_{\lambda+n}=0$,  for integers $n\ll 0$.  Ordinary $V$-modules are admissible, and when $V$ is rational, every simple admissible $V$-module is ordinary \cite{DongLiMasonTwisted}, \cite[Rmk 2.4]{DongLiMasonRegular}. It follows that for rational $V$, finitely generated $V$-modules are direct sums of simple ordinary $V$-modules. In particular,  a finitely generated admissible module $M=\bigoplus_{i \geq 0}\, M_i$ over a rational vertex algebra satisfies $\dim M_i<\infty$, for all $i$, and $L_0$ acts semisimply on such $M$. 

When $V$ is rational, the associative algebra $A(V)$ is semisimple \cite{zhu}. 
For a rational vertex algebra $V$, given a simple module $E$ over  Zhu's algebra $A(V)$, the Verma module $M(E)$ remains simple.   In general, Verma modules are not necessarily simple, but they are indecomposable. Hence complete reducibility   implies that simple and indecomposable coincide.

\subsection{Dual modules}
\label{ContraMod}
Isomorphism classes of simple $V$-modules and simple $A(V)$-modules are in correspondence (\S \ref{sec:VmodAVmod}). Here we describe the $V$-module (and corresponding $A(V)$-module) structures on their duals.
\subsubsection{Duality for $V$-modules}
Let $V$ be a vertex algebra. Following~\cite[\S5.2]{fhl}, given a $V$-module $\left(M = \oplus_{i\geq 0} M_i, Y^M(-,z)\right)$,  its \textit{contragredient module} is $\left( M',  Y^{M'}(-,z)\right)$, where $M'$ is the graded dual of $M$, that is, $M':=\oplus_{i\geq 0} M_i^\vee$, with $M^\vee_i:=\textrm{Hom}_{\mathbb{C}}(M_i,\mathbb{C})$, and
 \[
 Y^{M'}(-,z) \colon V \to \mathrm{End} \left(M' \right) \left\llbracket z, z^{-1}\right\rrbracket
 \]
is the unique linear map determined by
\begin{equation}
 \label{eq:contr} 
 \left\langle Y^{M'}(A,z)\psi, m \right\rangle = \left\langle \psi, Y^M\left(e^{zL_1}(-z^{-2})^{L_0}A, z^{-1}\right)m \right\rangle
\end{equation}
for $A\in V$, $\psi\in M'$, and $m\in M$. Here and throughout, $\langle \cdot, \cdot \rangle$ is the natural pairing between a vector space and its graded dual.

\subsubsection{Duality for $A(V)$-modules}
For a $V$-module $M$, the $A(V)$-module structure on $M^\vee_0$ requires  the involution $\vartheta \colon \LV^f \to \LV^f$ induced from
\begin{equation} 
\label{eq:iota} 
\vartheta\left(A_{[j]}\right) := (-1)^{a-1}\sum_{i\geq 0} \frac{1}{i!} \left(L_1^i A\right)_{[2a -j- i -2]}
\end{equation}
for a homogeneous element $A \in V$ of degree $a$. Here $\LV^f \subset \LV$ denotes the quotient of $V \otimes_\CC \CC[t, t^{-1}]$ by $\mathrm{Im}\,\partial$. Observe that since the operator $L_1$ has negative degree, the above sum is finite, and that $\vartheta$ is a Lie algebra homomorphism \cite[Prop 4.1.1]{nt} (note the sign  difference between $\vartheta$ used here and the one in \cite{nt}). This involution  appeared in \cite{Borcherds}, and it naturally arises from the action of the vertex operators on the contragredient module $V'$. Since $\vartheta$ restricts to an involution of $\LV_0$ leaving $O(V)$ invariant, it induces an involution on Zhu's algebra $A(V)$. The following statement is a direct consequence of the definition of contragredient modules.

\begin{lemma} 
\label{lem:contr} 
\begin{enumerate}[i)]
\item The image of $\psi \in M_0^\vee$ under the action of $\sigma\in \mathfrak{L}(V)_0$ 
is the linear functional 
\[
\sigma \cdot \psi = - \psi \circ \vartheta(\sigma).
\]
\item The image of $\psi \in M_0^\vee$ under the action of $o(A)\in A(V)$ is 
\begin{equation*} 
\label{eq:contraction2} o(A) \cdot \psi = - \psi \circ \vartheta(o(A)).
\end{equation*} 
For homogeneous $A \in V$ of degree $a$ and for  $m \in M_0$, this is
\[ 
\left\langle o(A) \cdot \psi , m \right\rangle = (-1)^{a}\left\langle \psi, \sum_{i \geq 0} \dfrac{1}{i!}\left(L_1^i A\right)_{(a-i-1)}m\right\rangle.
\]
\end{enumerate}
\end{lemma}

See \cite[\S 10.4.8]{bzf} for a  geometric realization of the involution $\vartheta$. In \S \ref{Chiral}, we use $\vartheta$ to describe chiral Lie algebras on nodal curves.

\subsection{Stable $k$-differentials}
Let $(C,P_\bullet)$ be a stable $n$-pointed curve with at least one node, and  let $\omega_C$ be the dualizing sheaf on $C$. We review here \textit{(stable) $k$-differentials} on~$C$, that is, sections of $\omega_C^{\otimes k}$, for an integer $k$. When $k\geq 1$, by $\omega_C^{\otimes -k}$ we mean $(\omega_C^\vee)^{\otimes k}$.

Let $\Ct\to C$ be the partial normalization of~$C$ at a node $Q$,  let $\Qp, \Qm  \in \Ct$ be the two preimages of $Q$, and set $Q_{\bullet}=(\Qp, \Qm)$.  Note that the curve $\Ct$ may not be connected. Let $\tQp$ and $\tQm$ be formal coordinates at the points $\Qp$ and $\Qm$, respectively. We write $\Qpm$ to denote either point, and similarly use $s_{\pm}$ to denote either formal coordinate. For a section 
\[
\mu\in \text{H}^0\left(\Ct\setminus P_\bullet\sqcup Q_\bullet, \omega_{\Ct}^{\otimes k}\right)=:\widetilde{\Gamma},
\] 
let $\mu_{\Qpm}\in \mathbb{C}(\!( s_\pm)\!)(ds_\pm)^k$ be the Laurent series expansion of $\mu$ at $\Qpm$, that is, the image of $\mu$ under the restriction morphism
\[
\text{H}^0\left(\Ct\setminus P_\bullet\sqcup Q_\bullet, \omega_{\Ct}^{\otimes k}\right) \rightarrow \text{H}^0\left(D^\times_{\Qpm}, \omega_{\Ct}^{\otimes k}\right)\simeq_{\tQpm} \mathbb{C}(\!(\tQpm)\!)(ds_\pm)^k.
\]
Here $D^\times_{\Qpm}$ is the punctured formal disk about $\Qpm$, that is,  the spectrum of the field of fractions of the completed local ring $\widehat{\mathscr{O}}_{\Qpm}$, and $\simeq_{\tQpm}$ denotes the isomorphism given by fixing the formal coordinate $s_\pm$ at $\Qpm$. 

For a  $k$-differential $\mu$, define the \textit{order} $\textup{ord}_{\Qpm}\left(\mu\right)$ of $\mu$ at the point $\Qpm$ as the highest integer $m$ such that $\mu_{\Qpm}\in s_{\pm}^m\mathbb{C}\llbracket s_\pm\rrbracket (ds_\pm)^k$. For a  $k$-differential $\mu$ with $\textup{ord}_{\Qpm}\left(\mu\right) \geq -k$, the \textit{$k$-residue} $\textup{Res}^k_{\Qpm} \left(\mu\right)$ of $\mu$ at the point $\Qpm$ is defined as the coefficient of $s_\pm^{-k}(ds_\pm)^k$ in $\mu_{\Qpm}$. 

The order and the $k$-residue at $\Qpm$ are independent of the formal coordinate $s_\pm$ at $\Qpm$.  Here we only require the case $\textup{ord}_{\Qpm}\left(\mu\right) \geq -k$. For the definition of the $k$-residue without the assumption $\textup{ord}_{\Qpm}\left(\mu\right) \geq -k$, see e.g., \cite{bcggm}.

\begin{lemma}
\label{lem:omegaotimesm}
Assume that $C\setminus P_\bullet$ is affine. For all integers $k$, one has 
\[
\eta^*\,\mathrm{H}^0\left(C\setminus P_\bullet, \omega_C^{\otimes k}\right) \\
= \left\{ \mu \in \widetilde{\Gamma} \; \Bigg| \,  
\begin{array}{l}
\textup{ord}_{\Qpm}\left(\mu\right) \geq -k, \\[0.2cm]
\textup{Res}^k_{\Qp} \left(\mu\right) = (-1)^{k} \,\textup{Res}^k_{\Qm} \left(\mu\right)
\end{array}
\right\}.
\]
\end{lemma}

\begin{proof} It is enough to prove the statement for $k\in \{-1,0,1\}$: indeed, for negative integers $k$, sections of $\omega_{C}^{\otimes k}$ on the affine $C\setminus P_\bullet$ are tensor products of sections of $\omega_{C}^{-1}$, and the Laurent series expansions of sections of $\omega_{\Ct}^{\otimes k}$ at $\Qp$ and $\Qm$ are obtained as tensors of the Laurent series expansions of sections of $\omega_{\Ct}^{-1}$ at $\Qp$ and $\Qm$, respectively. An analogous argument may be made in case $k$ is a positive integer.

When $k=1$, the statement is about sections of $\omega_C$, and by definition sections of $\omega_C$ are sections of $\omega_{\Ct}$ with at most simple poles at $\Qp$ and $\Qm$ such that $\textrm{Res}_{\Qp}(\mu)+\textrm{Res}_{\Qm}(\mu)=0$. When $k=0$, the statement is about sections of $\mathscr{O}_C$, and indeed a regular function on $C$ is a regular function $\mu$ on $\Ct$ such that $\mu(\Qp)=\mu(\Qm)$. When $k=-1$, by definition we have $\omega_C^{-1}=\mathscr{H}\!\text{om}_{\mathscr{O}_C}(\omega_C, \mathscr{O}_C)$, and the statement follows from a direct computation using the cases $k\in \{0, 1\}$.
\end{proof}

We have used the notation $\eta^*$ to denote that the identification of elements of $\text{H}^0\left(C\setminus P_\bullet, \omega_C^{\otimes k}\right)$ with elements of  $\widetilde{\Gamma}$ is naturally induced by the map $\eta$. We will use the same notation in the statement of Proposition \ref{prop:chiralnodal}.

\subsection{A consequence of Riemann-Roch}
\label{sec:RR}
We will frequently use the following corollary of the Riemann-Roch theorem. Let $C$ be a smooth curve, possibly disconnected, with two non-empty sets of distinct marked points $P_\bullet=(P_1,\dots,P_n)$ and $Q_\bullet=(Q_1,\dots,Q_m)$. Assume that each irreducible component of $C$ contains at least one of the marked points $P_\bullet$, so that  $C\setminus P_\bullet$ is affine. Let $s_i$ be a formal coordinate at the point $Q_i$, for each $i$. Fix an integer $k$. For all integers $d$ and $N$, there exists $\mu\in \text{H}^0\left(C\setminus P_\bullet\sqcup Q_\bullet, \omega_C^{\otimes k}\right)$ such that its Laurent series expansions at the points $Q_\bullet$ satisfy:
\begin{align*}
\mu_{Q_i} &\equiv s_{i}^d (ds_i)^k  &&  \in && \mathbb{C}(\!( s_{i} )\!)(ds_i)^k /s_{i}^N\mathbb{C}\llbracket s_{i} \rrbracket (ds_i)^k , &&\mbox{for a fixed } i,\\
\mu_{Q_j} & \equiv 0  && \in && \mathbb{C}(\!( s_{j} )\!)(ds_j)^k /s_{j}^N \mathbb{C}\llbracket s_{j} \rrbracket (ds_j)^k, &&\mbox{for all } j\not= i.
\end{align*}
This statement has appeared for instance in \cite{ZhuGlobal}.


\section{Sheaves of vertex algebras on  curves}
\label{sec:VAstable}

We describe  the sheaf of vertex algebras $\Vv_C$ on a curve $C$ with (at worst) simple nodal singularities in \S\ref{VSheaf}, and its flat connection in \S\ref{FlatConnection}.  This yields  a coordinate-free description of the Lie algebra ancillary to $V$ in \S\ref{coordfreeLV}.

\subsection{The group scheme $\mathrm{Aut}\,\mathcal{O}$}
\label{AutO}
Consider the functor which assigns to a $\mathbb{C}$-algebra $R$ the group:
\[
\textrm{Aut}\,\mathcal{O}(R) = \left\{z \mapsto \rho(z)= a_1 z + a_2 z^2 + \cdots \, | \, a_i \in R, \, a_1 \textrm{ a unit}  \right\}
\]
of continuous automorphisms of the algebra $R\llbracket z \rrbracket$ preserving the ideal $zR\llbracket z \rrbracket$. The  group law  is given by composition of series: $\rho_1 \cdot \rho_2 := \rho_2\circ \rho_1$. This functor is represented by a group scheme, denoted $\textrm{Aut}\,\mathcal{O}$.  

\medskip

To construct the  sheaf of vertex algebras $\Vv_C$ on a stable curve $C$, we describe below the principal $(\textrm{Aut}\,\mathcal{O})$-bundle $\mathscr{A}ut_C \rightarrow C$, and actions of $\textrm{Aut}\,\mathcal{O}$ on the vertex operator algebra $V$ and on $\mathscr{A}ut_C \times V$.

\subsection{Coordinatized curves} 
Given a flat family \mbox{$C \to S$} of stable curves of genus $g$,  we construct an $(\textrm{Aut}\,\mathcal{O})$-torsor $\mathscr{A}ut_{C/S} \to C/S$ in \S \ref{blowupmu}. This torsor is  pulled back from a universal $(\textrm{Aut}\,\mathcal{O})$-torsor $\widetriangle{\mathcal{M}}_{g,1} \to \overline{\mathcal{M}}_{g,1}$. To define these objects, we begin in \S \ref{FixedCurve} with $S=\mathrm{Spec}(\mathbb{C})$.

\subsubsection{The principal $(\mathrm{Aut}\,\mathcal{O})$-bundle $\mathscr{A}ut_C$ at a fixed curve $C$}\label{FixedCurve} Assume first that $C$ is a \textit{smooth} curve.  Let $\mathscr{A}ut_C$ be the infinite-dimensional smooth variety whose points consist of pairs $(P,t)$, where $P$ is a point in $C$, and $t$ is a formal coordinate at $P$ (see \cite{adkp}).  A formal coordinate $t$ at $P$ is an element of the completed local ring $\widehat{\mathscr{O}}_P$ such that  $t\in \mathfrak{m}_P\setminus \mathfrak{m}^2_P$, where $\mathfrak{m}_P$ is the maximal ideal of $\widehat{\mathscr{O}}_P$. There is a natural forgetful map $\mathscr{A}ut_C\rightarrow C$, with fiber  at a  point $P\in C$ equal to the set of formal coordinates at $P$:
\[
\mathscr{A}ut_P =\left\{t \in \widehat{\mathscr{O}}_P \,\, | \,\, t\in \mathfrak{m}_P\setminus \mathfrak{m}^2_P \right\}.
\]
The group scheme $\textrm{Aut}\,\mathcal{O}$ acts on fibers of $\mathscr{A}ut_C \rightarrow C$ by change of coordinates: 
\[
\mathscr{A}ut_C \times \textrm{Aut}\,\mathcal{O} \rightarrow \mathscr{A}ut_C, \qquad \left((P,t),\rho\right) \mapsto \left(P, t\cdot \rho := \rho(t)\right).
\]
This action is simply transitive, thus $\mathscr{A}ut_C$ is a  principal $(\textrm{Aut}\,\mathcal{O})$-bundle on~$C$. The choice of a formal coordinate $t$ at $P$ gives rise to the trivialization
\[
\mathrm{Aut}\,\mathcal{O}\xrightarrow{\simeq_t}\mathscr{A}ut_P, \qquad \rho\mapsto \rho(t).
\]

For a \textit{nodal} curve $C$,  let $\Ct\rightarrow C$ denote the normalization of $C$. Assume for simplicity that $C$ has only one node $Q$, and let $\Qp$ and $\Qm$ be the two preimages in $\Ct$ of  $Q$. A choice of formal coordinates $s_+$ and $s_-$ at $\Qp$ and $\Qm$, respectively, determines a smoothing of the nodal curve $C$ over $\mathrm{Spec}(\mathbb{C}\llbracket q\rrbracket)$ such that,  locally around the point $Q$ in $C$, the family is defined by $s_+ s_- = q$ (see \S\ref{smoothings} for more details). A principal $(\mathrm{Aut}\,\mathcal{O})$-bundle on a nodal curve is equivalent to the datum of a principal $(\mathrm{Aut}\,\mathcal{O})$-bundle on its normalization together with a gluing isomorphism between the fibers over the two preimages of each node. In particular, one constructs the principal $(\mathrm{Aut}\,\mathcal{O})$-bundle $\mathscr{A}ut_C$ on $C$ from the principal $(\mathrm{Aut}\,\mathcal{O})$-bundle $\mathscr{A}ut_{\Ct}$ on $\Ct$ by identifying the fibers at the two preimages $\Qp$ and $\Qm$ of the node $Q$ in $C$ by the following gluing isomorphism:
\begin{equation}
\label{eq:glisomAut}
\mathscr{A}ut_{\Qp} \simeq_{s_+} \mathrm{Aut}\,\mathcal{O}  \xrightarrow{\cong} \mathrm{Aut}\,\mathcal{O} \simeq_{s_-} \mathscr{A}ut_{\Qm},
\qquad  \rho(s_+)\mapsto \rho \circ \gamma (s_-),
\end{equation}
where $\gamma\in \mathrm{Aut}\,\mathcal{O}$ is defined as
\begin{equation}
\label{eq:muz}
\gamma(z) := \frac{1}{1+z}-1 = -z+z^2-z^3+\cdots.
\end{equation}
Note that $(\gamma\circ\gamma)(z)=z$, hence \eqref{eq:glisomAut} determines an involution of $\mathrm{Aut}\,\mathcal{O}$. The isomorphism \eqref{eq:glisomAut} is induced from the identification $s_+=\gamma(s_-)$.

\subsubsection{$\mathscr{A}ut_{C/S}$ on a family $C\to S$ and the moduli space $\widetriangle{\mathcal{M}}_{g,1}$}
\label{blowupmu}
The above construction of $\mathscr{A}ut_C$ can be carried out in families, and for a flat family $C \rightarrow S$ of stable curves, one thus obtains a principal $(\mathrm{Aut}\,\mathcal{O})$-bundle $\mathscr{A}ut_{C/S}\rightarrow C/S$. This determines a principal $(\mathrm{Aut}\,\mathcal{O})$-bundle \mbox{$\widetriangle{\mathcal{C}}_{g}\rightarrow \overline{\mathcal{C}}_g$,} where $\overline{\mathcal{C}}_g\rightarrow \overline{\mathcal{M}}_g$ is the universal curve over the moduli space of stable curves of genus $g$. One has the following diagram made of fiber product squares:
\[
\begin{tikzcd}
\mathscr{A}ut_{C/S} \arrow[rightarrow]{r} \arrow[rightarrow]{d}& \widetriangle{\mathcal{C}}_{g} \arrow[rightarrow]{d}{\textrm{Aut}\,\mathcal{O}} \\
C \arrow[rightarrow]{r} \arrow[rightarrow]{d}& \overline{\mathcal{C}}_{g} \arrow[rightarrow]{d}\\
S \arrow[rightarrow]{r}{}& 	\overline{\mathcal{M}}_{g}.
\end{tikzcd}
\]

It can be convenient to identify the universal curve $\overline{\mathcal{C}}_g$ with the moduli space of stable pointed curves $\overline{\mathcal{M}}_{g,1}$. This leads us to consider a principal $(\mathrm{Aut}\,\mathcal{O})$-bundle $\widetriangle{\mathcal{M}}_{g,1}\rightarrow \overline{\mathcal{M}}_{g,1}$  identified with $\widetriangle{\mathcal{C}}_{g}\rightarrow\overline{\mathcal{C}}_g$. Namely:

\begin{defn}
Let  $\widetriangle{\mathcal{M}}_{g,1}$ be the moduli space of {\em{coordinatized stable pointed curves}}. An object over a scheme $S$ consists of a
semistable curve $C \to S$ with a section $P\colon S \to C$ mapping to the smooth locus of $C$, together with a formally unramified thickening $S \times  {\rm{Spf}}(\mathbb{C}\llbracket t\rrbracket) \to C$ of the section $P$, such that every genus one  component has at least one special point and every rational component has at least three special points. 
\end{defn}

Here a special point is either a marked point or a node, counted with multiplicity. Moreover, ${\rm{Spf}}(\mathbb{C}\llbracket t\rrbracket)$ is the formal spectrum of the complete topological ring $\mathbb{C}\llbracket t\rrbracket$ \cite[\S A.1.1]{bzf}.

The action of $\mathrm{Aut}\,\mathcal{O}$ via change of coordinates identifies $\widetriangle{\mathcal{M}}_{g,1}$ as a principal $(\mathrm{Aut}\,\mathcal{O})$-bundle over $\overline{\mathcal{M}}_{g,1}$. One has  the diagram of Cartesian squares:
\[
\begin{tikzcd}
\mathscr{A}ut_{P/S} \arrow[rightarrow]{r} \arrow[rightarrow]{d} &\mathscr{A}ut_{C/S} \arrow[rightarrow]{r} \arrow[rightarrow]{d}& \widetriangle{\mathcal{M}}_{g,1} \arrow[rightarrow]{d}{\textrm{Aut}\,\mathcal{O}} \\
S \arrow[rightarrow]{r}{P} &C \arrow[rightarrow]{r} \arrow[rightarrow]{d}& \overline{\mathcal{M}}_{g,1} \arrow[rightarrow]{d}\\
&S \arrow[rightarrow]{r}{}& 	\overline{\mathcal{M}}_{g}.
\end{tikzcd}
\]
In particular, for a flat family $C\rightarrow S$  of stable curves of genus~$g$, the space $\mathscr{A}ut_{C/S}$ is the pull-back of $\widetriangle{\mathcal{M}}_{g,1}$ via the moduli map $S \rightarrow \overline{\mathcal{M}}_{g}$.

\smallskip

For a nodal curve $C$, the fiber of $\mathscr{A}ut_C$ over a node can be described as follows. Assume for simplicity that $C$ has only one node $Q$. Over the singular point $Q\in C$, the fiber $\mathscr{A}ut_Q$ can be identified with the space of formal coordinates at the point $Q'$ in $C'$, where $(C',Q')$ is formed by stabilization of the unstable pointed curve $(C,Q)$. Stabilization is carried out by blowing-up. As a result, $C'$ consists of the partial normalization $\Ct$ of $C$ at $Q$ together with a rational exceptional component meeting $\Ct$ transversally at the two preimages $\Qp$ and $\Qm\in \Ct$ of the node $Q$.  Such a rational component contains three special points: the two attaching points and a  point labelled~$Q'$. Up to isomorphism, one can identify this rational component with $\mathbb{P}^1$, with the points attached to $\Qp$ and $\Qm$ identified with $0$ and $\infty \in \mathbb{P}^1$, respectively, and the point $Q'$ identified with $1 \in \mathbb{P}^1$. The fiber $\mathscr{A}ut_Q$ of $\mathscr{A}ut_C$ over the node $Q$ in $C$ is identified with the space of formal coordinates at the point $1$ in $\mathbb{P}^1\subset C'$. 

Here are more details: Choose formal coordinates $s_+$ and $s_-$ at $\Qp$ and $\Qm$, respectively. As mentioned in the previous section, a choice of such coordinates determines a smoothing of $C$ over $\mathrm{Spec}(\mathbb{C}\llbracket q\rrbracket)$ such that,  locally around the point $Q$ in $C$, the family is defined by $s_+ s_- = q$. After blowing-up such a family at the point $Q$, locally around the two resulting nodes incident to the exceptional component $\mathbb{P}^1$, the curve $C'$ at $q=0$ is given by equations $s_+s_0=0$ and $s_\infty s_-=0$, where $s_0,s_\infty$ are formal coordinates at $0,\infty \in \mathbb{P}^1$, respectively. The formal coordinates $s_0$ and $s_\infty$ satisfy $s_0 s_\infty=1$. A pair of such coordinates $s_0,s_\infty$ induce formal coordinates $s_0-1$ and $s_\infty-1$ at $1\in \mathbb{P}^1$, with change of coordinates given precisely by $s_0-1 = \gamma(s_\infty-1)$, where $\gamma$ is as in \eqref{eq:muz}. It follows that for the nodal curve $C'$, one has the identifications $s_+ = \gamma(s_0)$, $s_- = \gamma(s_\infty)$, and $s_0-1 = \gamma(s_\infty-1)$. These identifications determine the identification of the space $\mathscr{A}ut_C$  with the fiber of the projection  $\widetriangle{\mathcal{M}}_{g,1}\rightarrow \overline{\mathcal{M}}_{g}$ over the moduli point $[C] \in 	\overline{\mathcal{M}}_{g}$.

More generally, one defines $\widetriangle{\mathcal{M}}_{g,n}$ as the moduli space  of objects of type $(C,P_\bullet, t_\bullet)$, where $(C,P_\bullet=(P_1,\dots,P_n))$ is a stable $n$-pointed genus $g$ curve, and $t_\bullet=(t_1,\dots,t_n)$ with $t_i$  a formal coordinate at $P_i$, for each $i$. The action of $(\textrm{Aut}\,\mathcal{O})^{\oplus n}$ by change of coordinates endows $\widetriangle{\mathcal{M}}_{g,n}$ with the structure of an $(\textrm{Aut}\,\mathcal{O})^{\oplus n}$-torsor over $\overline{\mathcal{M}}_{g,n}$. For further information about $\widetriangle{\mathcal{M}}_{g,n}$ over the locus of smooth curves, see  \cite{adkp} and \cite[\S 6.5]{bzf}; over  stable curves with singularities, see \cite[\S 3.2]{dgt}.

\subsection{Action of $\textup{Aut}\,\mathcal{O}$ on  $V$}
Let $\textrm{Der}_0\, \mathcal{O}$ be the Lie algebra functor attached to the group functor $\textrm{Aut}\,\mathcal{O}$ in the sense of \cite[Exp.~II \S 3,4]{sga3}. This~is
\[
\textrm{Der}_0\, \mathcal{O}(R) = R\llbracket z \rrbracket z\partial_z.
\] 
The Lie algebra $\textrm{Der}_0\, \mathcal{O}(R)$ is generated over $R$ by the Virasoro elements $L_p$, for $p\geq 0$, hence $\textrm{Der}_0\, \mathcal{O}$ is a Lie subalgebra of the Virasoro Lie algebra.   The action of the Virasoro Lie algebra on a vertex operator algebra $V$ restricts to  an action of  $\textrm{Der}_0$ on $V$. One can integrate this to obtain a left action of $\textrm{Aut}\,\mathcal{O}$ on $V$ defined as the inductive limit of the finite-dimensional submodules  $V_{\leq k}:=\bigoplus_{i\leq k}V_i$.  This follows from the fact that $L_0$ acts semi-simply with integral eigenvalues, and $L_p$ acts locally nilpotently for $p>0$ \cite[\S 6.3]{bzf}.
 
Explicitly, to compute the action on $V$ of an  an element $\rho\in\mathrm{Aut}\,\mathcal{O}$, one proceeds as follows.  The element $\rho(z)$ can be expressed as 
\[
\rho(z)=\exp\left(\sum_{i\geq 0} a_i z^{i+1}\partial_z\right)(z)
\] 
for some  $a_i\in\mathbb{C}$ (see e.g., \cite[\S6.3.1]{bzf}). Assuming $0\leq \mathrm{Im}(a_0)<2\pi$, the coefficients $a_i$ are uniquely determined. Hence, $\rho$ acts on $V$ as $\exp\left(\sum_{i\geq 0} -a_i L_i\right)$.
 
As an example and for later use,  the element $\gamma \in \textrm{Aut}\,\mathcal{O}$ from \eqref{eq:muz} can be expressed as
\begin{equation}
\label{eq:muzexp}
\gamma(z)= e^{-z^2\partial_z}(-1)^{-z\partial_z}(z).
\end{equation}
Thus, $\gamma$ acts on $V$ as $e^{L_1}(-1)^{L_0}$.
This is a special case of the computation in \cite[(10.4.3)]{bzf} and is essential to the gluing isomorphism for $\mathcal{V}_C$~below.

\subsection{Action of $\textup{Aut}\,\mathcal{O}$ on $\mathscr{A}ut_C \times V$}
The group $\textrm{Aut}\,\mathcal{O}$ has a right equivariant action on the trivial bundle $\mathscr{A}ut_C \times V\rightarrow \mathscr{A}ut_C$ defined by
\[
(P,t, A)\cdot \rho = \left(P, \rho(t), \rho^{-1}\cdot A\right),
\]
for $\rho\in \textrm{Aut}\,\mathcal{O}$ and $(P,t, A)\in \mathscr{A}ut_C \times V$.

\subsection{The sheaf of vertex algebras}
\label{VSheaf}

As we next describe,  the sheaf $\Vv_C$ of vertex algebras on a smooth curve $C$ is constructed by faithfully flat descent of an $(\textrm{Aut}\,\mathcal{O})$-equivariant sheaf along an $(\textrm{Aut}\,\mathcal{O})$-torsor, in order to remove coordinate dependence. The description of the sheaf $\Vv_C$ over a nodal curve is more complex. See \S \ref{sec:sheafofcoinariantsfinal} for the extension to families of stable curves.

Assume first that the curve $C$ is \textit{smooth}. The quotient of  $\mathscr{A}ut_C \times V$ by the action of $\textrm{Aut}\,\mathcal{O}$ descends along the map $\pi\colon \mathscr{A}ut_C \rightarrow C$ to the \textit{vertex algebra bundle} $V_C$ on $C$:
\[
\begin{tikzcd}
\mathscr{A}ut_C \times V \arrow[rightarrow]{r} \arrow[rightarrow]{d}& \mathscr{A}ut_C \times_{\textrm{Aut}\,\mathcal{O}}V =:V_C \arrow[rightarrow]{d}\\
\mathscr{A}ut_C \arrow[rightarrow]{r}{\textrm{Aut}\,\mathcal{O}}[swap]{\pi} & C.
\end{tikzcd}
\]
In $V_C$, one has identities
\begin{equation}
\label{eq:idVC}
(P,t, A) = \left(P, \rho(t), \rho^{-1}\cdot A\right),
\end{equation}
for $\rho\in \textrm{Aut}\,\mathcal{O}$ and $(P,t, A)\in \mathscr{A}ut_C \times V$.

The sheaf of sections of $V_{C}$ is the sheaf of vertex algebras:
\[
\Vv_C:=\left( V \otimes \pi_* \, \mathscr{O}_{\mathscr{A}ut_C} \right)^{\textrm{Aut}\,\mathcal{O}}.
\]
This is a locally-free  sheaf of $\mathscr{O}_C$-modules on $C$. The fiber of $\Vv_C$ at a point $P\in C$ is isomorphic to $\mathscr{A}ut_P\times_{\textrm{Aut}\,\mathcal{O}}V$. Given a formal coordinate $t$ at $P$, one has the trivialization 
\[
\mathscr{A}ut_P\times_{\textrm{Aut}\,\mathcal{O}}V \simeq_t V.
\]

For a \textit{nodal} curve $C$, the description of $\Vv_C$ is more involved. Assume for simplicity that $C$ has only one node $Q$. Let $\eta \colon \Ct \to C$ be the normalization of $C$, and let $\Qp$ and $\Qm$ in $\Ct$ be the two preimages of $Q$.  Choose formal coordinates $s_+$ and $s_-$ at $\Qp$ and $\Qm$, respectively. Denote by $\widetilde{\pi}\colon \mathscr{A}ut_{\Ct} \rightarrow \Ct$ the natural projection.  The action of $\textrm{Aut}\,\mathcal{O}$ on $V \otimes \widetilde{\pi}_* \, \mathscr{O}_{\mathscr{A}ut_{\Ct}}$ restricts to an action of $\textrm{Aut}\,\mathcal{O}$ on
\[
\bigoplus_{k\geq 0}\,  V_k  \otimes  \mathscr{O}_{\Ct}\left(-kQ_+ -kQ_- \right)  \otimes  \widetilde{\pi}_*  \,\mathscr{O}_{\mathscr{A}ut_{\Ct}}.
\]
Consider  the sheaf
\begin{equation}
\label{eq:VprimeCtilde}
\widetilde{\Vv} :=\left( \bigoplus_{k\geq 0}\,  V_k  \otimes  \mathscr{O}_{\Ct}\left(-kQ_+ -kQ_- \right)  \otimes  \widetilde{\pi}_*  \,\mathscr{O}_{\mathscr{A}ut_{\Ct}}  \right)^{\textrm{Aut}\,\mathcal{O}}. 
\end{equation}
The sheaf $\mathcal{V}_C$ is realized as a subsheaf of $\eta_* (\widetilde{\mathcal{V}})$ which coincides with $\eta_*(\widetilde{\Vv})$ on $C\setminus Q$. To describe its fiber over $Q$, we  use the involution \mbox{$\gamma\in \textup{Aut}\,\mathcal{O}$} from \eqref{eq:muz} and \eqref{eq:muzexp}, hence for homogeneous $A \in V$ of degree $a$, one has
\[
\gamma(A)= e^{L_1}(-1)^{L_0}A=(-1)^{a}\sum_{i \geq 0} \dfrac{1}{i!} L_1^i(A).
\] 
The fiber of $\Vv_C$ over $Q$ is obtained by identifying the fibers of $\widetilde{\mathcal{V}}$ at $Q_+$ and $Q_-$  via the gluing isomorphism induced by  $\gamma$ as in the diagram below:
\[
\begin{tikzcd}
V \arrow{r}{\gamma}& V \arrow{d}[]{\cong} \\
\underset{k\geq 0}{\bigoplus} V_k \otimes_\mathbb{C} s_+^{k} \arrow{u}{\cong} & \underset{k\geq 0}{\bigoplus} V_k \otimes_\mathbb{C} s_-^{k} \arrow{d}[]{\simeq_{s_-}} \\
 \mathscr{A}ut_{\Qp}\times_{\textrm{Aut}\,\mathcal{O}} V\arrow[dashed]{r}{\cong} \arrow{u}{\simeq_{s_+}} & 
 \mathscr{A}ut_{\Qm}\times_{\textrm{Aut}\,\mathcal{O}} V.
\end{tikzcd}
\]
Equivalently, after trivializing with respect to $s_\pm$, the gluing isomorphism is 
\begin{align*}
\underset{k\geq 0}{\bigoplus} V_k \otimes_\mathbb{C} s_+^{k}  & \rightarrow  \underset{k\geq 0}{\bigoplus} V_k \otimes_\mathbb{C} s_-^{k}, &  
A \otimes_\mathbb{C} s_+^{k}  & \mapsto  (-1)^{k}\sum_{i \geq 0} \dfrac{1}{i!} L_1^i(A) \otimes_\mathbb{C} s_-^{k-i}.
\end{align*}

We next describe spaces of sections of $\Vv_C$ for the nodal curve $C$. Over an open set $U\subset C$ with $Q\notin U$, one has \mbox{$\Vv_C|_U := \eta_*\left(\Vv_{\Ct}\,|\,_{\eta^{-1}(U)}\right)$.} Since $\eta$ is an isomorphism away from $Q$, sections of $\Vv_C$ over $U$ are defined as in the case of smooth curves. To define  spaces of sections of $\Vv_C$ over an open set containing $Q$, it is sufficient to do so on the formal neighborhood
\begin{equation}
\label{eq:DQ}
D_Q=\text{Spec} \,\widehat{\mathscr{O}}_Q =\text{Spec}(\CC \llbracket \tQp, \tQm \rrbracket /(\tQp\tQm)).
\end{equation}
The space of sections  $\Vv_C(D_Q)$ is defined  as the subspace of 
\[\eta_*{\widetilde{\Vv}}(D_Q)=\bigoplus_{k \geq 0} V_k \otimes \left(\tQp^k \CC\llbracket\tQp\rrbracket \oplus \tQm^k \CC\llbracket\tQm\rrbracket \right) \] 
consisting of elements in the kernel of the map 
\begin{equation*} 
\bigoplus_{k \geq 0} V_k \otimes \left(\tQp^k \CC\llbracket\tQp\rrbracket \oplus \tQm^k \CC\llbracket\tQm\rrbracket \right)  \longrightarrow     V
\end{equation*} 
induced by
\begin{align*}
 \left(A \otimes s_+^a f(s_+), B\otimes s_-^b g(s_-)\right) \mapsto f(0)\,\gamma(A) -g(0)\,B
\end{align*}
for homogeneous $A,B \in V$ with $\deg(A)=a$ and $\deg(B)=b$, and for all $f(s_+) \in \CC\llbracket\tQp\rrbracket$ and $g(s_-) \in \CC\llbracket\tQm\rrbracket$. Hence, $\Vv_C(D_Q)$ is spanned by elements
\[
\Big(A \otimes s_+^a, \, (-1)^a \sum_{i\geq 0} \dfrac{1}{i!}L_1^i(A) \otimes s_-^{a-i}\Big) \ \mbox{ and } \ \left(B \otimes \tQp^{b+1}f(s_+), D \otimes \tQm^{d+1}g(s_-)\right)
\]
for homogeneous  $A\in V_a$, $B \in V_b$, and $D \in V_d$, with $f(s_+) \in \CC\llbracket\tQp\rrbracket$ and $g(s_-) \in \CC\llbracket\tQm\rrbracket$. The sheaf $\Vv_C|_{D_Q}$ is naturally an $\widehat{\mathscr{O}}_Q$-module.

\smallskip

To summarize, the sheaf $\Vv_C$ on the nodal curve $C$ can be described as follows. Consider the exact sequence of $\mathscr{O}_{\widetilde{C}}$-modules
\[ 
\begin{tikzcd}
0 \ar[r] & \mathscr{O}_{\widetilde{C}}(-Q_+-Q_-) \otimes \widetilde{\Vv} \ar[r] & \widetilde{\Vv} \ar{r}{} & V_{Q_+} \oplus V_{Q_-} \ar[r] & 0,
\end{tikzcd}
\] 
where $V_{Q_\pm}$ is the skyscraper sheaf supported at $Q_\pm$ with space of sections isomorphic to $V$ via the choice of the coordinate $s_\pm$. Pushing forward this sequence along $\eta$, we obtain an exact sequence which fits in the diagram
\begin{equation} 
\label{eq:diagramVC} 
\begin{tikzcd}
0 \ar[r] & \eta_* \left(\mathscr{O}_{\widetilde{C}}(-Q_+-Q_-) \otimes \widetilde{\Vv}\right) \ar[r] & \eta_* \widetilde{\Vv} \ar{r}{q} & V_{Q}^{\oplus 2} \ar[r] \ar[->>]{d}{\gamma\,\circ\, \pi_1 -\pi_2} & 0 \\
&&& V_Q.
	\end{tikzcd}
\end{equation}
Here, $\pi_i\colon V_{Q} ^{\oplus 2} \rightarrow V_{Q}$ is the natural projection, for $i=1,2$.
The sheaf $\Vv_C$ is then defined as $\ker\left( (\gamma\circ \pi_1 -\pi_2) \circ q\right)$.

\subsection{The structure of the sheaf $\Vv_C$}\label{VSheafS}
We describe here some properties of $\Vv_C$. For a smooth curve $C$, the sheaf $\Vv_C$ is filtered by  the sheaves $\Vv_{\leq k}$ defined as the sheaves of sections of the vector bundles of finite rank $\mathscr{A}ut_C \mathop{\times}_{\textup{Aut}\, \mathcal{O}} V_{\leq k}$. While the action of $\textrm{Aut}\,\mathcal{O}$ on $V_{\leq k}$ is well-defined,  the action of $\textrm{Aut}\,\mathcal{O}$ on $V_k$ is so only modulo $V_{\leq k-1}$, for each~$k$. In particular, $\Vv_{\leq k}$ is well-defined, but $\mathscr{A}ut_C \mathop{\times}_{\textup{Aut}\, \mathcal{O}} V_{k}$ only makes sense as a quotient of $\mathscr{A}ut_C \mathop{\times}_{\textup{Aut}\, \mathcal{O}} V_{\leq k}$ modulo $\mathscr{A}ut_C \mathop{\times}_{\textup{Aut}\, \mathcal{O}} V_{\leq k-1}$.

Assuming for simplicity that the curve $C$ has only one node $Q$, the sheaves 
\[
\left( \bigoplus_{i=0}^k\,  V_i  \otimes  \mathscr{O}_{\Ct}\left(-i\,Q_+ -i\,Q_- \right)  \otimes  \widetilde{\pi}_*\,  \mathscr{O}_{\mathscr{A}ut_{\Ct}}  \right)^{\textrm{Aut}\,\mathcal{O}}
\]
provide an increasing filtration of the sheaf $\widetilde{\Vv}$ on the normalization $\Ct$ from \eqref{eq:VprimeCtilde}. The restriction of the gluing isomorphism in \S\ref{VSheaf} gives a gluing isomorphism between the fibers of such sheaves at the two preimages $Q_+$ and $Q_-$ of the node $Q$, and hence induces an increasing filtration of $\Vv_C$. We denote the subsheaves of such a filtration as $\Vv_{\leq k}$, as in the smooth case.

Consider the associated graded sheaf
\[
\textup{gr}_\bullet \Vv_C :=\oplus_{k\geq 0}\, \textup{gr}_k \Vv_C, \qquad\mbox{where}\qquad \textup{gr}_k \Vv_C:= \Vv_{\leq k} / \Vv_{\leq k-1}.
\]

\begin{lemma}
\label{lemma:gr}
One has
\[
\textup{gr}_\bullet \Vv_C \cong  \oplus_{k\geq 0}  \left(\omega_C^{\otimes -k} \right)^{\oplus \dim V_k}.
\]
\end{lemma}

This was proved in \cite[\S 6.5.9]{bzf} for smooth curves. The  argument made there extends to stable curves if one replaces the sheaf of differentials $\Omega^1_C$ with the dualizing sheaf $\omega_C$. We sketch the proof for the reader's convenience.

\begin{proof}
Consider $V_k$ as the quotient $(\textup{Aut}\, \mathcal{O})$-representation $V_{\leq k}/V_{\leq k-1}$, and let $A\in V_k$ be nonzero. One has $L_0\cdot A=kA$ and $L_p\cdot A=0$ in $V_k=V_{\leq k}/V_{\leq k-1}$ for $p>0$. Assume first that $C$ is smooth. Then it follows that 
\[
\mathscr{A}:=\left(\mathbb{C}A\otimes \pi_*\,\mathscr{O}_{\mathscr{A}ut_{C}}\right)^{\textup{Aut}\, \mathcal{O}}
\]
 is a line sub-bundle of $\textup{gr}_k \Vv_C$. From \cite[\S 6.5.9]{bzf}, the transition functions of $\mathscr{A}$ and $\omega_C^{\otimes -k}$ match, hence one concludes that $\mathscr{A}\cong \omega_C^{\otimes -k}$, and this implies the statement.

Next, we consider the case when $C$ is nodal. Assume for simplicity that $C$ has only one node $Q$. Consider the line bundle $\mathscr{A}$  constructed from the line bundle
\[
\mathscr{A}' := \left(\mathbb{C}A\otimes \widetilde{\pi}_*\,\mathscr{O}_{\mathscr{A}ut_{\Ct}}\right)^{\textup{Aut}\, \mathcal{O}}\otimes \mathscr{O}_{\Ct}(-kQ_+ -kQ_-)
\]
on the normalization $\Ct$ of $C$ by identifying the fibers at the preimages $Q_+$ and $Q_-$ of the node. From the discussion of the smooth case, one has 
 \[
 \mathscr{A}' \cong \omega_{\Ct}^{\otimes -k}\left(-kQ_+ -kQ_- \right).
 \]
It remains to determine the isomorphisms used to identify the fibers at $Q_+$ and $Q_-$.  The gluing isomorphism $A\mapsto e^{L_1}(-1)^{L_0}A$ of $\Vv_{\leq k}$ from \S\ref{VSheaf} induces the  gluing  $A\mapsto (-1)^{L_0}A$ of $\textup{gr}_k \Vv_C$. In particular, this is the gluing isomorphism of $\mathscr{A}$, which coincides with the gluing of sections for $\omega_C^{\otimes -k}$ given by the condition on the residues in Lemma \ref{lem:omegaotimesm}. Hence one has $\mathscr{A}\cong \omega_C^{\otimes -k}$ in this case as well. The assertion therefore follows.
\end{proof}

\begin{rmk} The reader will notice how both sheaves $\left(\omega_C^{\otimes -k} \right)^{\oplus \dim V_k}$ and $\Vv_{\leq k}/\Vv_{\leq k-1}$ are defined using diagrams similar to the one in \eqref{eq:diagramVC}. From this point of view, the above proof  can be seen as comparing the gluing data encoded by the  vertical maps in the corresponding diagrams.
\end{rmk}

As a consequence of \cite[\S 6.5.9]{bzf}, which is Lemma \ref{lemma:gr} for smooth curves, one has the following statement, which will be used throughout.

\begin{lemma}
\label{lem:small}
Let $(C,P_\bullet)$ be a smooth $n$-pointed curve. One has:  
\[
\mathrm{H}^0(C\setminus P_\bullet, \Vv_C )\cong  \mathrm{H}^0(C\setminus P_\bullet,  \textup{gr}_\bullet\Vv_C).
\]
\end{lemma}

\begin{proof}
We claim that on the affine open set $C\setminus P_\bullet$, one has 
\begin{equation}\label{claim1}
\Vv_{\leq k} \cong  \oplus_{i\le k}  \Vv_{\leq i} / \Vv_{\leq i-1} =  \textup{gr}_{\leq k} \Vv_C.
\end{equation}
 Assuming \eqref{claim1}, then one has, for every $k \in \mathbb{Z}_{\ge 0}$, an injection
\[
\phi_k \colon \textup{gr}_{\leq k}\Vv_C \hookrightarrow \Vv_C,
\]
altogether defining a map
$\phi \colon \textup{gr}_\bullet \Vv_C \longrightarrow \Vv_C.$
The map $\phi$ gives the isomorphism we seek.  To  see that $\phi$ is injective, note that any element $x$ in $\textup{gr}_\bullet \Vv_C$  is in fact a finite sum, and hence $x$ is in $\textup{gr}_{\leq k}\Vv$ for some $k$.  So if $x$ is mapped to zero by $\phi$, then $x$ is mapped to zero by $\phi_k$ for some $k$.  Since all maps $\phi_k$ are injective, $x$ is zero.  Surjectivity of $\phi$ follows from the fact that $\Vv_C$ is filtered by the
$\Vv_{\leq k}$.

We prove \eqref{claim1} by induction on $k$ with base case $k=0$.  Lemma \ref{lemma:gr} implies that $\textup{gr}_k \Vv_C$ is locally free.  On affines, locally free sheaves are  projective, and hence on the affine open set $C\setminus P_\bullet$, the following sequence splits:
\[
0 \rightarrow \Vv_{\leq k-1} 
\rightarrow \Vv_{\leq k}\rightarrow \textup{gr}_k \Vv_C \rightarrow 0.
\]
In particular, on $C\setminus P_\bullet$
\[
\Vv_{\leq k}\cong    \Vv_{\leq k-1} \oplus \textup{gr}_k \Vv_C
 \cong  \textup{gr}_{\leq k-1} \Vv_C \oplus \textup{gr}_k \Vv_C,
\]
 and \eqref{claim1} holds.
\end{proof}

\subsection{The logarithmic connection}
\label{FlatConnection}

In this section, we describe the logarithmic connection on $\Vv_C$. For smooth curves, this was defined in \cite[\S 6]{bzf}. The adjective \textit{logarithmic} is used here to emphasize that we work  with possibly nodal curves. For this purpose, we replace the sheaf $\Omega^1_C$, used in  \cite[\S 6]{bzf}, with the dualizing sheaf $\omega_C$ which arises from logarithmic differentials on the normalization of $C$. Hence a \textit{logarithmic connection} on $\Vv_C$ is a $\mathbb{C}$-linear map $\nabla \colon \Vv_C \to \Vv_C \otimes \omega_C$ satisfying $\nabla(fs)= s\otimes df + f\nabla(s)$ for all local sections $f$ of $\mathscr{O}_C$ and  $s$ of $\Vv_C$.  We will use that on open sets $U\subset C$ where $\omega_C$ is trivial,  one can describe a connection as an endomorphism of~$\Vv_C|_U$.

\subsubsection{The connection on smooth curves} \label{connonSmooth} Let $C$ be a smooth curve. On an open subset $U$ of $C$ admitting a global coordinate~$t$  (for instance, on an open $U$ admitting an \'etale map $U\rightarrow \mathbb{A}^1$), one has a trivialization $\Vv_C|_U\simeq_t V\otimes \mathscr{O}_U$. On $\Vv_C|_U$, the connection $\nabla$ is given by the endomorphism $L_{-1}\otimes \mathrm{id}_U+ \mathrm{id}_V\otimes \partial_t$ (compare with~\eqref{eq:partialancillary}).  
This canonically defines a connection $\nabla \colon \Vv_C \to \Vv_C \otimes \omega_C$ independent of the choice of the coordinate $t$ (as in \cite[Thm 6.6.3]{bzf}).

\subsubsection{The connection on nodal curves}
Let $C$ be a nodal curve, and let $\eta \colon \Ct\rightarrow C$ be  its normalization.

\begin{proposition}\label{prop:LC}  
The connection  on $\Vv_{\widetilde{C}}$ described in \S\ref{connonSmooth} naturally induces a logarithmic connection $\nabla \colon \Vv_C \rightarrow \Vv_C\otimes \omega_C$.
\end{proposition}

\begin{proof} 
For simplicity, we assume that $C$ has only one node $Q$.
Recall that the sheaf $\Vv_C$ is defined as a subsheaf of $\eta_*(\widetilde{\Vv})$, where $\widetilde{\Vv}$ is the sheaf  in \eqref{eq:VprimeCtilde}, and similarly, $\omega_C$ is a subsheaf of $\eta_*\,\omega_{\widetilde{C}}(Q_++Q_-)$. Since $\widetilde{C}$ is a smooth curve, restricting the prescription given in \S \ref{connonSmooth} to 
$\widetilde{\Vv}\subset\Vv_{\widetilde{C}}$
 defines 
\[ \widetilde{\nabla} \colon \widetilde{\Vv}\longrightarrow \widetilde{\Vv}\otimes \mathscr{O}_{\Ct}(Q_++Q_-)\otimes \omega_{\widetilde{C}}.
\] 
Pushing forward to $C$ along $\eta$ and restricting to $\Vv_C\subset \eta_*(\widetilde{\Vv})$, we obtain 
\begin{equation} 
\label{eq:defiNablaNodal} 
\nabla \colon \Vv_C \longrightarrow \eta_*\left(\widetilde{\Vv}\otimes \omega_{\widetilde{C}}(Q_++Q_-)\right).
\end{equation} 
We claim that this factors through $\Vv_C \otimes \omega_C$, defining in this way the desired logarithmic connection. That this holds away from $Q$ is clear. 
Therefore, it remains to check the assertion over the formal neighborhood $D_Q$ from \eqref{eq:DQ}.

The sheaf $\omega_C$ is locally free over $C$, and by an argument analogous to the proof of Lemma \ref{lem:omegaotimesm}, we obtain an explicit trivialization over $D_Q$ as  \begin{equation*} 
\omega_C(D_Q) \cong \dfrac{\CC\llbracket s_+, s_-\rrbracket}{(s_+s_-)} \left(\dfrac{ds_+}{s_+}, - \dfrac{ds_-}{s_-}\right) \subset \CC \llbracket s_+\rrbracket \dfrac{ds_+}{s_+} \oplus \CC \llbracket s_-\rrbracket \dfrac{ds_-}{s_-}.
\end{equation*} 
Using this, the restriction of \eqref{eq:defiNablaNodal} to $D_Q$ can be explicitly written as 
\begin{equation}
\label{eq:nablaDQ}
\nabla \colon \Vv_C(D_Q) \longrightarrow \bigoplus_{k\geq 0} V_k \otimes \left( s_+^{k-1} \CC \llbracket s_+\rrbracket \, ds_+  \oplus s_-^{k-1}\CC \llbracket s_-\rrbracket \, ds_- \right).
\end{equation}
To conclude, we show that this factors through $\Vv_C(D_Q)\otimes \left(\frac{ds_+}{s_+}, - \frac{ds_-}{s_-}\right)$. 
By definition, the image via $\nabla$ of an element $(A\otimes f(s_+), B\otimes g(s_-)) \in \Vv_C(D_Q)$, with $A$, $B \in V$, is 
\[
(L_{-1}(A) \otimes f(s_+) ds_+ + A \otimes f'(s_+)ds_+, L_{-1}(B) \otimes g(s_-) ds_- + B \otimes g'(s_-)ds_-).\] 
The coefficient of $\left(\frac{ds_+}{s_+}, - \frac{ds_-}{s_-}\right)$ is 
\begin{equation}\label{eq:exprVc}
(L_{-1}(A)\otimes s_+f(s_+)+A \otimes s_+f'(s_+),-L_{-1}(B)\otimes s_-g(s_-) - B\otimes s_-g'(s_-)).
\end{equation}
We are then left to check that if $(A\otimes f(s_+), B\otimes g(s_-)) \in \Vv_C(D_Q)$, then  \eqref{eq:exprVc} is also an element of $\Vv_C(D_Q)$. To do so, it is enough to check that this is true for the elements of  type 
\[
\Bigg(A \otimes s_+^a, \, (-1)^a \sum_{i\geq 0} \dfrac{1}{i!}L_1^i(A) \otimes s_-^{a-i}\Bigg) \in \Vv_C(D_Q)
\]
 for homogeneous  $A\in V$ of degree $a$. Checking this amounts to showing that
\[(-1)^{a+1} \sum_{j= 0}^{a+1} \dfrac{1}{j!}L_1^j(L_{-1}(A)) \otimes s_-^{a+1-j}+ a \, (-1)^a    \sum_{i= 0}^a \dfrac{1}{i!}L_1^i(A) \otimes s_-^{a-i}
\] 
is equal to 
\[(-1)^{a+1} \sum_{i= 0}^{a} \dfrac{1}{i!} L_{-1}(L_1^i(A)) \otimes s_-^{a-i+1} + (-1)^{a+1} \sum_{i= 0}^{a} \dfrac{a-i}{i!}L_1^i(A) \otimes s_-^{a-i}. \]
Here the sums are truncated, since the terms with larger values of $i$ and $j$ vanish for degree reasons.
Comparing powers of $s_-$, we need to verify  for every $k \in \{0,\dots, a\}$  that
\begin{equation} \label{eq:LHSconn}\dfrac{(-1)^{a+1}}{(k +1)!} L^{k+1}_{1}L_{-1}(A) + \dfrac{ a(-1)^a}{k!} L^k_{1}(A)
\end{equation} is equal to
\begin{equation} \label{eq:RHSconn}\dfrac{(-1)^{a+1}}{(k+1)!} L_{-1}L^{k+1}_{1}(A) + \dfrac{(a-k)(-1)^{a+1}}{k!} L^k_{1}(A)
\end{equation} 
in $V$. Repeatedly applying the commutator  $[L_1, L_{-1}]=2L_0$,  one has
\[ L_1^{k+1}L_{-1}(A)=L_{-1}L_1^{k+1}(A)+(2a-k)(k+1)\,L_1^{k}(A) ,\]
which indeed implies that \eqref{eq:LHSconn} and \eqref{eq:RHSconn} are  equal in $V$. It follows that \eqref{eq:nablaDQ} factors through $\Vv_C(D_Q)\otimes \left(\frac{ds_+}{s_+}, - \frac{ds_-}{s_-}\right)$, hence the statement.
\end{proof}

\subsubsection{The connection: summary} In conclusion, the connection is given~by:

\begin{defn} \label{def:nabla} 
Let $C$ be a curve which has at worst simple nodal singularities. We define the logarithmic connection $\nabla \colon \Vv_C \to \Vv_C \otimes \omega_C$ as the map induced by the following endomorphisms on formal neighborhoods of points of $C$:
for a smooth point $P \in C$ with formal coordinate $t$ at $P$, the endomorphism of $\Vv_C(D_P)$ given by
$L_{-1} \otimes \text{id}_{\CC\llbracket t \rrbracket} + \text{id}_V\otimes \partial_t$,
and for a node $Q \in C$ locally given by $s_+ s_-=0$, the endomorphism of $\Vv_C(D_Q)$ given by
\begin{equation} 
\label{eq:nablaendonod} \left(L_{-1} \otimes s_+ \text{id}_{\CC\llbracket s_+ \rrbracket} + \text{id}_V \otimes s_+ \partial_{s_+}, -L_{-1} \otimes s_- \text{id}_{\CC\llbracket s_- \rrbracket} - \text{id}_V \otimes s_-
 \partial_{s_-} \right).
\end{equation}
\end{defn}

\subsection{The coordinate-free Lie algebra ancillary to $V$} 
\label{coordfreeLV} As a first application, one obtains a coordinate-free version of the Lie algebra ancillary to~$V$. For a punctured disk $D^\times_{P}$ about a smooth point $P$ on $C$ and a formal coordinate $t$ at $P$, one has
\begin{equation}
\label{eq:simeqt}
\text{H}^0\left(D^\times_{P}, \Vv_C\otimes \omega_C/\textrm{Im}\nabla \right) \xrightarrow{\simeq_t} \mathfrak{L}_t(V).
\end{equation}
A section of $\Vv_C\otimes \omega_C$ on $D^\times_{P}$ with respect to the $t$-trivialization
\begin{equation*}
B\otimes \sum_{i\geq i_0}a_i t^i dt \quad\in\quad  V\otimes_{\mathbb{C}} \mathbb{C}(\!(t)\!) \otimes_{\mathbb{C}(\!(t)\!)} \mathbb{C}(\!(t)\!)dt \;\simeq_{t}\text{H}^0 \bigl(D^\times_{P},  \Vv_C\otimes \omega_C \bigr) 
\end{equation*}
maps to
\begin{equation*}
\textrm{Res}_{t=0}\,Y[B,t] \sum_{i\geq i_0}a_i t^i dt \quad\in\quad \mathfrak{L}_{t}(V),
\end{equation*}
where $Y[B,t]:=\sum_{i\in \mathbb{Z}}B_{[i]}t^{-i-1}$. Sections in $\textrm{Im}\nabla \subset\Vv_C\otimes \omega_C$ map to zero, and this defines a linear map from sections of $\Vv_C\otimes \omega_C/\textrm{Im}\nabla$ on $D^\times_{P}$ to  $\mathfrak{L}_{t}(V)$. The vector space $\text{H}^0\left(D^\times_{P}, \Vv_C\otimes \omega_C/\textrm{Im}\nabla \right)$ has the structure of a Lie algebra such that  \eqref{eq:simeqt} is an isomorphism of Lie algebras  \cite[\S\S 19.4.14, 6.6.9]{bzf}.


\section{The new chiral Lie algebra $\Ll_{C\setminus P_{\bullet}}(V)$}
\label{Chiral}

\subsection{Definition of the chiral Lie algebra}\label{ChiralDef}
 For $(C,P_{\bullet})$ a stable $n$-pointed curve and $V$ a vertex operator algebra,  set 
\[
\Ll_{C\setminus P_\bullet}(V) := \text{H}^0\left(C\setminus P_\bullet, \frac{\Vv_C\otimes \omega_C}{\textrm{Im}\nabla} \right).
\]
Here $\Vv_C$ and its logarithmic connection $\nabla$ are as in \S \ref{sec:VAstable}.

\subsection{The chiral Lie algebra maps to the Lie algebra ancillary to $V$}
\label{ChiralMaps}
For each~$i$, let $t_i$ be a formal coordinate at~$P_i$, let $D^\times_{P_i}$ be the punctured formal disk about $P_i$ on $C$,  and $\mathfrak{L}_{t_i}(V)$ be  the  Lie algebra ancillary to $V$ (\S\ref{LV}). Consider the linear map obtained as the composition
\begin{equation}
\label{eq:chiraltosingleancillary}
\Ll_{C\setminus P_\bullet}(V) \rightarrow  \text{H}^0\left(D^\times_{P_i}, \Vv_C\otimes \omega_C/\textrm{Im}\nabla \right) 
\xrightarrow{\cong}  \mathfrak{L}_{t_i}(V).
\end{equation} The first map is canonical and obtained by restricting sections. The second map is the isomorphism of Lie algebras  \eqref{eq:simeqt} and depends on the formal coordinates~$t_i$. From \eqref{eq:chiraltosingleancillary}, we construct the linear map
\begin{equation}
\label{phiL}
\varphi_{\Ll}\colon \Ll_{C\setminus P_\bullet}(V) \rightarrow \oplus_{i=1}^n \text{H}^0\left(D^\times_{P_i}, \Vv_C\otimes \omega_C/\textrm{Im}\nabla \right) 
\xrightarrow{\cong} \oplus_{i=1}^n \mathfrak{L}_{t_i}(V).
\end{equation}
After \cite[\S 19.4.14]{bzf}, when $C$ is smooth, the first map of \eqref{eq:chiraltosingleancillary} is a homomorphism of Lie algebras, hence so is $\varphi_{\Ll}$, next denoted simply $\varphi$. The map $\varphi$ thus induces an action of $\Ll_{C\setminus P_\bullet}(V)$ on $\mathfrak{L}(V)^{\oplus n}$-modules. This will be used in \S\ref{Coinvariants}. In Proposition \ref{lem:M3} we show an analogous result for nodal curves.

\subsection{A close look at the chiral Lie algebra for nodal curves}
\label{CloserLook}
Let $(C, P_\bullet)$ be a stable $n$-pointed curve such that $C\setminus P_\bullet$ is affine. Assume for simplicity that $C$ has exactly one simple node, which we denote by $Q$.  Let $\eta\colon\Ct\rightarrow C$ be the normalization of $C$,  let $\Qp$ and $\Qm$ be the two preimages of $Q$, and set $Q_\bullet=(\Qp, \Qm)$. Let $s_+$ and $s_-$ be formal coordinates at $\Qp$ and $\Qm$, respectively, such that locally around $Q$, the curve $C$ is given by the equation $s_+ s_-=0$. The chiral Lie algebra for $\left(\Ct,P_\bullet\sqcup Q_\bullet\right)$ is 
\[
\Ll_{\Ct\setminus P_\bullet\sqcup Q_\bullet}(V) = \text{H}^0\left(\Ct\setminus P_\bullet\sqcup Q_\bullet, \Vv_{\Ct}\otimes\omega_{\Ct}/\textrm{Im}\,\nabla\right),
\]
and consider the linear map given by restriction:
\begin{equation}
\label{chiraltoLQp}
\Ll_{\Ct\setminus P_\bullet\sqcup Q_\bullet}(V) \rightarrow \text{H}^0\left(D^\times_{\Qp}, \Vv_{\Ct}\otimes\omega_{\Ct}/\textrm{Im}\,\nabla\right)\xrightarrow{\simeq_{s_+}} \mathfrak{L}_{\Qp}(V).
\end{equation}
Recall the triangular decomposition of $\mathfrak{L}_{\Qpm}(V)$ from \eqref{eq:triangdecLV}:
\[
\mathfrak{L}_{\Qpm}(V)=\mathfrak{L}_{\Qpm}(V)_{<0} \oplus \mathfrak{L}_{\Qpm}(V)_0 \oplus \mathfrak{L}_{\Qpm}(V)_{>0}.
\]
Let $\sigma_{\Qpm} \in \mathfrak{L}_{\Qpm}(V)$ be the image of $\sigma \in \Ll_{\Ct\setminus P_\bullet\sqcup Q_\bullet}(V)$, and let $\left[\sigma_{\Qpm}\right]_0$ be the image of $\sigma_{\Qpm}$ under the projection $\mathfrak{L}_{\Qpm}(V)\rightarrow \mathfrak{L}_{\Qpm}(V)_0$.

Recall the involution $\vartheta$ of $\mathfrak{L}(V)^f$ in \eqref{eq:iota}.  This restricts to an involution on $\mathfrak{L}(V)_0$ given for homogeneous $A\in V$ of degree $a$ by
\[
\vartheta\left(A_{[a-1]} \right) = (-1)^{a-1} \sum_{i\geq 0} \frac{1}{i!} \left(L_1^i A\right)_{[a - i -1]}.
\]

\begin{proposition}
\label{prop:chiralnodal}
For $C\setminus P_\bullet$ affine, one has
\[
\eta^* \Ll_{C\setminus P_\bullet}(V) =
\left\{ \sigma \in \Ll_{\Ct\setminus P_\bullet\sqcup Q_\bullet}(V) \; \text{\Bigg|} \, 
\begin{array}{rl}
\mbox{(i)}& \sigma_{\Qp}, \sigma_{\Qm}\in \mathfrak{L}(V)_{\leq 0} \\ [0.2cm]
\mbox{(ii)}& \left[\sigma_{\Qm}\right]_0 = \vartheta\left( \left[\sigma_{\Qp}\right]_0\right) 
\end{array}
\right\}.
\]
\end{proposition}

\begin{proof}
Since  $C\setminus P_\bullet$ is affine, one has 
\[
\Ll_{C\setminus P_\bullet}(V)=\text{H}^0\left( C\setminus P_\bullet, \Vv_C\otimes \omega_C \right)/\nabla \text{H}^0\left( C\setminus P_\bullet, \Vv_C \right)
\]
and similarly, since $\widetilde{C} \setminus P_\bullet \sqcup Q_\bullet$ is also affine, one has
\[
\Ll_{\widetilde{C} \setminus P_\bullet \sqcup Q_\bullet}(V)=\text{H}^0\left(\widetilde{C} \setminus P_\bullet \sqcup Q_\bullet, \Vv_{\widetilde{C}} \otimes \omega_{\widetilde{C}} \right)/\nabla \text{H}^0\left( \widetilde{C} \setminus P_\bullet \sqcup Q_\bullet, \Vv_{\widetilde{C}} \right).
\]
To characterize elements in $\Ll_{C\setminus P_\bullet}(V)$, we can first describe their lifts in the vector space $\text{H}^0\left( C\setminus P_\bullet, \Vv_C\otimes \omega_C \right)$, 
and then show that the description descends to the quotient by the image of $\nabla$. 

By definition of $\Vv_C$ over nodal curves via the sheaf $\widetilde{\Vv}$ in \eqref{eq:VprimeCtilde} and by Lemma \ref{lem:omegaotimesm}, we have the inclusion
\begin{equation}
\label{eq:isoVomplusVomnew}
\eta^* \,\text{H}^0\left( C\setminus P_\bullet, \Vv_C\otimes \omega_C \right) \subseteq \text{H}^0\left(\Ct \setminus P_\bullet, \widetilde{\Vv} \otimes \omega_{\Ct}(Q_++Q_-)\right).
\end{equation} 
To see that \eqref{eq:isoVomplusVomnew}  implies \textit{(i)}, consider the composition of linear maps
\begin{equation}
\label{eq:sigmatosigmaQpmnew}
\begin{tikzcd}[column sep=1.5em, row sep=1.3em]
 \text{H}^0\left(\Ct \setminus P_\bullet, \widetilde{\Vv} \otimes \omega_{\Ct}(Q_++Q_-)\right)  \arrow[rightarrow]{d}\arrow[rightarrow, dashed]{r}
 & \underset{k\geq 0}{\bigoplus} V_k\otimes_{\mathbb{C}} s_{\pm}^{k-1}\mathbb{C}\left\llbracket s_{\pm}\right\rrbracket ds_{\pm}\\
\text{H}^0\left(D_{\Qpm}, \widetilde{\Vv} \otimes \omega_{\Ct}(Q_++Q_-)\right)
\arrow[rightarrow]{ru}[swap]{\simeq_{s_\pm} }	
\end{tikzcd}
\end{equation}
where the left vertical map is the restriction, followed by the $s_{\pm}$-trivialization. 
By \S \ref{coordfreeLV}, the projection $V\otimes \mathbb{C}(\!(s_\pm )\!)ds_{\pm}\rightarrow\mathfrak{L}_{\Qpm}(V)$~is given by
\begin{equation*} 
B\otimes \mu \mapsto \textrm{Res}_{s_\pm=0} \,Y\left[B,s_\pm\right]\mu \quad \in\quad \mathfrak{L}_{\Qpm}(V).
\end{equation*}
It follows that the image of 
\begin{equation}
\label{eq:projtoLqpmV}
\underset{k\geq 0}{\bigoplus} V_k\otimes_{\mathbb{C}} s_{\pm}^{k-1}\mathbb{C}\left\llbracket s_{\pm}\right\rrbracket ds_{\pm}
\rightarrow \mathfrak{L}_{\Qpm}(V)
\end{equation}
 lies in $\mathfrak{L}_{\Qpm}(V)_{\leq 0}$. 
 Composing \eqref{eq:sigmatosigmaQpmnew} with \eqref{eq:projtoLqpmV}, we  deduce that for $\sigma$ in \eqref{eq:isoVomplusVomnew}, its image $\sigma_{Q_\pm}$ in $\mathfrak{L}_{\Qpm}(V)$ lies in $\mathfrak{L}_{\Qpm}(V)_{\leq 0} \cong\mathfrak{L}(V)_{\leq 0}$, hence \textit{(i)}.

The assertion \textit{(ii)} follows from the gluing isomorphisms that define $\Vv_C$ and $\omega_C$ as subsheaves of $\eta_*\widetilde{\Vv}$ and $\eta_*\,\omega_{\Ct}(Q_++Q_-)$, respectively, using diagrams as in \eqref{eq:diagramVC}.  Given $\sigma$ in \eqref{eq:isoVomplusVomnew}, denote by
$\left[\sigma_{\Qpm}\right]_0$ its image via the composition of \eqref{eq:sigmatosigmaQpmnew} with the projection 
\begin{equation*}
\bigoplus_{k\geq 0} V_k\otimes_{\mathbb{C}} s_{\pm}^{k-1}\mathbb{C}\left\llbracket s_{\pm}\right\rrbracket ds_{\pm}
 \twoheadrightarrow \bigoplus_{k\geq 0}\, V_k\otimes_{\mathbb{C}} s_{\pm}^{k-1}ds_{\pm}.
\end{equation*}
In view of a diagram as in \eqref{eq:diagramVC},   $\left[\sigma_{\Qpm}\right]_0$ is the restriction of $\sigma$ at the fibers at $\Qpm$. For $\sigma$ to correspond to a section of $\eta^*\,\mathrm{H}^0(C \setminus P_\bullet, \Vv_C \otimes \omega_C)$, the elements $\left[\sigma_{Q_+}\right]_0$ and $\left[\sigma_{Q_-}\right]_0$ need to satisfy an identity coming from the gluing isomorphism between the fibers of $\widetilde{\Vv} \otimes \omega_{\Ct}(Q_++Q_-)$ at~$\Qpm$. 
For homogeneous $A\in V$ of degree $a$, the gluing isomorphism on fibers of $\eta_*\widetilde{\Vv}$ at~$\Qpm$ defining $\Vv_C$ is given by 
\[
A\otimes s_+^a \mapsto (-1)^{a}\sum_{i \geq 0} \dfrac{1}{i!} L_1^i(A) \otimes s_-^{a-i}.
\]
 The gluing on fibers of $\eta_*\,\omega_{\Ct}(Q_++Q_-)$ at~$\Qpm$ defining $\omega_C$ is given by $s_+^{-1}ds_+\mapsto -s_-^{-1}ds_+$.
Combining the two, the induced gluing on fibers of $\widetilde{\Vv} \otimes \omega_{\Ct}(Q_++Q_-)$ at~$\Qpm$ is given by
\begin{equation}
\label{eq:gluingchirallift}
A \otimes s_+^{a-1}ds_+ \mapsto
 (-1)^{a-1}\sum_{i\geq 0} \dfrac{1}{i!} L_1^i(A) \otimes s_-^{a-i-1} ds_-.
\end{equation}
After mapping to $\mathfrak{L}_{\Qpm}(V)_0$ via the restriction of \eqref{eq:projtoLqpmV}, 
the gluing \eqref{eq:gluingchirallift} implies that the images of the elements $\left[\sigma_{\Qpm}\right]_0$ in $\mathfrak{L}_{\Qpm}(V)_0$, still denoted $\left[\sigma_{\Qpm}\right]_0$, satisfy \textit{(ii)} by definition of $\vartheta$.

To show that the conditions are $\nabla$-equivariant,  it is enough to check on a neighborhood of $Q$. 
 Since $\nabla$ on $\Vv_C$ is constructed from $\nabla$ on $\Vv_{\widetilde{C}}$, one  verifies that $\eta^*\nabla \Vv_C(D_Q) \subset \nabla \Vv_{\widetilde{C}}(D_{Q_+}^\times \sqcup D_{Q_-}^\times)$, hence the statement.
\end{proof}

We  combine \S \ref{ChiralMaps} with Proposition \ref{prop:chiralnodal} to show the following:

\begin{proposition}
\label{lem:M3}
The normalization map $\eta\colon \Ct \rightarrow C$ identifies $\Ll_{C\setminus P_\bullet}(V)$ with a Lie subalgebra of $\Ll_{\Ct \setminus P_\bullet \sqcup Q_\bullet}(V)$. Moreover, this induces an action of $\Ll_{C\setminus P_\bullet}(V)$ on $\mathfrak{L}(V)^{\oplus n}$-modules.
\end{proposition}

\begin{proof}
For simplicity, we may assume that $C$ has a single node $Q$. 
By  Proposition \ref{prop:chiralnodal}, $\eta^*\Ll_{C\setminus P_\bullet}(V)$ can be identified with the subspace of sections in $\Ll_{\Ct \setminus P_\bullet \sqcup Q_\bullet}(V)$ whose restrictions to $Q_\pm$  is given by the subspace of
$\mathfrak{L}_{Q_+}(V) \oplus \mathfrak{L}_{Q_-}(V)\simeq_{s_\pm} \mathfrak{L}(V)^{\oplus 2}$
generated by \mbox{$\mathfrak{L}(V)_{< 0}^{\oplus 2}$} and the elements of  type $(A_{[a-1]}, \vartheta(A_{[a-1]}))\in \mathfrak{L}(V)_{0}^{\oplus 2}$ for homogeneous $A \in V$ of degree $a$. 

Since $\mathfrak{L}(V)_{<0}$ and $\mathfrak{L}(V)_{0}$ are Lie subalgebras of $\mathfrak{L}(V)$ and $\vartheta$ is a Lie algebra morphism, it follows that 
 $\eta^*\Ll_{C\setminus P_\bullet}(V)$ is a Lie subalgebra of $\Ll_{\Ct \setminus P_\bullet \sqcup Q_\bullet}(V)$.

To conclude, the analogue of the morphism \eqref{phiL} for the nodal $C\setminus P_\bullet$ is the composition of the Lie algebra morphisms:
\[
\begin{tikzcd}
 \Ll_{\Ct \setminus P_\bullet \sqcup Q_\bullet}(V) \ar{dr}{\widetilde{\varphi}}&\\
\Ll_{C\setminus P_\bullet}(V) \ar{u}{\eta^*} \ar[dashed]{r}[swap]{\varphi} &\oplus_{i=1}^n \mathfrak{L}_{t_i}(V),
\end{tikzcd}
\] 
where  $\eta ^*$ is  described in Proposition \ref{prop:chiralnodal} and $\widetilde{\varphi}$ is as in \eqref{phiL}.
\end{proof}

\subsection{A consequence of Riemann-Roch for chiral Lie algebras}
\label{sec:chiralRR}
We give here a statement parallel to \S\ref{sec:RR} for chiral Lie algebras.
Let $C$ be a smooth curve, possibly disconnected, with two non-empty sets of distinct marked points $P_\bullet=(P_1,\dots,P_n)$ and $Q_\bullet=(Q_1,\dots,Q_m)$. For  $i \in \{1, \dots, m\}$, let $s_i$ be a formal coordinate at  $Q_i$. For $\sigma\in \Ll_{C\setminus P_\bullet\sqcup Q_\bullet}(V)$, let $\sigma_{Q_i}$ be the image of $\sigma$ under the map given by restriction
\[
\Ll_{C\setminus P_\bullet\sqcup Q_\bullet}(V) \rightarrow \text{H}^0\left(D^\times_{Q_i}, \Vv_{C}\otimes\omega_{C}/\textrm{Im}\,\nabla\right)\xrightarrow{\simeq_{s_i}} \mathfrak{L}_{Q_i}(V).
\]
For an integer $N$, consider 
\[
\mathfrak{L}_{Q_i}(V,NQ_i) = V\otimes s_i^N \mathbb{C} \llbracket s_i\rrbracket / \mathrm{Im}\, \partial.
\]
This is a Lie subalgebra of $\mathfrak{L}_{Q_i}(V)$.

\begin{proposition}
\label{prop:chiralRR}
Assume  $C\setminus P_\bullet$ is affine, 
fix   $E\in V$ homogeneous, and integers $d$ and $N$. There exists 
$\sigma\in \Ll_{C\setminus P_\bullet\sqcup Q_\bullet}(V)$ such that: 
\begin{align*}
\sigma_{Q_i} &\equiv E_{[d]}  &&  \in &&  \mathfrak{L}_{Q_i}(V)/\mathfrak{L}_{Q_i}(V,NQ_i) , &&\mbox{for a fixed } i,\\
\sigma_{Q_j} & \equiv 0  && \in &&\mathfrak{L}_{Q_j}(V)/\mathfrak{L}_{Q_j}(V,NQ_j), &&\mbox{for all } j\not= i.
\end{align*}
\end{proposition}

\begin{proof}
Since $C\setminus P_\bullet$ is affine, so is $C\setminus P_\bullet\sqcup Q_\bullet$.
As in Proof of Proposition \ref{prop:chiralnodal}, elements of $\Ll_{C\setminus P_\bullet\sqcup Q_\bullet}(V)$ can be lifted to sections of $\Vv_C\otimes \omega_C$ on $C\setminus P_\bullet\sqcup Q_\bullet$, and thus, via Lemma \ref{lem:small}, described as sections of $\oplus_{k\geq 0}\left(\omega_C^{\otimes 1-k}\right)^{\oplus \dim V_k}$ on $C\setminus P_\bullet\sqcup Q_\bullet$. The statement thus follows from the analogous property of sections of tensor products of $\omega_C$, discussed in \S\ref{sec:RR}.
\end{proof}


\section{Spaces of coinvariants}
\label{Coinvariants}

Given a stable pointed curve $(C,P_\bullet)$ and a vertex operator algebra $V$, we define spaces of coinvariants using representations of the chiral Lie algebra.

\subsection{Representations of the chiral Lie algebra}
The chiral Lie algebra $\Ll_{C\setminus P_\bullet}(V)$ acts on the tensor product $M^{\bullet}:=M^1\otimes \cdots \otimes M^n$ of $V$-modules $M^1,\dots, M^n$.     
For each~$i$, let $t_i$ be a formal coordinate at~$P_i$,  and $\mathfrak{L}_{t_i}(V)$ be the  Lie algebra  ancillary to $V$ (\S\ref{LV}). 
 Each  $\mathfrak{L}_{t_i}(V)$ acts on the $V$-module $M^i$ as in \S\ref{LV},  and the sum $\oplus_{i=1}^n \mathfrak{L}_{t_i}(V)$ acts diagonally on $M^{\bullet}$.
 The map $\eqref{phiL}$ thus induces an action of $\Ll_{C\setminus P_{\bullet}}(V)$  on $M^{\bullet}$  as follows: for $\sigma\in \Ll_{C\setminus P_{\bullet}}(V)$ and $A^i\in M^i$, one has
\[
\sigma  \left(A^1\otimes \cdots \otimes A^n\right) = \mbox{$\sum_{i=1}^n$} A^1\otimes \cdots \otimes \sigma_{P_i} \left(A^i \right) \otimes \cdots \otimes A^n,
\]
where $\sigma_{P_i}$ is the restriction of the section $\sigma$ to the punctured formal disk $D^\times_{P_i}$   about $P_i$ on $C$.

\subsection{Coinvariants}\label{sec:Coinvariants}
When $C\setminus P_\bullet$ is affine, the \textit{space of coinvariants} at $(C,P_{\bullet}, t_\bullet)$ is
\[
\mathbb{V}\left(V;M^{\bullet}\right)_{(C,P_{\bullet}, t_\bullet)} := M^{\bullet}_{\Ll_{C\setminus P_{\bullet}}(V)} = M^{\bullet}\big/\Ll_{C\setminus P_{\bullet}}(V) \cdot M^{\bullet}.
\]
This is the largest quotient of  $M^{\bullet}$ on which $\Ll_{C\setminus P_{\bullet}}(V)$ acts trivially.  In general, when $C\setminus P_\bullet$ is not necessarily affine, the \textit{space of coinvariants} at $(C,P_{\bullet}, t_\bullet)$ is defined as the direct limit
\begin{equation}
\label{eq:coinvdef}
\mathbb{V}\left(V;M^{\bullet}\right)_{(C,P_{\bullet}, t_\bullet)} := \varinjlim_{(Q_\bullet, s_\bullet)} \mathbb{V}\left(V;M^{\bullet}\sqcup(V,\dots,V)\right)_{(C,P_{\bullet}\sqcup Q_\bullet, t_\bullet\sqcup s_\bullet)}
\end{equation}
where $Q_\bullet=(Q_1,\dots, Q_m)$ ranges over the set of smooth points of $C$ such that $P_\bullet \cap Q_\bullet=\emptyset$ and $C\setminus P_\bullet\sqcup Q_\bullet$ is affine, and $s_\bullet=(s_1,\dots,s_m)$, with $s_i$ a formal coordinate at $Q_i$, for each $i$. 
As in the case of affine Lie algebras \cite{fakhr, loowzw},
the above direct limit is well defined thanks to the propagation of vacua theorem, which is discussed in \S \ref{sec:POV}.

The construction of the chiral Lie algebra in \S \ref{Chiral} extends to families of \textit{smooth} pointed curves over an arbitrary smooth base, and one obtains sheaves of coinvariants as follows. Let $(C\rightarrow S, P_\bullet)$ be a family of smooth $n$-pointed curves. In this case, $C\setminus P_\bullet (S)$ is affine over $S$. Let $t_i$ be formal coordinates at $P_i(S)$, for $i=1,\dots,n$. Equivalently, fix a formally unramified thickening $S\times \mathrm{Spf}(\mathbb{C}\llbracket t_i\rrbracket)\rightarrow C$ of the section $P_i$, for each $i$. One then obtains  sheaves of Lie algebras $\Ll_{{C}\setminus P_\bullet}(V)$ and of coinvariants 
\[
\mathbb{V}(V; M^{\bullet})_{\left({C}/S, P_\bullet, t_\bullet\right)}:=
\left(M^\bullet \otimes \mathscr{O}_S\right)_{\Ll_{{C}\setminus P_\bullet}(V)}
\]
over $S$, for given $V$-modules $M^1,\dots,M^n$. 

We will extend  sheaves of coinvariants over families of \textit{stable} pointed curves over an arbitrary smooth base in \S\ref{sec:sheafofcoinariantsfinal}.

\subsection{Propagation of vacua}\label{sec:POV}
The propagation of vacua theorem, first proved by Tsuchiya, Ueno, and Yamada  for spaces of coinvariants constructed from representations of affine Lie algebras  \cite[Prop 2.2.3, Cor 2.2.4]{tuy},  says that spaces of coinvariants associated to a stable $n$-pointed curve with coordinates remain invariant when adding a new marked point and the trivial module.  

The result was established in the generality we need here by Codogni \cite[Thm 3.6]{codogni}  (see also \cite[Thm 6.2]{dgt}). 
Other special cases were previously treated in the literature, including  the case of coinvariants defined by quasi-primary generated vertex operator algebras $V$ for which $V_0\cong \mathbb{C}$  either at a fixed smooth pointed curve with coordinates \cite{ZhuGlobal, an1}, or on stable pointed rational curves \cite{nt}.    Moreover, propagation of vacua was proved for conformal blocks defined at a fixed smooth curve  in \cite[\S 10.3.1]{bzf}.

To state the theorem, we need the following setup. Let $\left({C}\rightarrow S, P_\bullet, t_\bullet\right)$ be a family of stable $n$-pointed curves with coordinates.  Let \mbox{$Q\colon S\rightarrow {C}$} be a section such that $Q(S)$ is contained in the smooth locus of ${C}$ and is disjoint from $P_i(S)$, for each $i \in \{1,\ldots, n\}$, and let $r$ be a formal coordinate at $Q(S)$. 

\begin{theorem}[Propagation of Vacua {\cite[Thm 3.6]{codogni}}]
\label{thm:POV}
Let $V$ be a  vertex operator algebra with one-dimensional weight zero space.
Assume that \mbox{${C}\setminus P_\bullet(S)$} is affine over $S$. The linear map 
\[
M^\bullet \rightarrow M^\bullet \otimes V, \qquad u\mapsto u\otimes \bm{1}^V
\]
 induces a canonical $\mathscr{O}_S$-module isomorphism
\[
\mathbb{V}\left(V;M^\bullet\right)_{\left({C}/ S, P_\bullet, t_\bullet\right)} \xrightarrow{\cong}
\mathbb{V}\left(V;M^\bullet\otimes V\right)_{\left({C}/ S, P_\bullet \sqcup Q, t_\bullet\sqcup r\right)}.
\]
Varying $(Q,r)$, the induced isomorphisms are compatible. Moreover, as $\bm{1}^V$ is fixed by the action of $\mathrm{Aut}\,\mathcal{O}$, the isomorphism is equivariant with respect to change of coordinates.
\end{theorem}

The proof requires two main ingredients: (1) the axiom on the vacuum vector;  and (2) the existence of a PBW basis for $V$  \cite{GN}.


\section{Finite-dimensionality of  coinvariants}
\label{sec:FINITESection}

Using coinvariants by the action of Zhu's Lie algebra (\S\ref{sec:CoinvariantsHistory}), Abe and Nagatomo show that spaces of coinvariants at smooth pointed curves  of arbitrary genus are finite-dimensional \cite{an1}. 

 We show here that the result of \cite{an1} extends to coinvariants by the action of the chiral Lie algebra. 
Moreover, we further extend the result in \cite{an1} by allowing the following twist of the chiral Lie algebra: given a smooth $n$-pointed curve
$(C, P_\bullet)$, and an effective divisor $D=\sum_{i=1}^m n_i Q_i$ on~$C$ not supported at $P_{\bullet}$, consider
\begin{equation} \label{eq:Lcp(D)}
\Ll_{C\setminus P_\bullet}(V,D):= \text{H}^0\left(C\setminus P_\bullet, \Vv_C\otimes \omega_C(-D)/ \textrm{Im}\nabla\right),
\end{equation} 
where $\textrm{Im}\nabla$ denotes the intersection of $\nabla(\Vv_C)$ and $\Vv_C\otimes \omega_C(-D)$.
This is the space of sections in $\Ll_{C\setminus P_\bullet}(V)$ vanishing with order at least $n_i$ at $Q_i$, for each $i$,
and gives a Lie subalgebra of $\Ll_{C\setminus P_\bullet}(V)$.

\subsection{$C_2$-cofiniteness}
\label{sec:C2}
For $k\geq 2$ and  a $V$-module $M$ (e.g., $M=V$), set:
\[
C_k(M):=\mathrm{span}_{\mathbb{C}}\left\{A_{(-k)}m \,:\, A \in V,\, m \in M\right\}.
\]
One says that $M$ is \textit{$C_k$-cofinite} if $\dim_{\mathbb{C}} M/C_k(M)<\infty$.  
For  a  $C_2$-cofinite vertex operator algebra $V$ with one-dimensional weight zero space,
any finitely generated $V$-module is $C_k$-cofinite, for $k\geq 2$ \cite{BuhlSpanning}. As explained in \cite{ArakawaC2Lisse}, the $C_2$-cofiniteness has a natural geometric interpretation which generalizes the concept of lisse modules introduced in \cite{bfm} for the Virasoro algebra.

\begin{proposition}
\label{prop:FiniteTwisted}
Let $V$ be a  $C_2$-cofinite vertex operator algebra with one-dimensional weight zero space.
Let $C$ be a smooth curve with distinct  points $P_1,\dots, P_n$, and $D$ an effective divisor on $C$ not supported at $P_\bullet$.  Fix formal coordinates $t_i$ at $P_i$, for each $i$.
For finitely generated $V$-modules $M^1,\ldots, M^n$, the  coinvariants $M^{\bullet}_{\Ll_{C\setminus P_ \bullet}(V,D)}$ are finite-dimensional.
\end{proposition}

\begin{proof}
 Recall the map from \eqref{eq:chiraltosingleancillary} obtained by fixing the formal coordinates $t_i$ at $P_i$, for each $i$: 
 $\Ll_{C\setminus P_ \bullet}(V,D)\rightarrow \mathfrak{L}_{P_i}(V)$, $\sigma\mapsto \sigma_{P_i}$. For $k\in \mathbb{N}$, define
\[
\mathcal{F}_k\,\Ll_{C\setminus P_ \bullet}(V,D) := \left\{\sigma \in \Ll_{C\setminus P_ \bullet}(V,D) \, | \, \deg \sigma_{P_i}\leq k, \, \mbox{for all $i$}  \right\},
\]
which gives $\Ll_{C\setminus P_ \bullet}(V,D)$ the structure of a filtered Lie algebra. Let
\[
\mathcal{F}_{k}M^{\bullet}=\bigoplus_{0\leq d \leq k}M^{\bullet}_d, \quad\mbox{where}\quad
M^{\bullet}_d:=\sum_{d_1+\cdots+d_n=d} M^1_{d_1} \otimes \cdots \otimes M^n_{d_n}.
 \]
Since $\mathcal{F}_{k}\,\Ll_{C\setminus P_ \bullet}(V,D)\cdot \mathcal{F}_{l}M^{\bullet} \subset \mathcal{F}_{k+l}M^{\bullet}$,  the $\Ll_{C\setminus P_ \bullet}(V,D)$-module $M^\bullet$ is a filtered $\Ll_{C\setminus P_ \bullet}(V,D)$-module.
One has an induced filtration on $M^\bullet_{\Ll_{C\setminus P_ \bullet}(V,D)}$:
\[
\mathcal{F}_{k} \left( M^\bullet_{\Ll_{C\setminus P_ \bullet}(V,D)}\right) :=
\left( \mathcal{F}_{k} M^\bullet + \Ll_{C\setminus P_ \bullet}(V,D)\cdot  M^\bullet\right)\big/\Ll_{C\setminus P_ \bullet}(V,D) \cdot M^\bullet.
\]

\noindent \textit{Step 1.} Let $U$ be a finite-dimensional subspace of $V$ such that $V=U\oplus C_2(V)$. 
Contrary to \cite{an1}, elements of $U$ are not required to be quasi-primary here.
Let $d_U$ be the maximum of the degree of the homogeneous elements in $U$. 
Similar to \cite[Lemma 4.1]{an1}, by an application of the Riemann-Roch and the Weierstrass gap theorem, there exists an integer $N$ such that
\[
\text{H}^0\left(C,\omega_C^{\otimes 1-k}\left(lP_i-D\right)\right)\neq 0, \quad\mbox{for all } k\leq d_U, \, l\geq N, \,i\in\{1,\dots,n\}.
\]

\noindent \textit{Step 2.} For a $V$-module $M$ and with $N$ as in Step 1, define the subset
\[
C_N(U,M)=\mathrm{span}_{\mathbb{C}}\left\{A_{(-l)}m \,:\, A \in U, \, m \in M, \, l\geq N\right\}.
\]
We claim that for each $i$ the set $M^1\otimes \cdots \otimes C_{N}\left(U,M^i\right) \otimes \cdots \otimes M^{n}$ is in the kernel of the canonical surjective linear map 
\[
M^{\bullet}
\overset{\pi}{\twoheadrightarrow} \mathrm{gr}_{\bullet}\left(M^\bullet_{\Ll_{C\setminus P_ \bullet}(V,D)} \right)
:= \underset{k\geq 0}{\oplus} \mathcal{F}_{k} \left( M^\bullet_{\Ll_{C\setminus P_ \bullet}(V,D)}\right)\big/ \mathcal{F}_{k-1} \left( M^\bullet_{\Ll_{C\setminus P_ \bullet}(V,D)}\right).
\]
For this, it is enough to show that
$\pi\left(m_1\otimes \cdots \otimes A_{(-l)}m_i\otimes \cdots \otimes m_n \right)=0,$
for homogeneous $A\in U$ of degree $a$, $m_i\in M^i_{d_i}$, and $l\geq N$. Note that $C\setminus P_i$ is affine for all $i$. As in the proof of Proposition \ref{prop:chiralnodal}, elements of $\Ll_{C\setminus P_i}(V,D)\subset \Ll_{C\setminus P_\bullet}(V)$ can be lifted to sections of $\Vv_C\otimes \omega_C(-D)$ on $C\setminus P_i$. By Lemmas \ref{lem:small} and \ref{lemma:gr}, the vector space of such sections is  isomorphic to the space of sections of 
\begin{equation}
\label{eq:Vkomega}
\oplus_{k\geq 0} V_k\otimes \omega_C^{\otimes 1-k}(-D)
\end{equation}
on $C\setminus P_i$. 
Following Step 1, there exists a section $\sigma=A\otimes \mu$ of \eqref{eq:Vkomega} on $C\setminus P_i$ such that its image via the map $\Ll_{C\setminus P_ i}(V,D)\rightarrow \mathfrak{L}_{P_i}(V)$ from \eqref{eq:chiraltosingleancillary} is
\[
\sigma_{P_i}=  A_{[-l]} + \sum_{j>-l} c_j A_{[j]}, \quad \mbox{for some $c_j\in\mathbb{C}$.}
\]
One has $A_{[-l]}\cdot M^i_{d_i}\subset M^i_{d_i+a+l-1}$ and $A_{[j]}\cdot M^i_{d_i}\subset M^i_{d_i+a+l-2}$ for $j>-l$.
Moreover, since $\mu$ is holomorphic at a point $P_r\neq P_i$, one has 
$\sigma_{P_r}=\sum_{s\geq 0}a_s A_{[s]}$, for some $a_s\in\mathbb{C}$.
 It follows that $\sigma_{P_r}\cdot M^r_{d_r}\subset M^r_{d_r+a-1}$. 
From the identity
\[
\sigma \left(m_1\otimes \cdots\otimes m_n\right) =  \mbox{$\sum_{r=1}^n$}\,  m_1\otimes \cdots\otimes \sigma_{P_r} (m_r) \otimes \cdots\otimes m_n,
\]
one has
\[
m_1\otimes \cdots \otimes A_{(-l)}m_i\otimes \cdots \otimes m_n \in \mathcal{F}_{\sum_r d_r +a+l-2}M^\bullet+ \Ll_{C\setminus P_ \bullet}(V,D)\cdot  M^\bullet.
\]
Since the element on the left-hand side is in $\mathcal{F}_{\sum_r d_r +a+l-1}M^\bullet$, it follows that it maps to zero via $\pi$. The claim follows.

\noindent \textit{Step 3.} After Step 2, the map $\pi$ factors through
\begin{equation}
\label{eq:pithrough}
M^1/C_N\left(U,M^1\right) \otimes \cdots \otimes M^n/C_N\left(U,M^n\right)  \overset{\pi}{\twoheadrightarrow} \mathrm{gr}_{\bullet}\left(M^\bullet_{\Ll_{C\setminus P_ \bullet}(V,D)} \right).
\end{equation}
By \cite[Prop.~4.5]{an1},  there is a positive integer $k$ such that $C_k\left(M^i\right) \subset C_N\left(U,M^i\right)$ for all $i$.  In particular,
$\dim M^i/C_N\left(U,M^i\right)<\dim M^i/C_k\left(M^i\right)$.
These are finite as the $M^i$ are all $C_k$-cofinite by \cite{BuhlSpanning}.
It follows that the source in \eqref{eq:pithrough} is finite-dimensional, hence so is the target.
This implies that the coinvariants are finite-dimensional as well.
\end{proof} 

The proof of Proposition \ref{prop:FiniteTwisted} extends over families of smooth curves \mbox{${C}\rightarrow S$} with $n$ disjoint sections $P_1, \dots, P_n$, and for each $i$,  a formal coordinate $t_i$ at $P_i(S)$.  Hence, we conclude:

\begin{corollary}
\label{cor:CoherenceOnMgn} 
Let $V$ be a   $C_2$-cofinite vertex operator algebra with one-dimensional weight zero space. For any collection of finitely generated $V$-modules $M^1, \dots, M^n$, the sheaf of coinvariants $\mathbb{V}(V; M^{\bullet})_{\left({C}/S, P_\bullet, t_\bullet\right)}$ is a coherent $\mathscr{O}_{S}$-module.
\end{corollary}


\section{The modules $Z$ and $\overline{Z}$}

In service of the proof of the factorization theorem,  we consider the modules ${Z}$ and $\overline{Z}$ in \S\ref{sec:defZZ}, and coinvariants constructed from them in \S\ref{sec:coinvZhat}.

\subsection{Definitions and properties}
\label{sec:defZZ}
Let $V$ be a vertex operator algebra. Recall the associative algebra $\mathscr{U}\!(V)$ from \S\ref{sec:UV}. Consider the $\mathscr{U}\!(V)^{\otimes 2}$-module
\[
Z:= \left( \mathrm{Ind}^{\mathscr{U}\!(V)}_{\mathscr{U}\!(V)_{\leq 0}} A(V) \right)^{\otimes 2} = \left(\mathscr{U}\!(V) \otimes_{\mathscr{U}\!(V)_{\leq 0}} A(V) \right)^{\otimes 2}
\]
where $\mathscr{U}\!(V)_{<0}$ acts trivially on $A(V)$, and the action of $\mathscr{U}\!(V)_{0}$ on $A(V)$ is induced from the projection $\mathscr{U}\!(V)_{0} \rightarrow A(V)$. With the notation from \S \ref{sec:irrVAVmod}, one has that $Z=M(A(V))^{\otimes 2}$, where $M(A(V))$ is the generalized Verma $\mathscr{U}\!(V)$-module induced from the natural representation $A(V)$ of $A(V)$. 

We will also consider a quotient $\overline{Z}$ of $Z$ defined as follows. Let $\mathscr{P}$ be the subalgebra of $\mathscr{U}\!(V)^{\otimes 2}$ generated by $\mathscr{U}\!(V)\otimes_{\mathbb{C}} \mathscr{U}\!(V)_{< 0}$,  \mbox{$\mathscr{U}\!(V)_{< 0}\otimes_{\mathbb{C}} \mathscr{U}\!(V)$}, and $\mathscr{U}\!(V)_0\otimes_{\mathbb{C}} \mathscr{U}\!(V)_0$.
Consider the $\mathscr{U}\!(V)^{\otimes 2}$-module
\[
\overline{Z}:= \mathrm{Ind}^{\mathscr{U}\!(V)^{\otimes 2}}_{\mathscr{P}} A(V) = \left(\mathscr{U}\!(V)^{\otimes 2}\right) \otimes_{\mathscr{P}} A(V)
\]
where $\mathscr{U}\!(V)\otimes_{\mathbb{C}} \mathscr{U}\!(V)_{<0}$ and $\mathscr{U}\!(V)_{<0}\otimes_{\mathbb{C}} \mathscr{U}\!(V)$ act trivially on $A(V)$, 
and the action of $\mathscr{U}\!(V)_0\otimes_{\mathbb{C}} \mathscr{U}\!(V)_0$ on $A(V)$ is induced  
via the natural surjection $\mathscr{U}\!(V)_0\otimes_{\mathbb{C}} \mathscr{U}\!(V)_0\rightarrow A(V)\otimes_\mathbb{C}A(V)$
from the action of $A(V)\otimes_\mathbb{C}A(V)$ given~by
\[
(a\otimes b)(c)= a\cdot c \cdot (-\vartheta(b)),\quad \mbox{for $a\otimes b\in A(V)\otimes A(V)$, $c\in A(V)$.}
\]

\begin{lemma}\label{lemma:ZZ}
Let $V$ be a rational vertex operator algebra. 
One has $\mathscr{U}\!(V)^{\otimes 2}$-module isomorphisms
\[
Z \cong \bigoplus_{W,Y\in \mathscr{W}} \left( W \otimes W_0^\vee\right)\otimes \left( Y\otimes Y_0^\vee\right) \ \text{and } \ \overline{Z} \cong \bigoplus_{W\in \mathscr{W}} W \otimes W', 
\]
with $\mathscr{W}$ the set of representatives of isomorphism classes of simple $V$-modules.
\end{lemma}

\begin{proof}
Since $V$ is rational, the algebra $A(V)$ is semisimple \cite{zhu}.    From Wedderburn's theorem, one has $A(V)=\oplus_{E\in \mathscr{E}} \, E\otimes E^\vee$, where $\mathscr{E}$ is the finite set of representatives of isomorphism classes of simple $A(V)$-modules. Using  the one-to-one correspondence between simple $V$-modules and simple $A(V)$-modules \cite{zhu}, and rationality of $V$ which implies  that  the $V$-module induced from any simple $A(V)$-module is simple,  it follows that each simple $V$-module  is $W=\mathscr{U}\!(V)\otimes_{\mathscr{U}\!(V)_{\leq 0}}E$, for some $E\in \mathscr{E}$. Moreover, there exists a canonical $V$-module isomorphism 
$\mathscr{U}\!(V)\otimes_{\mathscr{U}\!(V)_{\leq 0}}E^\vee\cong (\mathscr{U}\!(V)\otimes_{\mathscr{U}\!(V)_{\leq 0}}E)'$, for $E\in \mathscr{E}$ \cite[Prop.~7.2.1]{nt}.
The statement follows by linearity.
\end{proof}

\subsection{Replacing coinvariants with $Z$}
\label{sec:coinvZhat}
The main result of this section is Proposition \ref{prop:coinvZhat},
which generalizes \cite[Prop.~7.2.2, Cor.~8.6.2]{nt}  to curves of arbitrary genus. 
The statement describes coinvariants of the action of  a Lie subalgebra $\mathcal{L}_{C\setminus P_\bullet}\left(V, \{Q_+, Q_-\}\right)$ of the chiral Lie algebra $\Ll_{C\setminus P_ \bullet}(V)$. 

We begin by defining $\mathcal{L}_{C\setminus P_\bullet}\left(V, \{Q_+, Q_-\}\right)$.
For this,  let $C$ be a smooth curve, possibly disconnected, with two nonempty, disjoint sets of distinct marked points $P_\bullet=(P_1,\dots,P_n)$ and $Q_\bullet=(Q_+, Q_-)$. Assume that $C\setminus P_\bullet$ is affine. 
After Lemmas \ref{lem:small} and \ref{lemma:gr}, one has
\[
\text{H}^0\left(C\setminus P_\bullet, \mathcal{V}_C\right)\cong\oplus_{k\geq 0}\, \text{H}^0\left(C\setminus P_\bullet, V_k\otimes_{\mathbb{C}}\omega_C^{\otimes -k}\right).
\]
Fixing an isomorphism, consider the following Lie subalgebra of the chiral Lie algebra $\Ll_{C\setminus P_ \bullet}(V)$:
\begin{equation}\label{eq:modifiedchiral}
\mathcal{L}_{C\setminus P_\bullet}\left(V, \{Q_+, Q_-\}\right):= \dfrac{\bigoplus_{k\geq 0}\, \text{H}^0\left( C\setminus P_\bullet, V_k\otimes_{\mathbb{C}}\omega_C^{\otimes 1-k}(-k Q_+ -k Q_-) \right)}{\nabla \text{H}^0(C \setminus P_\bullet, \mathcal{V}_C)}.
\end{equation}
As in \eqref{eq:Lcp(D)},   $\nabla \text{H}^0(C \setminus P_\bullet, \mathcal{V}_C)$ is the intersection of $\mathrm{Im}\nabla$ with the subspace $\bigoplus_{k\geq 0}\, \text{H}^0\left( C\setminus P_\bullet, V_k\otimes_{\mathbb{C}}\omega_C^{\otimes 1-k}(-k Q_+ -k Q_-) \right)$ of $\text{H}^0(C\setminus P_\bullet, \mathcal{V}_C \otimes \omega_C)$.
 
Select  formal coordinates $t_i$ at  $P_i$ and $s_i$ at $Q_i$. Let
 $\mathfrak{L}_{P_\bullet}(V):=\oplus_{i=1}^n\mathfrak{L}_{P_i}(V)$ and $\mathfrak{L}_{Q_\bullet}(V):=\mathfrak{L}_{Q_+}(V)\oplus \mathfrak{L}_{Q_-}(V)$. There are  
Lie algebra injections
\begin{equation}
\label{eq:resmaps}
\mathcal{L}_{C\setminus P_\bullet}\left(V, \{Q_+, Q_-\}\right)\rightarrow \mathfrak{L}_{P_\bullet}(V) 
\,\, \mbox{and} \,\,  \mathcal{L}_{C\setminus P_\bullet}\left(V, \{Q_+, Q_-\}\right)\rightarrow \mathfrak{L}_{Q_\bullet}(V).
\end{equation}
The image of an element of \eqref{eq:modifiedchiral} in $\mathfrak{L}_{Q_\bullet}(V)\cong \mathfrak{L}(V)^{\oplus 2}$ via the restriction map in $\eqref{eq:resmaps}$ is
\[
\sum_{i\geq k} a_i A_{[i]} \oplus \sum_{j\geq k} b_j A_{[j]} \quad\in\quad \mathfrak{L}(V)_{<0}^{\oplus 2}\subset \mathfrak{L}_{Q_\bullet}(V)
\]
for homogeneous $A\in V$ of degree $k\geq 0$ and coefficients $a_i, b_j\in\mathbb{C}$.

We use here the assumption that $V$ is \textit{$C_1$-cofinite}, i.e., \mbox{$\dim_{\mathbb{C}} V/C_1(V) <\infty$}, where $C_1(V)$ is the subspace of $V$ linearly spanned by $A_{(-1)}B$ for $A,B\in V_{>0}$ and by $L_{-1}V$. If $V$ is $C_1$-cofinite and $V_0\cong \mathbb{C}\bm{1}^V$, then lowest weight  $V$-modules admit spanning sets of PBW-type \cite{KarelLi}. Also, $C_2$-cofinitess implies $C_1$-cofinitess \cite{KarelLi}, hence the assumption here is weaker.

\begin{proposition}\label{prop:coinvZhat}
Consider $(C,P_\bullet\sqcup Q_\bullet, t_\bullet\sqcup s_\bullet)$, i.e., a smooth coordinatized $(n+2)$-pointed curve, possibly disconnected, such that $C\setminus P_\bullet$ is affine. Let $V$ be a  $C_1$-cofinite vertex operator algebra with one-dimensional weight zero space, 
and let $M^1, \ldots, M^n$ be $V$-modules. The map 
\[
M^\bullet \to M^\bullet \otimes_{\mathbb{C}} Z, \qquad w\mapsto w\otimes \bm{1}^{A(V)}\otimes \bm{1}^{A(V)},
\]
where $\bm{1}^{A(V)}\in A(V)$ is the unit, induces an isomorphism of vector spaces
\[
h\colon M^\bullet_{\mathcal{L}_{C\setminus P_\bullet}\left(V, \{Q_+,Q_-\}\right)} \xrightarrow{\cong} \left( M^\bullet \otimes_{\mathbb{C}} Z \right)_{\mathcal{L}_{C\setminus P_\bullet\sqcup Q_\bullet}(V)}.
\] 
\end{proposition}

\begin{proof} 
We proceed in three steps.

\smallskip

\noindent
{\textit{Step 1}}\label{St61}.  We first show that the map $h$ is well-defined. Observe that $Z$ is naturally equipped with a left action of $\mathcal{L}_{C \setminus P_\bullet \sqcup Q_\bullet}(V)$ induced by the Lie algebra homomorphisms $\mathcal{L}_{C \setminus P_\bullet \sqcup Q_\bullet }(V) \to \mathfrak{L}_{Q_\bullet}(V) \to \mathscr{U}\!(V)^{\otimes 2}$. 
Given an element $\sigma\in \mathcal{L}_{C\setminus P_\bullet}\left(V, \{Q_+, Q_-\}\right)\subset \mathcal{L}_{C\setminus P_\bullet\sqcup Q_\bullet}(V)$, let $\sigma_{P_\bullet}$ be the image of $\sigma$ in $\mathfrak{L}_{P_\bullet}(V)$ via the restriction map in \eqref{eq:resmaps}, and similarly let $\sigma_{Q_i}$ be its image in $\mathfrak{L}_{Q_i}(V)$, for $i=1,2$. Since $\sigma_{Q_i}\in \mathfrak{L}(V)_{<0}$, 
the elements $\sigma_{Q_+}\otimes 1$ and $1\otimes \sigma_{Q_-}$ act trivially on $A(V)\otimes A(V)\subset Z$.
This implies
\begin{multline}
\sigma_{P_\bullet}(w)\otimes \bm{1}^{A(V)}\otimes \bm{1}^{A(V)} 
=\\
\sigma\left(w\otimes \bm{1}^{A(V)}\otimes \bm{1}^{A(V)} \right) - w\otimes \sigma_{Q_+}\left(\bm{1}^{A(V)}\right) \otimes \bm{1}^{A(V)} - w\otimes \bm{1}^{A(V)} \otimes \sigma_{Q_-}\left(\bm{1}^{A(V)}\right)\\
=  \sigma\left(w\otimes \bm{1}^{A(V)}\otimes \bm{1}^{A(V)} \right)
\end{multline}
for $w\in M^\bullet$. It follows that the image of any element is independent of 
the equivalence class representative in the quotient, hence the map $h$ between the spaces of coinvariants is well-defined.

\smallskip

\noindent\textit{Step 2.}\label{St62}
 Next, we show that the map $h$ is surjective:
Given $w\otimes z_1\otimes z_2$ in $M^\bullet \otimes Z$, there exists $w'\in M^\bullet$ such that
\[
w\otimes z_1\otimes z_2 \equiv w'\otimes \bm{1}^{A(V)}\otimes \bm{1}^{A(V)} \quad \mbox{mod $\mathcal{L}_{C\setminus P_\bullet\sqcup Q_\bullet}(V) \left(M^\bullet \otimes Z\right)$.}
\]

By linearity, and reordering elements in $\mathscr{U}\!(V)$, we can reduce to the case
\[
z_1\otimes z_2 = D_l \cdots D_1  \bm{1}^{A(V)} \otimes E_m \cdots E_1  \bm{1}^{A(V)},
\]
with each $D_i$ and $E_j$ in $\mathfrak{L}(V)_{\geq 0}$. The surjectivity is clear when $l=m=0$. By induction on $l$ (and similarly on $m$), it is then enough to show that
\[
w\otimes z_1\otimes z_2 \equiv w'\otimes z'_1\otimes z_2 \qquad \mbox{mod $\mathcal{L}_{C\setminus P_\bullet\sqcup Q_\bullet}(V) (M^{\bullet}\otimes Z)$}
\]
for some $w'$ in $M^\bullet$, when
$z_1=D_{[d]}\left( z'_1\right)$ for some homogeneous $D\in V$ and  $D_{[d]}$ in $\mathfrak{L}(V)_{\geq 0}$. 
Each component of the curve $C$ has at least one of the marked points in $P_\bullet$. 
By Proposition~\ref{prop:chiralRR}, there exists 
$\sigma\in \mathcal{L}_{C\setminus P_\bullet\sqcup Q_\bullet}(V)$
such that
\begin{align*}
\sigma_{Q_+} &\equiv D_{[d]}  &  \in &&  \mathfrak{L}_{Q_+}(V)/\mathfrak{L}_{Q_+}(V,N Q_+) , \\
\sigma_{Q_-} & \equiv 0  & \in &&\mathfrak{L}_{Q_-}(V)/\mathfrak{L}_{Q_-}(V,NQ_-),
\end{align*}
for $N\gg0$. It is enough to take $N$ such that 
\[
D_{[i]}(z'_1)\otimes z_2 =  z'_1\otimes D_{[i]}(z_2) =0  \quad \mbox{ in $Z$ for all $i\geq N$.}
\]
Such $N$ exists because $\mathscr{U}\!(V)$ acts smoothly on each factor. This implies
\[
\sigma_{Q_+}(z'_1)\otimes z_2 + z'_1\otimes \sigma_{Q_-}(z_2) = D_{[d]}\cdot z'_1\otimes z_2 = z_1\otimes z_2.
\]
It follows that
\[
w\otimes z_1\otimes z_2 =  \sigma (w\otimes z'_1\otimes z_2) -  \sigma_{P_\bullet} \left( w\right) \otimes z'_1\otimes z_2,
\]
hence 
\[
w\otimes z_1\otimes z_2 \equiv -  \sigma_{P_\bullet} \left( w\right) \otimes z'_1\otimes z_2 \qquad \mbox{mod $\mathcal{L}_{C\setminus P_\bullet\sqcup Q_\bullet}(V) (M^{\bullet}\otimes Z)$}.
\]
Repeating the same argument for $z_2$, the surjectivity of $h$ follows.\\

\smallskip

\noindent\textit{Step 3.} \label{St63} 
It remains to show that $h$ is injective.
Equivalently, we show the surjectivity of the dual map 
\[
h^\vee \colon \left(\left( M^\bullet \otimes_{\mathbb{C}} Z \right)_{\mathcal{L}_{C\setminus P_\bullet\sqcup Q_\bullet}(V)}\right)^\vee \rightarrow 
\left( M^\bullet_{\mathcal{L}_{C\setminus P_\bullet}\left(V, \{Q_+, Q_-\}\right)} \right)^\vee
\] given by 
\[
h^\vee(\Phi)(w)= \Phi\left(w \otimes \bm{1}^{A(V)}\otimes \bm{1}^{A(V)}\right), \qquad \mbox {for $w\in M^\bullet$.}
\]
For this, select an element $\overline{\Phi}$ in the target of $h^\vee$. We  construct an element $\Phi$ in the source of $h^\vee$ such that $h^\vee(\Phi)=\overline{\Phi}$.	

First, we define $\Phi$ as a linear functional on  \mbox{$M^\bullet \otimes_{\mathbb{C}} Z$.}
We use a spanning set for $Z$ obtained as follows.
Since $V$ is $C_1$-cofinite and $V_0\cong \mathbb{C}\bm{1}^V$, lowest weight   $V$-modules admit spanning sets of PBW-type \cite[Cor.~3.12]{KarelLi}. Since $Z$ is the tensor product of two lowest weight   $V$-modules,  we conclude that $Z$ admits a spanning set of PBW-type. Select a spanning set of PBW-type for $Z$, and consider a generating element 
\begin{equation}
\label{eq:z1z2}
z_1\otimes z_2= D^l_{[i_l]} \cdots D^1_{[i_1]} \, \bm{1}^{A(V)} \otimes E^m_{[j_m]} \cdots E^1_{[j_1]} \, \bm{1}^{A(V)} \quad \in Z
\end{equation}
with $D^1,\dots, D^l, E^1, \dots, E^m\in V$, and $i_1,\dots, i_l,j_1,\dots,j_m\in\mathbb{Z}$ such that 
\[
D^l_{[i_l]},\dots, D^2_{[i_2]}, E^m_{[j_m]}, \dots, E^2_{[j_2]}\in \mathfrak{L}(V)_{>0} \quad\mbox{ and }\quad D^1_{[i_1]}, E^1_{[j_1]} \in \mathfrak{L}(V)_{\geq 0}.
\]
Looking more closely at \cite{KarelLi}, one can assume that all $D^1, \dots, D^l, E^1, \dots, E^m$ are in the complement of $C_1(V)$ in $V$.
Note that the elements $D^1_{[i_1]}, E^1_{[j_1]} \in \mathfrak{L}(V)_{0}$ generate the lowest weight space from the vector $\bm{1}^{A(V)} \otimes \bm{1}^{A(V)}$.
We proceed by induction on $l$ and $m$.
We start by defining
\[
\Phi\left(w \otimes \bm{1}^{A(V)}\otimes \bm{1}^{A(V)}\right) := \overline{\Phi}(w).
\]
Next, assume that $\Phi$ has been defined on elements $w\otimes z_1\otimes z_2$ for $w\in M^\bullet$ and $z_1\otimes z_2$ as in \eqref{eq:z1z2} for \textit{fixed} $l$ and $m$.
Consider an element 
\[
w\otimes D^{+}_{[i_{+}]} z_1\otimes z_2 \in M^\bullet \otimes_{\mathbb{C}} Z
\]
for some $w\otimes z_1\otimes z_2 \in M^\bullet \otimes_{\mathbb{C}} Z$, $D^+\in V$, and $i_+\in\mathbb{Z}$.
Choose an integer $N \gg 0$ such that 
\[
D^+_{[i]}\, z_1\otimes z_2=z_1\otimes D^+_{[i]}\, z_2= 0 \quad \mbox{ in $Z$ for all $i\geq N$.}
\]
As in Step 2, such $N$ exists because $\mathscr{U}\!(V)$ acts smoothly on each of the two factors of  $Z$.
By Proposition~\ref{prop:chiralRR}, there exists 
$\sigma\in \mathcal{L}_{C\setminus P_\bullet\sqcup Q_\bullet}(V)$
such that
\begin{equation}\label{eq:condsigma}\begin{aligned} 
\sigma_{Q_+} &\equiv D^+_{[i_+]}  & \quad  \in &\quad &  \mathfrak{L}_{Q_+}(V)/\mathfrak{L}_{Q_+}(V,N Q_+) , \\
\sigma_{Q_-} & \equiv 0  & \quad  \in &\quad &\mathfrak{L}_{Q_-}(V)/\mathfrak{L}_{Q_-}(V,NQ_-).
\end{aligned}
\end{equation}
Define
\[
\Phi\left( w\otimes D^{+}_{[i_{+}]} z_1\otimes z_2 \right):= - \sum_{i=1}^n \Phi \left( w_1\otimes \cdots \otimes \sigma_{P_i}(w_i)\otimes \cdots \otimes w_n \otimes z_1\otimes z_2 \right)
\]
where $w=w_1\otimes \cdots\otimes w_n \in M^\bullet$. 
Note that since the generating elements for $Z$ are expressed in terms of $D^1, \dots, D^l, E^1, \dots, E^m$, which are not in $C_1(V)$, then $\Phi$ automatically respects the relations in $Z$, and hence defines a linear map on $M^\bullet \otimes_{\mathbb{C}} Z$.
Since  $\overline{\Phi}$ vanishes on $\mathcal{L}_{C\setminus P_\bullet}(V,\{Q_+,Q_-\})$, the definition of $\Phi$ here is independent of the choice of $\sigma$ satisfying \eqref{eq:condsigma}. One proceeds in a similar way to define $\Phi$ on elements of type 
\[
w\otimes z_1\otimes E^{+}_{[j_{+}]}  z_2 \in M^\bullet \otimes_{\mathbb{C}} Z
\]
for  $w\otimes z_1\otimes z_2 \in M^\bullet \otimes_{\mathbb{C}} Z$, $E^+\in V$, and $j_+\in\mathbb{Z}$. 
By induction, this defines $\Phi$ on a spanning set of PBW-type for $M^\bullet \otimes_{\mathbb{C}} Z$, hence by linearity, on $M^\bullet \otimes_{\mathbb{C}} Z$.
A direct calculation shows that
\begin{multline*}
\Phi\left( w\otimes D^{a}_{[i_{a}]} D^{b}_{[i_{b}]} \, z_1\otimes z_2 \right) 
- \Phi\left( w\otimes  D^{b}_{[i_{b}]} D^{a}_{[i_{a}]} \, z_1\otimes z_2 \right) \\
= \Phi\left( w\otimes \left[D^{a}_{[i_{a}]}, D^{b}_{[i_{b}]} \right] \, z_1\otimes z_2 \right) 
\end{multline*}
and similarly
\begin{multline*}
\Phi\left( w\otimes  z_1\otimes E^{a}_{[j_{a}]} E^{b}_{[j_{b}]}\, z_2 \right) 
- \Phi\left( w\otimes  z_1\otimes E^{b}_{[j_{b}]} E^{a}_{[j_{a}]}\, z_2 \right)  \\
= \Phi\left( w\otimes  z_1\otimes \left[E^{a}_{[j_{a}]}, E^{b}_{[j_{b}]} \right] \, z_2 \right), 
\end{multline*}
hence $\Phi$ is independent of the choice of the spanning set.

\smallskip

Finally, we check that $\Phi$  vanishes on $\mathcal{L}_{C\setminus P_\bullet\sqcup Q_\bullet}(V) \left(M^\bullet \otimes_{\mathbb{C}} Z\right)$. 
  Fix $\tau$ in $\mathcal{L}_{C\setminus P_\bullet\sqcup Q_\bullet}(V)$. 
By definition of $\Phi$, one has
\[
\Phi\left(w\otimes \tau_{Q_+}(z_1)\otimes z_2 \right) + \Phi\left(w\otimes z_1 \otimes \tau_{Q_-}(z_2)\right) = 
-\Phi\left(\sigma_{P_\bullet}(w)\otimes z_1\otimes z_2 \right)
\] for any $\sigma \in \mathcal{L}_{C\setminus P_\bullet \sqcup Q_\bullet}(V)$ satisfying conditions analogous to \eqref{eq:condsigma}, hence we can choose $\sigma = \tau$. It follows that $\Phi$ vanishes on $\tau \left(M^\bullet \otimes_{\mathbb{C}} Z\right)$ for all $\tau \in \mathcal{L}_{C\setminus P_\bullet\sqcup Q_\bullet}(V)$, hence $\Phi$ is in the source of~$h^\vee$. By construction, one has $h^\vee(\Phi)=\overline{\Phi}$, hence $h^\vee$ is surjective.
This ends the proof.
\end{proof}


\section{Proof of the factorization theorem}\label{section:Fact}
Here we prove our main result,  which we state in complete detail below.  For this let us first set some notation. 
Let $(C,P_{\bullet})$ be a stable $n$-pointed curve with exactly one simple node, denoted~$Q$.  Let $\Ct\to C$ be the normalization of~$C$,  let $\Qp$ and $\Qm\in \Ct$ be the two preimages of $Q$, and set $Q_\bullet=(\Qp,\Qm)$. The curve $\Ct$ may not be connected. 
Fix formal coordinates $t_i$ at $P_i$, for each $i=1,\dots,n$, and $s_\pm$ at $Q_\pm$.
Suppose $M^1,\dots, M^n$ are  $V$-modules, set $M^{\bullet}=\otimes_{i=1}^n M^i$, and let
 $\mathscr{W}$ be the set of representatives of isomorphism classes of simple $V$-modules. Consider the map 
\begin{align}
\label{map}
M^{\bullet} \to  \bigoplus_{W \in \mathscr{W}} M^{\bullet} \otimes_{\mathbb{C}} W\otimes_{\mathbb{C}} W',  \qquad
u \mapsto  \mathop{\oplus}_{W\in \mathscr{W}} u \otimes \bm{1}^{W_0}
\end{align}
where
$
\bm{1}^{W_0} = \textrm{id}_{W_0} \in \textrm{End}(W_0)\cong W_0 \otimes W^\vee_0.
$
Here $W_0$ is the degree zero space of the module $W=\bigoplus_{i\geq 0}W_i$. Recall that, by definition, the vector spaces $W_0$ and $W^\vee_0$ are finite-dimensional. 

\begin{theorem}[Factorization theorem] \label{thm:factorization} 
Let $V$ be a  rational, $C_2$-cofinite vertex operator algebra with one-dimensional weight zero space.
The map \eqref{map} gives rise to a canonical isomorphism of vector spaces
\[
\mathbb{V}(V;M^{\bullet})_{(C,P_{\bullet}, t_{\bullet})} \cong \bigoplus_{W \in \mathscr{W}} \mathbb{V}\left(V;M^{\bullet}\otimes_{\mathbb{C}} W\otimes_{\mathbb{C}} W' \right)_{\left(\Ct,P_{\bullet}\sqcup Q_\bullet, t_\bullet\sqcup s_\bullet\right)}.
\]
This isomorphism is equivariant with respect to change of the coordinates  $t_\bullet$.
\end{theorem}

The proof we give here roughly follows the outline of the proof in \cite[\S 8.6]{nt}, with the generalizations to coinvariants defined using the chiral Lie algebra instead of  Zhu's Lie algebra, and for curves of arbitrary genus, made possible by
 Propositions \ref{prop:chiralnodal} and \ref{prop:coinvZhat}.
 
\begin{proof}
By  definition \eqref{eq:coinvdef}, due to propagation of vacua (Theorem \ref{thm:POV}), we can reduce to the case $C\setminus P_\bullet$ affine, after possibly adding more marked points $P_i$ and corresponding modules $V$. Using the formal coordinates $t_i$ at $P_i$, for  $i=1,\dots,n$, and $s_\pm$ at $Q_\pm$, one has Lie algebra homomorphisms
\[
\Ll_{C\setminus P_\bullet}(V) \rightarrow \mathfrak{L}_{P_\bullet}(V) \quad\mbox{and}\quad
\Ll_{\Ct\setminus P_{\bullet}\sqcup Q_\bullet}(V)\rightarrow \mathfrak{L}_{P_\bullet}(V)\oplus \mathfrak{L}_{Q_\bullet}(V).
\]
In the following, we show that \eqref{map} induces a canonical isomorphism
\begin{equation}\label{Eqclaim}
M^\bullet_{\Ll_{C\setminus P_\bullet}(V)} \cong
\bigoplus_{W \in \mathscr{W}} \left(M^{\bullet}\otimes_{\mathbb{C}} W\otimes_{\mathbb{C}} W' \right)_{\Ll_{\Ct\setminus P_{\bullet}\sqcup Q_\bullet}(V)}.
\end{equation}
We will argue that there is a commutative diagram
\[
\begin{tikzcd}
M^\bullet_{\Ll_{\Ct\setminus P_\bullet}\left(V, \{Q_+,Q_-\}\right)} \arrow[rightarrow]{r}{h}[swap]{\cong} \arrow[twoheadrightarrow]{d} & \left( M^\bullet \otimes_{\mathbb{C}} {Z} \right)_{\Ll_{\Ct\setminus P_\bullet\sqcup Q_\bullet}(V)} \arrow[twoheadrightarrow]{d}\\
M^\bullet_{\Ll_{C\setminus P_\bullet}(V)}  \arrow[rightarrow]{r}{f}[swap]{\cong} & \left( M^\bullet \otimes_{\mathbb{C}} \overline{Z} \right)_{\Ll_{\Ct\setminus P_\bullet\sqcup Q_\bullet}(V)}.
\end{tikzcd}
\]
Then after Lemma \ref{lemma:ZZ},  the isomorphism $f$ gives \eqref{Eqclaim}. 

\vskip0.5pc

\noindent \setword{\textit{Step 1}}{St71}.  The top horizontal isomorphism $h$ is given by Proposition \ref{prop:coinvZhat}. 

\vskip0.5pc

\noindent \setword{\textit{Step 2}}{St72}. We argue that there is an inclusion
 \[
 \iota\colon \Ll_{\Ct\setminus P_\bullet}\left(V, \{Q_+,Q_-\}\right)\hookrightarrow \Ll_{C\setminus P_\bullet}(V).
 \] 
 Indeed, by Proposition \ref{prop:chiralnodal} an element of  $\Ll_{C\setminus P_\bullet}(V)$ can be realized as $\sigma$ in $\Ll_{\Ct\setminus P_\bullet\sqcup Q_\bullet}(V)$ such that $\sigma_{Q_\pm}\in \mathfrak{L}(V)_{\leq 0}$ and the restrictions $\left[ \sigma_{Q_\pm}\right]_0$ of $\sigma_{Q_\pm}$ to $\mathfrak{L}(V)_{0}$ satisfy $\left[ \sigma_{Q_-}\right]_0=\vartheta\left[ \sigma_{Q_+}\right]_0$.
In particular,  $\Ll_{\Ct\setminus P_\bullet}\left(V, \{Q_+,Q_-\}\right)$ from \eqref{eq:modifiedchiral} is identified as a Lie subalgebra of
$\Ll_{C\setminus P_\bullet}(V)$ by pullback along the 
normalization, since its elements satisfy \mbox{$\sigma_{Q_\pm}\in \mathfrak{L}(V)_{<0}$}. 

\vskip0.5pc

\noindent \setword{\textit{Step 3}}{St73}. To show that the bottom horizontal map $f$ is an isomorphism, it remains to verify that the kernel of the left vertical map is identified 
with the kernel of the right vertical map by the isomorphism $h$.

 \vskip0.5pc

\noindent \setword{\textit{Step 3(a)}}{St73a}. The kernel of the left vertical map is the space 
\[ 
K:= \Ll_{C\setminus P_\bullet}(V) \left(M^\bullet_{{\Ll_{\Ct\setminus P_\bullet}(V,\{Q_+,Q_-\})}}\right).
\]
Note that for $\sigma$ in ${\Ll_{C\setminus P_\bullet}(V)}$, the formula for the brack\-et \eqref{bracket} gives 
\[
\left[\varphi(\sigma), \varphi\left(\Ll_{\Ct\setminus P_\bullet}(V,\{Q_+,Q_-\})\right)\right]\subset \varphi\left(\Ll_{\Ct\setminus P_\bullet}(V,\{Q_+,Q_-\})\right),
\]
where $\varphi$ is as in \eqref{phiL}. It follows that $\sigma$ acts on the source of $h$.
The left vertical map is thus the quotient by the action of ${\Ll_{C\setminus P_\bullet}(V)}/{\Ll_{\Ct\setminus P_\bullet}(V,\{Q_+,Q_-\})}$.

\vskip0.5pc

\noindent \setword{\textit{Step 3(b)}}{St73b}. 
We conclude the argument by showing that
the right vertical map is the quotient by $h(K)$.
Recall that $h(w)=w\otimes \bm{1}^{A(V)}\otimes \bm{1}^{A(V)}$, for $w\in M^\bullet$.
Hence, $h(K)$ is linearly spanned by
\begin{multline}
\label{eq:spanninghK}
\sigma_{P_\bullet}(w)\otimes \bm{1}^{A(V)}\otimes \bm{1}^{A(V)} \equiv \\
-w\otimes \sigma_{Q_+}\left(\bm{1}^{A(V)}\right)\otimes \bm{1}^{A(V)} 
-w\otimes \bm{1}^{A(V)}\otimes \sigma_{Q_-}\left(\bm{1}^{A(V)}\right) 
\end{multline} 
modulo $\Ll_{\Ct\setminus P_\bullet\sqcup Q_\bullet}(V)\left(M^\bullet \otimes {Z}\right)$,
for $\sigma\in \Ll_{C\setminus P_\bullet}(V)$ and $w\in M^\bullet$.
Since $\mathscr{U}\!(V)_{<0}$ acts trivially on $A(V)$, \eqref{eq:spanninghK} is congruent to
\[
-w\otimes \left[\sigma_{Q_+}\right]_0\left(\bm{1}^{A(V)}\right)\otimes \bm{1}^{A(V)} -w\otimes \bm{1}^{A(V)}\otimes \left[\sigma_{Q_-}\right]_0\left(\bm{1}^{A(V)}\right) 
\]
modulo $\Ll_{\Ct\setminus P_\bullet\sqcup Q_\bullet}(V)\left(M^\bullet \otimes {Z}\right)$. 
From Proposition \ref{prop:chiralnodal}, one has $\left[\sigma_{Q_-}\right]_0 = \vartheta\left(\left[\sigma_{Q_+}\right]_0\right)$.
By linearity, it follows that $h(K)$ is linearly spanned by
\begin{equation}
\label{eq:spanninghK2}
-w\otimes B_{[k-1]}\left(\bm{1}^{A(V)}\right)\otimes \bm{1}^{A(V)} -w\otimes \bm{1}^{A(V)}\otimes \vartheta\left(B_{[k-1]}\right)\left(\bm{1}^{A(V)}\right) 
\end{equation}
modulo $\Ll_{\Ct\setminus P_\bullet\sqcup Q_\bullet}(V)\left(M^\bullet \otimes {Z}\right)$, for $w\in M^\bullet$ and homogeneous $B\in V$ of degree $k$.
Here, $B_{[k-1]}$ and $\vartheta\left(B_{[k-1]}\right)$ are in $\mathfrak{L}(V)_0$, and recall that the action of $\mathfrak{L}(V)_0$ on $A(V)$ is induced from the projection $\mathfrak{L}(V)_0 \rightarrow A(V)$. 

One has $b\left(\bm{1}^{A(V)}\right)= \left(\bm{1}^{A(V)}\right)b$ for all $b\in \mathfrak{L}(V)_0$, where $\left(\bm{1}^{A(V)}\right)b$ denotes the right action of $b\in \mathfrak{L}(V)_0$ on $\bm{1}^{A(V)}$. Using this in \eqref{eq:spanninghK2}, we conclude that $h(K)$ is linearly spanned by
\begin{equation}
\label{eq:spanninghK3}
-w\otimes \left(\bm{1}^{A(V)}\right)B_{[k-1]}\otimes \bm{1}^{A(V)} -w\otimes \bm{1}^{A(V)}\otimes \left(\bm{1}^{A(V)}\right) \vartheta\left(B_{[k-1]}\right)
\end{equation}
modulo $\Ll_{\Ct\setminus P_\bullet\sqcup Q_\bullet}(V)\left(M^\bullet \otimes {Z}\right)$, for $w\in M^\bullet$ and homogeneous $B\in V$ of degree $k$.

From Lemma \ref{lemma:ZZ}, the target of $h$ is isomorphic to
\begin{eqnarray}
\begin{split}
\label{eq:htargetiso}
\left( M^\bullet \otimes {Z} \right)_{\Ll_{\Ct\setminus P_\bullet\sqcup Q_\bullet}(V)} &\cong&
\bigoplus_{W,Y\in\mathscr{W}} \left( M^\bullet \otimes W \otimes W_0^\vee\otimes Y \otimes Y_0^\vee \right)_{\Ll_{\Ct\setminus P_\bullet\sqcup Q_\bullet}(V)}\\
&\cong& \bigoplus_{W,Y\in\mathscr{W}} \left( M^\bullet \otimes W \otimes Y \right)_{\Ll_{\Ct\setminus P_\bullet\sqcup Q_\bullet}(V)} \otimes W_0^\vee \otimes Y_0^\vee.
\end{split}
\end{eqnarray}
The second isomorphism is due to the fact that $\mathfrak{L}_{Q_+}(V)$ acts on the left of $W \otimes W_0^\vee$, and similarly, $\mathfrak{L}_{Q_-}(V)$ acts on the left of $Y \otimes Y_0^\vee$, so that $\Ll_{\Ct\setminus P_\bullet\sqcup Q_\bullet}(V)$ only acts on $M^\bullet \otimes W \otimes Y$.
Expressing \eqref{eq:spanninghK3} via the isomorphisms in \eqref{eq:htargetiso}, $h(K)$ is linearly spanned by
\begin{align*}
&-\left( M^\bullet \otimes W \otimes Y \right)_{\Ll_{\Ct\setminus P_\bullet\sqcup Q_\bullet}(V)} \otimes B_{[k-1]}\left(W_0^\vee\right) \otimes Y_0^\vee\\
&-\left( M^\bullet \otimes W \otimes Y \right)_{\Ll_{\Ct\setminus P_\bullet\sqcup Q_\bullet}(V)} \otimes W_0^\vee \otimes \vartheta\left( B_{[k-1]}\right) \left(Y_0^\vee\right)
\end{align*}
 for $W,Y\in\mathscr{W}$ and homogeneous $B\in V$ of degree $k$.
Here the right action of $B_{[k-1]}$ and $\vartheta\left( B_{[k-1]}\right)$ in $\mathfrak{L}(V)_0$ on $W_0^\vee$ and $Y_0^\vee$ is given by Lemma \ref{lem:contr}, and recall that $\vartheta$ is an involution.
It follows that $h(K)$ in $\left( M^\bullet \otimes {Z} \right)_{\Ll_{\Ct\setminus P_\bullet\sqcup Q_\bullet}(V)}$ is isomorphic to
\[
\bigoplus_{W,Y\in\mathscr{W}}\left( M^\bullet \otimes W \otimes Y \right)_{\Ll_{\Ct\setminus P_\bullet\sqcup Q_\bullet}(V)}  \otimes 
\mathscr{I}\left(W_0^\vee , Y_0^\vee \right)
\]
with $\mathscr{I}\left(W_0^\vee , Y_0^\vee \right)\subset W_0^\vee \otimes Y_0^\vee$ linearly spanned by 
\[
\psi_W \circ\vartheta\left(B_{[k-1]} \right)\otimes \psi_Y +  \psi_W \otimes \psi_Y\circ B_{[k-1]} , 
\]
where $\psi_W \in W_0^\vee$,  $\psi_Y\in Y_0^\vee$ for  $W,Y\in\mathscr{W}$, and  $B\in V_k$, for $k\geq 0$.
One has 
\[
W_0^\vee \otimes Y_0^\vee / \mathscr{I}\left(W_0^\vee , Y_0^\vee \right) =\mathrm{Hom}_{A(V)}\left( W_0, Y_0^\vee \right),
\]
and by Schur's Lemma, this is isomorphic to $\mathbb{C}$ when $Y=W'$ and zero otherwise. 
This and the description of $\overline{Z}$ from Lemma \ref{lemma:ZZ} imply that, after taking the quotient of $\left( M^\bullet \otimes_{\mathbb{C}} {Z} \right)_{\Ll_{\Ct\setminus P_\bullet\sqcup Q_\bullet}(V)}$ by $h(K)$, one obtains $ \left( M^\bullet \otimes_{\mathbb{C}} \overline{Z} \right)_{\Ll_{\Ct\setminus P_\bullet\sqcup Q_\bullet}(V)}$,
hence the statement.
\end{proof}


\section{Sewing and local freeness}\label{section:Sewing}

In this section we prove \hyperlink{thm:VBS}{VB corollary}.
For this, we start with the sewing theorem (Theorem \ref{thm:SewingAndFactorization}), a refined version of the \hyperlink{thm:fact}{factorization theorem}. 
This requires the notion of formal smoothings, reviewed below.

\subsection{Formal smoothings}\label{smoothings} 
For a $\CC$-algebra $R$ with smooth $\textrm{Spec}(R)$, let $\mathscr{C}_0\rightarrow {S_0}:=\textrm{Spec}(R)$ be a flat family of nodal curves with a single simple node defined by a section $Q$ and with $n$ distinct smooth points  given by sections $P_\bullet=(P_1,\ldots,P_n)$. Assume that $\mathscr{C}_0\setminus P_\bullet(S_0)$ is affine over $S_0$. Up to an \'etale base change of $S_0$ of degree two, we can normalize $\mathscr{C}_0$ and obtain a smooth family of $(n+2)$-pointed curves $\widetilde{\mathscr{C}}_0\to S_0$ with sections $P_\bullet \sqcup (Q_{+}, Q_{-})$, where $Q_{\pm}(S_0) \subset \widetilde{\mathscr{C}}_0$ are the preimages of the node in $\mathscr{C}_0/S_0$. Fix formal coordinates $s_+$ and $s_-$ at $Q_+(S_0)$ and $Q_-(S_0)$, respectively. Such coordinates determine a \textit{smoothing} of $(\mathscr{C}_0, P_\bullet)$ over  $S:=\textrm{Spec}(R\llbracket q\rrbracket)$.  That is, a flat family $\mathscr{C}\rightarrow S=\textrm{Spec}(R\llbracket q \rrbracket)$ with sections $P_{\bullet}=(P_1,\ldots, P_n)$ such that  the  general fiber is smooth and the special fiber is identified with  $\mathscr{C}_0\rightarrow {S_0}$.
The family $\widetilde{\mathscr{C}}_0\to S_0$ trivially extends to a family of smooth curves $\widetilde{\mathscr{C}}\rightarrow S$ with $n+2$ sections $P_\bullet$,  $Q_+$, and $Q_-$:
\begin{center}
 \begin{tikzcd}
\widetilde{\mathscr{C}}
 \arrow[rightarrow]{dr} && \mathscr{C} \arrow[rightarrow]{ld}\\
& S=\textrm{Spec}(R\llbracket q \rrbracket) \arrow[rightarrow, bend left=-30, swap]{ur}{P_\bullet} \arrow[rightarrow, bend left=30]{ul}{P_\bullet, Q_+, Q_-} & .
 \end{tikzcd}
\end{center}
The formal coordinate at $Q_{\pm}(S_0)$  extends to a formal coordinate, still denoted
 $s_\pm$,  at $Q_\pm(S)$  --- that is, $s_\pm$ is a generator of the ideal of the completed local $R\llbracket q\rrbracket$-algebra of $\widetilde{\mathscr{C}}$ at $Q_\pm(R)$  --- such that locally around the node, the family $\mathscr{C}$ is defined by $s_+  s_-=q$.
For more details, see \cite[p.~457]{loowzw} and  \cite[pp.~184-5]{ACG2}. 
 We emphasize that the existence of such families holds  over the formal base $S=\textrm{Spec}(R\llbracket q \rrbracket)$, or equivalently, over the complex open unit disk around $S_0$ in the analytic category, but fails over a more general base.
Moreover, one still has that $\widetilde{\mathscr{C}}\setminus P_\bullet(S)$ and ${\mathscr{C}}\setminus P_\bullet(S)$ are affine over $S$.

\subsection{The sheaf of vertex algebras over formal smoothings} 
\label{sec:Vforfamilies}
We define here the sheaf of vertex algebras $\Vv_\Cc$ over a family $\mathscr{C}\rightarrow S$ as in \S\ref{smoothings}. After \S\ref{VSheaf}, it is enough to describe $\Vv_\Cc$ on the formal neighborhood $\mathfrak{D}_Q$ of the node $Q$. The completed local ring $\widehat{\mathscr{O}}_Q$ consists of elements of the form $\sum_{i,j\geq 0} \alpha_{i,j}\,s_+^i s_-^j$ for $\alpha_{i,j}\in R$. After identifying $s_+=(s_+, \frac{q}{s_-})$ and $s_-=(\frac{q}{s_+}, s_-)$,
the ring $\widehat{\mathscr{O}}_Q$ is realized as the subring of 
$R(\!(s_+)\!)\llbracket q \rrbracket \oplus R(\!(s_-)\!) \llbracket q \rrbracket$
consisting of elements of  type
\[
\sum_{i,j\geq 0} \alpha_{i,j}\left(s_+^{i-j}q^j, s_-^{j-i}q^i\right)
\] 
with $\alpha_{i,j} \in R$.  
The space of sections of $\Vv_\Cc$ over $\mathfrak{D}_Q$ is generated  by 
\begin{equation} 
\label{eq:typeA}
\sum_{i,j \geq 0}  \alpha_{i, j} \Big(A \otimes s_+^{a+i-j}q^j, (-1)^a \sum_{k \geq 0} \dfrac{1}{k!}L_1^k(A) \otimes s_-^{a-k+j-i} q^i\Big)
\end{equation}
for homogeneous $A \in V$ of degree $a$ and $\alpha_{i,j} \in R$. For $q=0$, one recovers the description given in \S \ref{VSheaf} for nodal curves.

 Similar to the case of a nodal curve, we can identify a logarithmic connection $\nabla \colon \Vv_\Cc \to \Vv_\Cc \otimes \omega_{\Cc/S}$. The description of $\nabla$ on a smooth open set  follows from the description of the connection $\nabla$ over a smooth curve. Specifically, for every smooth point $P \in \Cc \setminus Q$, there is an open set $S' \subset S$ and an open set $U \subset \Cc \setminus Q$ over $S'$ and containing $P$ such that there exists an \'etale map $U \to \mathbb{A}_{S'}^1$. This implies that $U$ has a global coordinate $t$, and thus  $\Vv|_{U} \simeq_t V \otimes \mathscr{O}_{U}$. On the open set $U$, the connection $\nabla$ is provided by the endomorphism of $\Vv|_{U}$ given by $L_{-1} \otimes \text{id}_U + \text{id}_V \otimes \partial_t$, as in \S\ref{connonSmooth}.

We are left to describe  $\nabla$ over $\mathfrak{D}_Q$. Recall that $\left(\frac{ds_+}{s_+}, - \frac{ds_-}{s_-} \right)$ is a generator of $\omega_{\Cc/S}$ over $\widehat{\mathscr{O}}_Q$. Using this trivialization, the connection $\nabla$ is defined by the endomorphism \eqref{eq:nablaendonod} extended over $R\llbracket q \rrbracket$ by linearity, that is: 
\begin{equation*}\left(L_{-1} \otimes s_+ \text{id}_{R(\!(s_+)\!)\llbracket q \rrbracket} + \text{id}_V \otimes s_+ \partial_{s_+}, -L_{-1} \otimes s_-\text{id}_{R(\!(s_-)\!)\llbracket q \rrbracket}  - \text{id}_V \otimes s_- \partial_{s_-} \right).
\end{equation*} 
As in the proof of Proposition \ref{prop:LC}, one needs to show that the target of this map lies inside $\Vv_{\Cc}(\mathfrak{D}_Q)$, and not merely in $V \otimes \big(R(\!(s_+)\!)\llbracket q \rrbracket \oplus R(\!(s_-)\!) \llbracket q \rrbracket\big)$. This amounts to a verification as in the proof of Proposition \ref{prop:LC}.

\subsection{The chiral Lie algebra over formal smoothings} \label{sec:chiralnodalfam}
Next, we globalize the Lie algebra $\mathfrak{L}(V)$ ancillary to $V$ and the chiral Lie algebra. 

Select a formal coordinate $t_i$ at $P_i(S)$. On the formal neighborhood  $\mathfrak{D}_{P_i} = \Spec \, R\llbracket t_i, q \rrbracket$, one has $\Vv_\Cc|_{\mathfrak{D}_{P_i}} \simeq_{t_i} \left( V \otimes R \llbracket t_i \rrbracket \right)\llbracket q \rrbracket$. Since the restriction of $\nabla$ over $\mathfrak{D}_{P_i}$ is given by $L_{-1} \otimes \text{id}_{R\llbracket t_i, q \rrbracket} + \text{id}_V \otimes \partial_{t_i}$, we have an identification as in~\eqref{eq:simeqt}:
\begin{equation}
\label{eq:idsheafanc}
\text{H}^0\left(\mathfrak{D}^\times_{P_i}, \Vv_\Cc\otimes \omega_{\Cc/S}/\textrm{Im}\nabla \right) 
\xrightarrow{\simeq_{t_i}}  \mathfrak{L}_{t_i}(V)\llbracket q \rrbracket. 
\end{equation}
We define the \textit{chiral Lie algebra} assigned to $V$ and  $(\Cc, P_\bullet)$ as 
\[ \Ll_{\Cc \setminus P_\bullet}(V):= \text{H}^0\left(\Cc \setminus P_\bullet(S), \dfrac{\Vv_\Cc \otimes \omega_{\Cc/S}}{\text{Im} \nabla} \right).
\] 
This extends \S \ref{ChiralDef}. The map
\begin{equation*}
\Ll_{\Cc\setminus P_\bullet}(V) \rightarrow  \oplus_{i=1}^n \text{H}^0\left(\mathfrak{D}^\times_{P_i}, \Vv_\Cc\otimes \omega_{\Cc/S}/\textrm{Im}\nabla \right) 
\xrightarrow{\cong}  \oplus_{i=1}^n \mathfrak{L}_{t_i}(V) \llbracket q \rrbracket
\end{equation*} induced by the restriction of sections and  \eqref{eq:idsheafanc} is a Lie algebra morphism.

\begin{rmk} \label{rmk:chiralnodalfam} 
As in Proposition \ref{prop:chiralnodal}, $\Ll_{\Cc \setminus P_\bullet}(V)$ can be described as the subspace of $\Ll_{\widetilde{\Cc} \setminus P_\bullet \sqcup Q_\bullet}(V)$ generated by those elements $\sigma$ whose restrictions $\sigma_{Q_\pm}$ to $ \mathfrak{L}_{Q_\pm}(V)\llbracket q \rrbracket$ satisfy
\begin{eqnarray*} 
\sigma_{Q_+} &=& \sum_{i,j\geq 0} \alpha_{i,j} \, A_{[a+i-j-1]} \, q^j, \\
\sigma_{Q_-} &=& (-1)^{a-1}\sum_{i,j\geq 0}  \alpha_{i,j} \sum_{k \geq 0} \dfrac{1}{k!}\left(L_1^kA \right)_{[a-k+j-i-1]} q^i \\
&=& \sum_{i,j\geq 0}  \alpha_{i,j} \, \vartheta\left(A_{[a+i-j-1]}\right) q^i
\end{eqnarray*} 
for homogeneous $A \in V$ of degree $a$ and integers $i, j \in \ZZ_{\geq 0}$. This uses that sections of $\Vv_\Cc$ over $\mathfrak{D}_Q$ are generated by \eqref{eq:typeA},  sections of $\omega_{\Cc/S}$ over $\mathfrak{D}_Q$ are generated by $\left(\frac{ds_+}{s_+}, - \frac{ds_-}{s_-} \right)$ over $\widehat{\mathscr{O}}_Q$, and the definition of $\vartheta$ from \eqref{eq:iota}.
\end{rmk}

\subsection{Sheaf of coinvariants over formal smoothings}
\label{sec:sheafcoinvariantsformalsmoothings}

Let $\left(\mathscr{C}/S, P_\bullet \right)$ be as in \S\ref{smoothings}. Fix formal coordinates $t_i$ at $P_i(S)$, for $i=1,\dots,n$. Given $V$-modules $M^1, \dots, M^n$, let $M^\bullet:=\otimes_{i=1}^n M^i$.
One defines the sheaf of coinvariants $\mathbb{V}(V; M^{\bullet})_{\left(\mathscr{C}/S, P_\bullet, t_\bullet\right)}$ as in \S\ref{Coinvariants}, that is:
\[
\mathbb{V}(V; M^{\bullet})_{\left(\mathscr{C}/S, P_\bullet, t_\bullet\right)}:=
\left(M^{\bullet} \otimes \mathscr{O}_S\right)_{\Ll_{\mathscr{C}\setminus P_\bullet}(V)}, 
\]
The \hyperlink{thm:Fact}{factorization theorem} holds for the restriction of $\mathbb{V}(V; M^{\bullet})_{\left(\mathscr{C}/S, P_\bullet, t_\bullet\right)}$ to the special fiber \mbox{$\mathscr{C}_0\rightarrow {S_0}$}:

\begin{theorem}
\label{thm:Factfamily}
Let $V$ be a  rational, $C_2$-cofinite vertex operator algebra with one-dimensional weight zero space. The map \eqref{map} induces a canonical $\mathscr{O}_{S_0}$-module isomorphism
\[
\mathbb{V}(V; M^{\bullet})_{\left(\mathscr{C}_0/S_0, P_\bullet, t_\bullet\right)} \cong 
\underset{W\in \mathscr{W}}\bigoplus{\mathbb{V}}\left(V; M^{\bullet}\otimes W \otimes W'\right)_{\left(\widetilde{\mathscr{C}}_0/S_0, P_\bullet\sqcup Q_\bullet, t_\bullet \sqcup s_\bullet\right)}.
\]
\end{theorem}

\noindent As a consequence of Corollary \ref{cor:CoherenceOnMgn} and Theorem \ref{thm:Factfamily} we obtain the following result on coherence.

\begin{theorem}
\label{thm:CoherenceOnS} 
Let $V$ be a  rational, $C_2$-cofinite vertex operator algebra with one-dimensional weight zero space. For finitely generated $V$-modules $M^1, \dots, M^n$, the sheaf  $\mathbb{V}(V; M^{\bullet})_{\left(\mathscr{C}/S, P_\bullet, t_\bullet\right)}$ is a coherent $\mathscr{O}_{S}$-module.
\end{theorem}

\begin{proof} Corollary \ref{cor:CoherenceOnMgn} implies that the restriction of $\mathbb{V}(V; M^{\bullet})_{\left(\mathscr{C}/S, P_\bullet, t_\bullet\right)}$ to $S\setminus S_0$ is coherent. We are left to show that the restriction of $\mathbb{V}(V; M^{\bullet})_{\left(\mathscr{C}/S, P_\bullet, t_\bullet\right)}$ to $S_0$, i.e. $\mathbb{V}(V; M^{\bullet})_{\left(\mathscr{C}_0/S_0, P_\bullet, t_\bullet\right)}$  is coherent. This follows from Theorem \ref{thm:Factfamily} and Corollary \ref{cor:CoherenceOnMgn} applied to ${\mathbb{V}}\left(V; M^{\bullet}\otimes W \otimes W'\right)_{\left(\widetilde{\mathscr{C}}_0/S_0, P_\bullet\sqcup Q_\bullet, t_\bullet \sqcup s_\bullet\right)}$.\end{proof}

\subsection{Sewing}
Given a simple $V$-module $W=\bigoplus_{i\geq 0}W_i$, define
\[
\bm{1}^W :=\sum_{i\geq 0}\bm{1}^{W_i} q^i \quad  \in \quad  (W\otimes W')\llbracket q\rrbracket,
\]
where
$
\bm{1}^{W_i}:=\mathrm{id}_{W_i}\in\mathrm{End}(W_i)\cong W_i\otimes W_i^\vee.
$
Consider the map
\begin{equation}
\label{eq:sew}
M^{\bullet} \longrightarrow   \oplus_{W\in\mathscr{W}} M^{\bullet} \otimes (W \otimes W')\llbracket q\rrbracket,\qquad
u  \mapsto \oplus_{W\in\mathscr{W}} \, u \otimes  \bm{1}^{W}.
\end{equation}
The following result extends  \cite[Thm 8.4.6]{nt} to curves of arbitrary genus.     
 
 \begin{theorem}[Sewing Theorem]
 \label{thm:SewingAndFactorization}
Let $V$ be a  rational, $C_2$-cofinite vertex operator algebra with one-dimensional weight zero space, and set $M^{\bullet}=\bigotimes_{i=1}^n M^i$ for   $V$-modules $M^i$.   
The map \eqref{eq:sew} induces a canonical $\mathscr{O}_{S_0} \llbracket q \rrbracket$-module isomorphism $\Psi$
such that the following diagram commutes
\begin{center}
 \begin{tikzcd}[column sep=0.4em]
{\mathbb{V}}(V; M^{\bullet})_{\left({\mathscr{C}}/ S, P_\bullet, t_\bullet\right)} \arrow[r,  "\Psi"] \arrow[d , two heads ]
    &\!\!\!\!\underset{W\in \mathscr{W}}{\bigoplus} \!\!\! {\mathbb{V}}\left(V; M^{\bullet}\otimes W \otimes W'\right)_{\left(\widetilde{\mathscr{C}}_0/S_0, P_\bullet\sqcup Q_\bullet, t_\bullet\sqcup s_\bullet \right)} \! \underset{\mathscr{O}_{S_0}}\otimes \! \mathscr{O}_{S_0}\llbracket q\rrbracket \arrow[d, two heads]\\
\mathbb{V}(V; M^{\bullet})_{\left(\mathscr{C}_0/S_0, P_\bullet, t_\bullet\right)} \arrow[r, "\cong" ]
&\!\!\!\underset{W\in \mathscr{W}}\bigoplus \!\! {\mathbb{V}}\left(V; M^{\bullet}\otimes W \otimes W'\right)_{\left(\widetilde{\mathscr{C}}_0/S_0, P_\bullet\sqcup Q_\bullet, t_\bullet\sqcup s_\bullet\right)}.
\end{tikzcd}
\end{center}
This isomorphism is equivariant with respect to change of the coordinates  $t_\bullet$.
\end{theorem}

\begin{rmk} 
Theorems \ref{thm:Factfamily} and  \ref{thm:SewingAndFactorization} give a canonical isomorphism
\[ 
{\mathbb{V}}(V; M^{\bullet})_{\left({\mathscr{C}}/ S, P_\bullet, t_\bullet\right)} \cong \mathbb{V} \left(V; M^{\bullet}\right)_{\left({\mathscr{C}}_0/S_0, P_\bullet, t_\bullet\right)} \llbracket q \rrbracket.
\]
In particular, this means that to the non-trivial deformation $\mathscr{C}$ of $\mathscr{C}_0$, there corresponds a trivial deformation of the space of conformal blocks.
\end{rmk}

\begin{proof}[Proof of Theorem \ref{thm:SewingAndFactorization}]
As in the proof of the factorization theorem, we can reduce to the case $\mathscr{C}\setminus P_\bullet$ affine over $S$, 
and show that \eqref{eq:sew} induces a canonical $R\llbracket q \rrbracket$-module isomorphism, still denoted $\Psi$, such that the following diagram commutes
\begin{center}
 \begin{tikzcd}[column sep=0.5em]
\left(M^{\bullet}\otimes \mathscr{O}_S\right)_{\Ll_{\mathscr{C}\setminus P_\bullet}(V)} \arrow[r,  "\Psi"] \arrow[d , two heads ]
    &\underset{W\in \mathscr{W}}{\bigoplus} \left( M^{\bullet}\otimes W \otimes W' \otimes \mathscr{O}_{S_0} \right)_{\Ll_{\widetilde{\mathscr{C}}_0\setminus P_\bullet\sqcup Q_\bullet} (V)} \!\!\underset{\mathscr{O}_{S_0}}\otimes\!\! \mathscr{O}_{S_0}\llbracket q\rrbracket \arrow[d, two heads]\\
\left( M^{\bullet}\otimes \mathscr{O}_{S_0} \right)_{\Ll_{\mathscr{C}_0\setminus P_\bullet}(V)} \arrow[r, "\cong" ]
&\underset{W\in \mathscr{W}}\bigoplus  \left( M^{\bullet}\otimes W \otimes W' \otimes \mathscr{O}_{S_0}\right)_{\Ll_{\widetilde{\mathscr{C}}_0\setminus P_\bullet\sqcup Q_\bullet} (V)}.
\end{tikzcd}
\end{center}
Here, the vertical maps are obtained by specializing at $q= 0$.

\noindent \setword{\textit{Step 1}}{StSew1}.
We show that \eqref{eq:sew} induces a well-defined map $\Psi$ between spaces of coinvariants. 
For this, it is enough to show that for each $\sigma\in  \Ll_{\mathscr{C}\setminus P_\bullet}(V)$
and $W\in \mathscr{W}$, one has
$\mbox{$\sigma_{P_\bullet}\left( M^\bullet\right) \otimes \bm{1}^W$} = 
\sigma \left(M^\bullet\otimes \bm{1}^W \right)$,
or equivalently
\begin{equation}
\label{eq:vanishing}
\left( \sigma_{Q_+}\otimes 1 +1\otimes  \sigma_{Q_-}\right)\left(\bm{1}^W\right)=0.
\end{equation}
From Remark \ref{rmk:chiralnodalfam} and by linearity, we can reduce to the case when 
\begin{align*}
\sigma_{Q_+} &=  A_{[a+i-j-1]}\, q^j &\mbox{and}&& \sigma_{Q_-} &= \vartheta\left(A_{[a+i-j-1]}\right) q^i
\end{align*}
for homogeneous $A \in V$ of degree $a$ and integers $i, j \geq 0$.  The vanishing of \eqref{eq:vanishing} follows from the identity
\begin{equation}
\label{eq:identityA11A}
\left( A_{[a+i-j-1]} \otimes 1 + 1\otimes \vartheta\left(A_{[a+i-j-1]} \right)q^{i-j} \right)\left( \bm{1}^W \right)=0
\end{equation}
established in \cite[Lemma 8.7.1]{nt} (there is a sign  difference between the involution $\vartheta$ used here and the involution used in \cite{nt}).
Thus we conclude that the map $\Psi$ is well-defined and makes the diagram above commute.

\noindent \setword{\textit{Step 2}}{StSew2}. Since (i) the target of $\Psi$ is a free $\mathscr{O}_{S_0}\llbracket q \rrbracket$-module of finite rank, (ii) the source is finitely generated (Theorem \ref{thm:CoherenceOnS}), and (iii) $\Phi$ is an isomorphism modulo $q$ (Theorem \ref{thm:Factfamily}), Nakayama's lemma implies that $\Psi$ is an isomorphism
(this is as in \cite{tuy,loowzw,nt}).
\end{proof}

\subsection{The sheaf of coinvariants on $\widetriangle{\mathcal{M}}_{g,n}$}
Let $\widehat{\mathcal{M}}_{g,n}$ be the restriction of $\widetriangle{\mathcal{M}}_{g,n}$ over the locus ${\mathcal{M}}_{g,n}$ of smooth pointed curves (see \S\ref{blowupmu} for definitions). 
The vertex algebra bundle and the chiral Lie algebra defined  on smooth curves in \S\S\ref{sec:VAstable} and \ref{Chiral}, respectively, give the vertex algebra bundle and the sheaf of chiral Lie algebras on $\widehat{\mathcal{M}}_{g,n}$. In \S\S\ref{sec:Vforfamilies} and \ref{sec:chiralnodalfam}, the vertex algebra bundle and the sheaf of chiral Lie algebras are defined on formal smoothings of families of nodal curves. Gluing as in \cite{BLDescente}, one obtains the vertex algebra bundle and the sheaf of chiral Lie algebras on  $\widetriangle{\mathcal{M}}_{g,n}$.  Similarly, the sheaf of coinvariants  defined  on families of smooth curves in \S\ref{Coinvariants} gives the sheaf of coinvariants $\mathbb{V}(V;M^\bullet)$ on $\widehat{\mathcal{M}}_{g,n}$. In \S\ref{sec:sheafcoinvariantsformalsmoothings}, the sheaf of coinvariants is defined on formal smoothings of families of nodal curves. By gluing as in \cite{BLDescente}, one obtains the \textit{sheaf of coinvariants} $\mathbb{V}(V;M^\bullet)$ on  $\widetriangle{\mathcal{M}}_{g,n}$.

\subsection{The sheaf of coinvariants on $\overline{\mathcal{M}}_{g,n}$}
\label{sec:sheafofcoinariantsfinal} 
Throughout this section, we require every $V$-module $M$ to further satisfy the following property: there exists a complex number $c_M$ called the \textit{conformal dimension} (or conformal weight) of $M$ such that for every homogeneous $v \in M$ one has $L_0(v)=(\deg(v) + c_M)v$. This condition holds whenever $M=V$ or, for instance, when $M$ is a simple $V$-module. Furthermore, if  $V$ is $C_2$-cofinite, then $c_M$ is a rational number \cite{MiyamotoC2}.

The sheaf of coinvariants on $\overline{\mathcal{M}}_{g,n}$ is obtained from a two-step process \cite[\S 6.3]{dgt}, reviewed next. Consider the group scheme $\textrm{Aut}_+\mathcal{O}$ which represents the functor assigning to a $\mathbb{C}$-algebra $R$ the group:
\[
\textrm{Aut}_+\mathcal{O}(R) = \left\{z \mapsto \rho(z)= z + a_2 z^2 + \cdots \, | \, a_i \in R \right\}.
\]
This is a subgroup scheme of the group scheme $\textrm{Aut}\,\mathcal{O}$.  In particular, one has $\textrm{Aut}\,\mathcal{O} =\mathbb{G}_m \ltimes \textrm{Aut}_+\mathcal{O}$.
The $(\textrm{Aut}\,\mathcal{O})^{\oplus n}$-torsor $\widetriangle{\mathcal{M}}_{g,n} \rightarrow \overline{\mathcal{M}}_{g,n}$ factors as the composition of an $(\textrm{Aut}_+\mathcal{O})^{\oplus n}$-torsor and a $\mathbb{G}_m^{\oplus n}$-torsor:
\begin{equation}
\label{eq:torsors}
\begin{tikzcd}
&\widetriangle{\mathcal{M}}_{g,n} \arrow[rightarrow]{dl}[swap]{(\textrm{Aut}_+\mathcal{O})^{\oplus n}} \arrow[rightarrow]{dd}{(\textrm{Aut}\,\mathcal{O})^{\oplus n}} \\
\overline{\mathcal{J}}_{g,n}^{\times} \arrow[rightarrow]{dr}[swap]{\mathbb{G}_m^{\oplus n}}&\\
& \overline{\mathcal{M}}_{g,n}.
\end{tikzcd}
\end{equation}
Here $\overline{\mathcal{J}}_{g,n}^{\times}$ is the space of tuples $(C,P_\bullet, \tau_\bullet)$, where  $\tau_\bullet=(\tau_1,\dots,\tau_n)$ with $\tau_i$ a non-zero $1$-jet of a formal coordinate at $P_i$ for each $i$.  As in \cite[\S 6.3.1]{dgt}, $(\textrm{Aut}_+\mathcal{O})^{\oplus n}$ acts on $\mathbb{V}(V;M^\bullet)$, and descending along the $(\textrm{Aut}_+\mathcal{O})^{\oplus n}$-torsor in \eqref{eq:torsors}, one obtains a sheaf of coinvariants $\mathbb{V}^J(V;M^\bullet)$ on $\overline{\mathcal{J}}_{g,n}^{\times}$.   

The idea for the descent of $\mathbb{V}^J(V;M^\bullet)$ to $\overline{\mathcal{M}}_{g,n}$ is inspired by Tsuchimoto  \cite{ts} (and  used to prove \cite[Theorem 8.1]{dgt}): first, one  tensors $\mathbb{V}^J(V;M^\bullet)$ with an appropriate line bundle to obtain a new sheaf on which $\mathbb{G}_m^{\oplus n}$ acts;  after descending the new sheaf, one then tensors back with the dual of the  line bundle.  Next, we detail this argument using root stacks.

The case $n >1$ can be treated by iterating the procedure used for $n=1$, hence we discuss only this latter case and set $M^1=:M$. We will restrict to the case in which the conformal dimension $c_M$ of $M$ is a rational number, and we write $c_M=\frac{a}{d}$ for  $a\in \mathbb{Z}$ and $d \in \mathbb{N}$.  We consider line bundles
$\mathcal{L}_{\mathcal{M}}=\left(\Psi^\vee\right)^{\otimes a}$  on  $\mathcal{M}:=\overline{\mathcal{M}}_{g,1}$,  and $\mathcal{L}_{\mathcal{J}}=\pi^*\mathcal{L}_{\mathcal{M}}$ on $\mathcal{J}:= \overline{\mathcal{J}}_{g,1}^{\times}$, where $\pi \colon \mathcal{J} \to \mathcal{M}$ is the map forgetting the $1$-jets and $\Psi$ is the cotangent line bundle corresponding to the marked point.

\subsubsection{Root stacks} 
We briefly review some properties of the root stacks $\sqrt[\leftroot{-2}\uproot{2}d]{\mathcal{L}_{\mathcal{M}}/\mathcal{M}}$ and $\sqrt[\leftroot{-2}\uproot{2}d]{\mathcal{L}_{\mathcal{J}}/\mathcal{J}}$.  Our primary reference is \cite{zbMATH05240826}, and more information can be found in 
\cite[\S\S 3.3 and 4]{MaxModernPrimer},   \cite[\S 10.3]{Olsson}, and  \cite[App.~B]{AGVRoot}.

The root stack $\sqrt[\leftroot{-2}\uproot{2}d]{\mathcal{L}_{\mathcal{M}}/\mathcal{M}}$ is the stack parametrizing $d$-th roots of the line bundle $\mathcal{L}_\mathcal{M}$. In other words, $\sqrt[\leftroot{-2}\uproot{2}d]{\mathcal{L}_{\mathcal{M}}/\mathcal{M}}$ represents the functor which associates to every scheme $\phi \colon Y \to \mathcal{M}$ the groupoid of pairs $(\mathcal{N}, f)$ where $\mathcal{N}$ is a line bundle on $Y$ and $f \colon \mathcal{N}^{\otimes d}\to \phi^*\mathcal{L}_{\mathcal{M}}$ is an isomorphism. An isomorphism between $(\mathcal{N}_1, f_1)$ and $(\mathcal{N}_2, f_2)$ is an isomorphism of line bundles $g \colon  \mathcal{N}_1\to \mathcal{N}_2$  such that $f_2 \circ g^{\otimes d} = f_1$. One defines $\sqrt[\leftroot{-2}\uproot{2}d]{\mathcal{L}_{\mathcal{J}}/\mathcal{J}}$ similarly. 

Since $\mathcal{L}_{\mathcal{J}}$ is the pullback of $\mathcal{L}_{\mathcal{M}}$ along $\pi$, one has a Cartesian diagram
\begin{equation}\label{CDRS}
\begin{tikzcd}
\sqrt[\leftroot{-2}\uproot{2}d]{\mathcal{L}_{\mathcal{J}}/\mathcal{J}}  \arrow{r}{\sqrt[\leftroot{-2}\uproot{2}d]{\pi}} \arrow[swap]{d}{p_{\mathcal{J}}} & \sqrt[\leftroot{-2}\uproot{2}d]{\mathcal{L}_{\mathcal{M}}/\mathcal{M}} \arrow{d}{p_{\mathcal{M}}} \\
\mathcal{J} \arrow{r}{\pi} & \mathcal{M}.
\end{tikzcd}
\end{equation} 
In particular, $\sqrt[\leftroot{-2}\uproot{2}d]{\mathcal{L}_{\mathcal{J}}/\mathcal{J}}  \to \sqrt[\leftroot{-2}\uproot{2}d]{\mathcal{L}_{\mathcal{M}}/\mathcal{M}}$ is a $\mathbb{G}_m$-torsor. The stacks $\sqrt[\leftroot{-2}\uproot{2}d]{\mathcal{L}_{\mathcal{M}}/\mathcal{M}}$ and $\sqrt[\leftroot{-2}\uproot{2}d]{\mathcal{L}_{\mathcal{J}}/\mathcal{J}}$ have universal  line bundles $\mathcal{U}_{\mathcal{M}}$ and $\mathcal{U}_{\mathcal{J}}=\left(\sqrt[\leftroot{-2}\uproot{2}d]{\pi}\right)^* \mathcal{U}_{\mathcal{M}}$ such that $\mathcal{U}_{\mathcal{M}}^{\otimes d} = p_{\mathcal{M}}^* \, \mathcal{L}_{\mathcal{M}}$ and  $\mathcal{U}_{\mathcal{J}}^{\otimes d} = p_{\mathcal{J}}^* \, \mathcal{L}_{\mathcal{J}}$.

The key property that we will use about $\sqrt[\leftroot{-2}\uproot{2}d]{\mathcal{L}_{\mathcal{J}}/\mathcal{J}}$ (and  $\sqrt[\leftroot{-2}\uproot{2}d]{\mathcal{L}_{\mathcal{M}}/\mathcal{M}}$ as well)
is that its category of quasi-coherent sheaves has an eigendecomposition with respect to the action of the inertial group $\mu_d$ of $d$-th roots of unity, and the degree zero component on $\sqrt[\leftroot{-2}\uproot{2}d]{\mathcal{L}_{\mathcal{J}}/\mathcal{J}}$ (resp.,~$\sqrt[\leftroot{-2}\uproot{2}d]{\mathcal{L}_{\mathcal{M}}/\mathcal{M}}$) consists of pullbacks of quasi-coherent sheaves on $\mathcal{J}$ (resp.,~$\mathcal{M}$). By \cite[Lemma 3.1.1.7]{zbMATH05240826},
the eigensheaves for the trivial character on the root stack are identified with sheaves on the base stack $\mathcal{M}$. This allows one to conclude that the pullback via $p_{\mathcal{J}}$ is fully faithful, i.e., two quasi-coherent sheaves on $\mathcal{J}$ are isomorphic if and only if they are isomorphic when pulled back to $\sqrt[\leftroot{-2}\uproot{2}d]{\mathcal{L}_{\mathcal{J}}/\mathcal{J}}$.

\subsubsection{The final descent}  \label{sec:finaldescent}
Let $\mathbb{V}^{J}:=\mathbb{V}^J(V;M)$. The quasi-coherent sheaf $p_{\mathcal{J}}^*\mathbb{V}^{J}\otimes \mathcal{U}_{\mathcal{J}}$
on $\sqrt[\leftroot{-2}\uproot{2}d]{\mathcal{L}_{\mathcal{J}}/\mathcal{J}}$  has an action of $\mathbb{G}_m$ as in \cite[\S\S 4.2.1, 6.3.2]{dgt}\footnote{For this action, one uses the filtration on $\mathbb{V}^J(V;M)$ induced by the $\mathbb{Z}_{\geq 0}$-grading of the module $M$.}. 
Descending along the $\mathbb{G}_m$-torsor $\sqrt[\leftroot{-2}\uproot{2}d]{\mathcal{L}_{\mathcal{J}}/\mathcal{J}}  \to \sqrt[\leftroot{-2}\uproot{2}d]{\mathcal{L}_{\mathcal{M}}/\mathcal{M}}$, we obtain a sheaf $\mathcal{F}$ on $\sqrt[\leftroot{-2}\uproot{2}d]{\mathcal{L}_{\mathcal{M}}/\mathcal{M}}$ for which
\begin{equation*}\label{rootdescent}
\left(\sqrt[\leftroot{-2}\uproot{2}d]{\pi}\right)^*\mathcal{F}\cong p_{\mathcal{J}}^*\mathbb{V}^{J}\otimes \mathcal{U}_{\mathcal{J}}.
\end{equation*}

Tensoring the above with $\left(\sqrt[\leftroot{-2}\uproot{2}d]{\pi}\right)^*\mathcal{U}_{\mathcal{M}}^{\vee}= \mathcal{U}_{\mathcal{J}}^{\vee}$, we have  
\begin{equation}\label{rootdiagram}
\left(\sqrt[\leftroot{-2}\uproot{2}d]{\pi}\right)^*\left(\mathcal{F} \otimes \mathcal{U}_{\mathcal{M}}^{\vee} \right) 
\cong p_{\mathcal{J}}^*\mathbb{V}^{J}.
\end{equation}
It follows that $\mathcal{F} \otimes \mathcal{U}_{\mathcal{M}}^{\vee}$ lives in the degree zero component, hence it descends to 
 a sheaf $\mathbb{V}(V;M)$ on $\mathcal{M}$ such that
$p_{\mathcal{M}}^* \, \mathbb{V}(V;M) \cong \mathcal{F} \otimes \mathcal{U}_{\mathcal{M}}^{\vee}$. By its construction and the commutativity of \eqref{CDRS}, one has 
\begin{equation*}\label{FF}
\left(\pi \circ p_{\mathcal{J}}\right)^*  \mathbb{V}(V;M)  = \left(p_{\mathcal{M}} \circ \sqrt[\leftroot{-2}\uproot{2}d]{\pi} \right)^*  \mathbb{V}(V;M) =p_{\mathcal{J}}^*\mathbb{V}^{J}.
\end{equation*} Since the pullback of sheaves to a root stack is fully faithful, we deduce that $\mathbb{V}^{J} = \pi^* \, \mathbb{V}(V;M)$. In particular, $\mathbb{V}^{J}$ descends to a sheaf $\mathbb{V}(V;M)$, which is therefore well-defined on  $\overline{\mathcal{M}}_{g,1}$.

\begin{rmk} \label{rmk:vratdescent} Here  we list some observations.
\begin{enumerate} \item  When $V$ is $C_2$-cofinite and rational, we can descend $\mathbb{V}^J(V;M)$ to a sheaf $\mathbb{V}(V;M)$ over $\overline{\mathcal{M}}_{g,n}$ for every finitely-generated $V$-module $M$. Indeed, since the category of $V$-modules is semisimple, we can decompose $M=\oplus_{\ell \in I} M^{\ell}$ with $M^{\ell}$ a simple $V$-module with rational conformal dimension and $I$ a finite index set. This induces the decomposition $\mathbb{V}^J(V;M) = \oplus_{\ell \in I} \mathbb{V}^J(V;M^{\ell})$, and we can apply the descent argument of \S\ref{sec:finaldescent} to each component $\mathbb{V}^J(V;M^{\ell})$.

\item For simplicity, we have given the details of the argument when the conformal dimensions are rational. This is indeed the case for simple modules over a $C_2$-cofinite $V$, and  allows for the use of a root stack defined by a line bundle. While not needed for this paper, an analogous argument can be made for complex, irrational conformal dimensions, using a gerbe that is a generalization of the root stack.

\item When $M^i=V$ for all $i$, the action of $\mathbb{G}_m^{\oplus n}$  on $\mathbb{V}^J(V;M^\bullet)$ from \cite[\S\S 4.2.1, 6.3.2]{dgt} is compatible with the restriction of the action of 
$(\textrm{Aut}\,\mathcal{O})^{\oplus n}$. In this case, the above construction simplifies, since the two descents along the $(\textrm{Aut}_+\mathcal{O})^{\oplus n}$-torsor and $\mathbb{G}_m^{\oplus n}$-torsor in  \eqref{eq:torsors} are equivalent to the descent along the $(\textrm{Aut}\,\mathcal{O})^{\oplus n}$-torsor. 

\item The method used  to construct the sheaf $\mathbb{V}(V;M^\bullet)$ on $\overline{\mathcal{M}}_{g,n}$ from $\mathbb{V}^J(V;M^\bullet)$ guarantees that, when $\mathbb{V}(V;M^\bullet)$ is of finite rank on ${\mathcal{M}}_{g,n}$, the Chern character of $\mathbb{V}(V;M^\bullet)$ on ${\mathcal{M}}_{g,n}$ is given by \cite[Cor 9.1]{dgt}.\end{enumerate}
\end{rmk}

\subsection{Proof of the {\protect\hyperlink{thm:VBS}{VB corollary}}}

\label{sec:proofofVBC}

By means of Theorems \ref{thm:CoherenceOnS} and  \ref{thm:SewingAndFactorization}, one concludes that the sheaf of coinvariants $\mathbb{V}(V;M^\bullet)$ is a vector bundle of finite rank on $\widetriangle{\mathcal{M}}_{g,n}$, and this gives rise to a vector bundle of finite rank on $\overline{\mathcal{M}}_{g,n}$,
 as in \cite{tuy}, \cite[\S 2.7]{sorger}, \cite{loowzw}, \cite{nt}. We sketch the argument  for completeness.

First, we argue that $\mathbb{V}(V;M^\bullet)$ is a vector bundle of finite rank on $\widetriangle{\mathcal{M}}_{g,n}$.
The sheaf $\mathbb{V}(V;M^\bullet)$ on $\widehat{\mathcal{M}}_{g,n}$ is coherent (Corollary \ref{cor:CoherenceOnMgn}) and is equipped with a projectively flat connection  \cite{dgt}. As in \cite{tuy}, see also \cite[\S 2.7]{sorger}, it follows that $\mathbb{V}(V;M^\bullet)$ is locally free of finite rank on $\widehat{\mathcal{M}}_{g,n}$.  After Theorem \ref{thm:CoherenceOnS} and gluing the sheaf as in \cite{BLDescente},  it follows that the sheaf $\mathbb{V}(V;M^\bullet)$ is also coherent on $\widetriangle{\mathcal{M}}_{g,n}$.
It remains to show that $\mathbb{V}(V;M^\bullet)$ is locally free  on $\widetriangle{\mathcal{M}}_{g,n}$.
For this, consider a stable family of $n$-pointed nodal curves $(\mathscr{C}_0 \to \text{Spec}(R), P_\bullet)$, and for simplicity, assume that it has only one simple node. Consider its formal smoothing $(\mathscr{C} \to \mathrm{Spec}(R\llbracket q\rrbracket), P_\bullet)$ as described in \S \ref{smoothings}. For each $i$, fix a formal coordinate $t_i$ at $P_i(S)$. The sewing theorem (Theorem \ref{thm:SewingAndFactorization}) implies that $\mathbb{V}(V;M^\bullet)_{(\mathscr{C}/S,P_\bullet, t_\bullet)}$ is locally free of finite rank, hence we conclude the argument.  
For families of curves with more nodes, one proceeds similarly. It follows that $\mathbb{V}(V;M^\bullet)$ is a vector bundle of finite rank on $\widetriangle{\mathcal{M}}_{g,n}$.

Finally, since $V$ is rational, by \S\ref{sec:sheafofcoinariantsfinal} and Remark \ref{rmk:vratdescent}(i), we can descend $\mathbb{V}(V;M^\bullet)$ to a sheaf of coinvariants on $\overline{\mathcal{M}}_{g,n}$. As the descent of a vector bundle is a vector bundle, this concludes the proof of the \hyperlink{thm:VBS}{VB corollary}.
\qed


\section{Examples}\label{ExampleSection}
Here we list examples of vertex operator algebras $V$ satisfying the hypotheses of our theorems, namely:
 (1) $V=\bigoplus_{i\in \mathbb{Z}_{\ge 0}}V_i$ with $V_0\cong \mathbb{C}$; (2) $V$ is rational; and (3) $V$ is $C_2$-cofinite.

\subsection{Virasoro VOAs}
Given the Lie subalgebra 
$\mathrm{Vir}_{\ge 0}:= \mathbb{C}K \oplus z \mathbb{C}\llbracket z \rrbracket \partial_z$ of the Virasoro Lie algebra $\mathrm{Vir}$, and  $c$, $h \in \mathbb{C}$, 
let $M_{c,h}:=U(\mathrm{Vir})\otimes_{U\left(\mathrm{Vir}_{\ge 0}\right)} \mathbb{C}\bf{1}$, 
where $\mathbb{C}\bm{1}$ inherits the structure of a $\mathrm{Vir}_{\ge 0}$-module by setting 
$L_{p>0}{\bf{1}} =0$, $L_0 {\bf{1}} =h  {\bf{1}}$, and $K {\bf{1}} =c {\bf{1}}$. There is a unique maximal proper submodule $J_{c,h} \subset M_{c,h}$. 
For $h=0$,  $J_{c,0}$ contains a submodule generated by  the singular vector $L_{-1}{\bf{1}}\in M_{c,0}$  \cite{ff}.  Set
\[
L_{c,h}:=M_{c,h}/J_{c,h}, \quad  \ M_{c}:= M_{c,0}/\langle L_{-1}{\bf{1}} \rangle, \quad \mbox{ and }\quad Vir_{c} :=L_{c,0}.
\]
If $c\ne c_{p,q}:=1-\frac{6(p-q)^2}{pq}$, with relatively prime $p,q\in \mathbb{N}$ such that $1<p<q$, then $M_c\cong Vir_{c}$, that is, $J_{c,0}=\langle L_{-1}{\bf{1}} \rangle$, 
while for $c=c_{p,q}$, the submodule $J_{c,0}$ is generated by  two singular vectors \cite{ff}.
By \cite[Thm 4.3]{FrenkelZhu},  $M_{c}$ and $Vir_{c}$ are VOAs. Since  $A(M_{c})\cong \mathbb{C}[x]$  \cite[Thm 4.6]{FrenkelZhu},  $M_c$ is not rational. However, $Vir_{c}$ is rational if and only if $c= c_{p,q}$ \cite[Thm 4.2]{WWang} if and only if $Vir_{c}$ is $C_2$-cofinite  \cite[Lemma 12.3]{DongLiMasonModular} (see also \cite[Prop.~3.4.1]{ArakawaC2Lisse}).  
In this case, $A(Vir_{c})$ is a quotient of $\mathbb{C}[x]$, and simple $Vir_c$-modules are the $L_{c,h}$, for $h=\frac{(np-mq)^2-(p-q)^2}{4pq}$ with  $0<m<p$ and  $0<n<q$.

\subsection{Simple affine VOAs} 
One may associate to a finite-dimensional complex simple Lie algebra $\mathfrak{g}$  and  $\ell \in \mathbb{Z}_{>0}$, a simple vertex operator algebra $L_{\ell}(\mathfrak{g})$, 
described in \cite[\S 2]{FrenkelZhu}, \cite[\S 6.2]{lepli}  (see also  \cite[\S A.1.1]{nt}).
This is rational by \cite[Thm 3.1.3]{FrenkelZhu} 
and $C_2$-cofinite by \cite{zhu}
(see also \cite[Prop.~12.6]{DongLiMasonModular}, \cite[Prop.~3.5.1]{ArakawaC2Lisse}). 

\subsection{The moonshine module $V^{\natural}$}
A rational vertex operator algebra $V$ with no nontrivial simple $V$-modules is called \textit{holomorphic}, and if $(1)$  and $(3)$ also hold, then $V$ must have central charge divisible by $8$ \cite{dong2004holomorphic}.  One example is given by the moonshine module $V^{\natural}$ of central charge $24$, whose automorphism group is the monster  group \cite{FLM1}.

\subsection{Even lattice vertex algebras}\label{Lattice}
Vertex operator algebras $V_L$  given by positive-definite even lattices $L$ of finite rank \cite{Borcherds} are rational \cite{DLattice} and $C_2$-cofinite  \cite{DongLiMasonModular}. Zhu's algebra is described in \cite[Thm 3.4]{DLMZhu}.

 \subsection{Exceptional $W$-algebras}\label{W}
 Arakawa \cite{ArakawaC2W} has shown that a large class of simple 
 $W$-algebras are $C_2$-cofinite, including the minimal series principal $W$-algebras \cite{EFKW} and the exceptional $W$-algebras  of Kac-Wakimoto \cite{KW08}.
Moreover, the minimal series principal $W$-algebras and a large subclass of exceptional affine $W$-algebras are rational \cite{arakawa2013, ArakawaRatW, arakawa2019rationality}.

 \subsection{Orbifolds, commutants, and tensor products}
More vertex operator algebras can be obtained from the  examples discussed above through standard constructions resulting in orbifold algebras, commutants, and tensor products.
These constructions  often preserve  our desired properties.  For instance,
for  a finite subgroup $G$ of the automorphism group $Aut(V)$, the orbifold algebra
 $V^G$ is conjecturally $C_2$-cofinite and rational when so is $V$.  This holds for $G$ solvable and $V$ simple  \cite{MiyamotoCyclic, CarnahanMiyamoto}.  For a subalgebra $\mathcal{A}\subset V$,
the commutant $\operatorname{Com}(\mathcal{A}, V)$ \cite{FrenkelZhu} is conjecturally $C_2$-cofinite and rational if so are both  $\mathcal{A}$ and $V$. This holds for parafermions \cite{DongRen}.
The VOAs $V^1, \ldots, V^m$ are rational if and only if $V=\bigotimes_{i=1}^mV^i$ is rational \cite{DongMasonZhu}; in this case, if the $V^i$ are $C_2$-cofinite, then so is $V$ \cite{DongLiMasonRegular}.


\appendix

\section{Zhu's Lie algebra and isomorphic coinvariants}
\label{sec:CoinvariantsHistory}
For a smooth curve $C$ and a  quasi-primary generated vertex algebra $V$ with $V_0\cong \mathbb{C}$, in addition to the chiral Lie algebra $\Ll_{C\setminus P_\bullet}(V)$ (\S\ref{ChiralDef}), one also has Zhu's Lie algebra  $\mathfrak{g}_{\,C\setminus P_\bullet}(V)$, reviewed in \S\ref{sec:gCoinv}. 

In Proposition \ref{prop:IsoCo},  we show that when defined, $\mathfrak{g}_{C\setminus P_ \bullet}(V)$ is isomorphic to the image of  $\Ll_{C\setminus P_\bullet}(V)$ under the Lie algebra homomorphism $\varphi_{\Ll}$ (Proposition \ref{prop:IsoCo}). 
Nagatomo and Tsuchiya extend the definition of $\mathfrak{g}_{C\setminus P_ \bullet}(V)$ to stable pointed rational curves \cite{nt}, and they  indicate that  their coinvariants are equivalent to those studied by Beilinson and Drinfeld in \cite{bd}, suggesting they knew that Proposition \ref{prop:IsoCo} holds in that case. 

A quasi-primary vector  is an element $A\in V$ such that $L_{1}A=0$, and  $V$  is quasi-primary generated if and only if $L_{1}V_1=0$ \cite{dlinm}. 
A vertex algebra $V=\bigoplus_{i\geq 0}V_i$ with $V_0 \cong \mathbb{C}$ satisfies $L_{1}V_1=0$
if and only if $V\cong V'$ (see \cite[\S 5.3]{fhl} and \cite[\S2]{DM}).  In particular, in the results of Huang \cite{HuangVerlinde} and Codogni \cite{codogni}, the vertex algebras studied are quasi-primary generated.

\subsection{The Lie algebra $\mathfrak{g}_{\,C\setminus P_\bullet}(V)$}
\label{sec:gCoinv}

In \cite{zhu},  given a smooth pointed curve $(C,P_\bullet)$ and  a  quasi-primary generated vertex operator algebra~$V$ for which $V_0\cong \mathbb{C}$, Zhu defines a Lie algebra   $\mathfrak{g}_{\,C\setminus P_\bullet}(V)$, generalizing the construction of Tsuchiya, Ueno, and Yamada for affine Lie algebras.  Namely, consider
\begin{equation*}
\mathfrak{g}_{\,C\setminus P_{\bullet}}(V) := \varphi_{\mathfrak{g}}\left(\oplus_{k\geq 0} V_k \otimes \text{H}^0\left(C\setminus P_\bullet, \omega_C^{\otimes 1-k}\right)\right)
\end{equation*}
where
\begin{equation}
\label{eq:phig}
\varphi_{\mathfrak{g}}\colon \oplus_{k\geq 0} V_k \otimes \text{H}^0\left(C\setminus P_\bullet, \omega_C^{\otimes 1-k}\right) \rightarrow \oplus_{i=1}^n \mathfrak{L}_{t_i}(V)\end{equation}
is the map induced by
\[
B\otimes \mu \mapsto \left(  \textrm{Res}_{t_i=0}\,Y[B,t_i] \mu_{P_i}(dt_i)^k \right)_{i=1,\dots,n}.\
\]
Here $t_i$ is a formal coordinate at the point $P_i$, $Y[B,t_i]:=\sum_{k\in \mathbb{Z}}B_{[k]}t_i^{-k-1}$, and  $\mu_{P_i}$ is the Laurent series expansion of $\mu$ at $P_i$,  the image of $\mu$ via
\[
\text{H}^0\left(C\setminus P_\bullet, \omega_C^{\otimes 1-k}\right) \rightarrow \text{H}^0\left(D^\times_{P_i}, \omega_C^{\otimes 1-k}\right) \simeq_{t_i} \mathbb{C}(\!( t_i )\!) (dt_i)^{1-k}.
\]
When $V$ is assumed to be quasi-primary generated with $V_0\cong \mathbb{C}$, Zhu shows that $\mathfrak{g}_{\,C\setminus P_{\bullet}}(V)$ is a Lie subalgebra of $\mathfrak{L}(V)^{\oplus n}$. The argument uses that any \textit{fixed} smooth algebraic curve admits an atlas such that all transition functions are M\"obius transformations. 
Transition functions between charts on families of curves of arbitrary genus are more general, hence the need to consider the more involved construction for the chiral Lie algebra based on the $(\mathrm{Aut}\, \mathcal{O})$-twist of $V$ in \S\ref{sec:VAstable}.

\subsection{Isomorphism of coinvariants}
\label{sec:IsoCo}

When $\mathfrak{g}_{\,C\setminus P_\bullet}(V)$ is well-defined and $C\setminus P_\bullet$ is affine, one can define the  space of coinvariants $M^{\bullet}_{\mathfrak{g}_{\,C\setminus P_\bullet}(V)}$ as the quotient of the $\mathfrak{L}(V)^{\oplus n}$-module  $M^\bullet$ by the action of the Lie subalgebra $\mathfrak{g}_{\,C\setminus P_\bullet}(V)$ of $\mathfrak{L}(V)^{\oplus n}$.
These spaces were introduced in \cite{ZhuGlobal}, and studied also in \cite{an1, nt}.
Recall the homomorphisms $\varphi_{\Ll}$ from \eqref{phiL} and $\varphi_{\mathfrak{g}}$ from~\eqref{eq:phig}.

\begin{proposition}
\label{prop:IsoCo}
When $\mathfrak{g}_{\,C\setminus P_\bullet}(V)$ is well-defined (\S\ref{sec:gCoinv}), one has
\[
\textup{Im}(\varphi_{\Ll}) \cong \textup{Im}(\varphi_{\mathfrak{g}}).
\]
It follows that there exists an isomorphism of vector spaces
\[
M^\bullet_{\mathfrak{g}_{\,C\setminus P_\bullet}(V)}  \; \cong \; M^\bullet_{\Ll_{C\setminus P_\bullet}(V)}.
\]
\end{proposition}

\begin{proof}
One has
\begin{multline}
\label{eq:gCPH0oplus1}
\oplus_{k\geq 0} V_k \otimes \text{H}^0\left(C\setminus P_\bullet, \omega_C^{\otimes 1-k}\right) \cong \text{H}^0\left(C\setminus P_\bullet, \oplus_{k\geq 0}V_k \otimes \omega_C^{\otimes 1-k} \right) \\
 \cong \text{H}^0\left(C\setminus P_\bullet, \oplus_{k\geq 0} \left(\omega_C^{\otimes 1-k} \right)^{\oplus \dim V_k}\right).
 \end{multline}
From Lemma \ref{lemma:gr}, one has $\textup{gr}_\bullet \Vv_C \cong  \oplus_{k\geq 0}  \left(\omega_C^{\otimes -k} \right)^{\oplus \dim V_k}$.
It follows that 
\[
\text{H}^0\left(C\setminus P_\bullet, \oplus_{k\geq 0} \left(\omega_C^{\otimes 1-k} \right)^{\oplus \dim V_k}\right)
 \cong \text{H}^0\bigl(C\setminus P_\bullet, \textup{gr}_\bullet \Vv_C \otimes \omega_C \bigr).
 \]
 Now by Lemma \ref{lem:small},  
 \begin{equation}\label{eq:gCPH0oplus}
 \text{H}^0\left(C\setminus P_\bullet, \textup{gr}_\bullet \Vv_C \otimes \omega_C \right)
\cong \text{H}^0\left(C\setminus P_\bullet, \Vv_C \otimes \omega_C \right).
\end{equation}
On the other hand, as $C\setminus P_\bullet$ is assumed to be affine, one has
\[
\Ll_{C\setminus P_\bullet}(V) \cong  \text{H}^0\left(C\setminus P_\bullet, \Vv_C\otimes \omega_C \right) / \nabla \text{H}^0\left(C\setminus P_\bullet, \Vv_C \right) .
\]
The map $\varphi_{\Ll}$  is induced from the composition
\begin{equation}
\label{eq:H0oplusH0End}
\text{H}^0\left(C\setminus P_\bullet, \Vv_C\otimes \omega_C \right) \rightarrow \oplus_{i=1}^n \text{H}^0\left(D^\times_{P_i}, \Vv_C\otimes \omega_C \right)
\rightarrow \oplus_{i=1}^n \mathfrak{L}_{t_i}(V).
\end{equation}
The first map is canonical and obtained by restricting sections; the second is \eqref{eq:simeqt}.
By \cite[\S 6.6.9]{bzf}, sections in $\nabla \text{H}^0\left(D^\times_{P_i}, \Vv_C \right)$ act trivially. Hence \eqref{eq:H0oplusH0End} induces a map from the Lie algebra 
$\Ll_{C\setminus P_\bullet}(V)$ to $\oplus_{i=1}^n \mathfrak{L}_{t_i}(V)$.
It follows that the image of $\varphi_{\Ll}$ coincides with the image  of
$\text{H}^0\left(C\setminus P_\bullet, \Vv_C\otimes \omega_C \right)$ in $\oplus_{i=1}^n \mathfrak{L}_{t_i}(V)$ via~\eqref{eq:H0oplusH0End}.
Composing \eqref{eq:gCPH0oplus1} and \eqref{eq:gCPH0oplus}, and by the definition of $\varphi_{\mathfrak{g}}$ in \eqref{eq:phig}, the image of the map in \eqref{eq:H0oplusH0End} coincides with the image of~$\varphi_{\mathfrak{g}}$.
\end{proof}

\section*{Acknowledgements}
We thank Dennis Gaitsgory for  questions which led to the need to replace $\mathscr{V}_C$ with $\Vv_C$.  
We thank Igor Frenkel, Bin Gui, Yi-Zhi Huang, Danny Krashen, Jim Lepowsky, and Sven M\"oller for helpful discussions. Thanks also to Yi-Zhi Huang and Giulio Codogni for comments on a preliminary version of the manuscript, and to Andr\'{e} Henriques, on a later version. We are indebted to \cite{bzf} for their work on smooth curves, and to \cite{nt} for the arguments on the factorization and sewing for coinvariants by Zhu's Lie algebra on rational curves.  Finally, we wish to thank the referees for carefully reading our work, and one generous referee in particular, who has provided many valuable comments.  
 Gibney was supported by NSF DMS--1902237.



\bibliographystyle{alphanumN}
\bibliography{Biblio}

\end{document}